\documentclass[11pt]{article}

\usepackage{geometry,showlabels, fullpage, authblk}
\usepackage{xr-hyper}
\usepackage{times,enumitem}
\usepackage{hyperref}[]
\hypersetup{
	colorlinks=true,
	linkcolor=blue,
	filecolor=magenta,      
	urlcolor=cyan,
	citecolor=blue,
}
\usepackage{ulem}

\usepackage{url}

\usepackage{bigints}
\usepackage{amsfonts,amsmath,amssymb,amsthm, bbm}
\usepackage{wrapfig}
\usepackage{verbatim,float,url}
\usepackage{graphicx}
\usepackage{subcaption}
\usepackage[round]{natbib}   % omit 'round' option if you prefer square brackets
\usepackage[english]{babel}
\usepackage{cancel}
\usepackage{color}
\usepackage[dvipsnames]{xcolor}
\usepackage{cleveref}

\usepackage{mathtools}

\newtheorem{theorem}{Theorem}
\newtheorem{lemma}[theorem]{Lemma}
\newtheorem{corollary}[theorem]{Corollary}
\usepackage{mdwlist}	

\newtheorem{remark}{Remark}
\theoremstyle{definition}
\newtheorem{definition}{Definition}

\newtheorem{proposition}{Proposition}

\usepackage{tikz}  
\usetikzlibrary{trees}
\usetikzlibrary{arrows.meta}
\usetikzlibrary{decorations.pathreplacing}

\graphicspath{{figs/}}

\usepackage{algorithm}% http://ctan.org/pkg/algorithms
\usepackage{algpseudocode}% http://ctan.org/pkg/algorithmicx

\algnewcommand\algorithmicinput{\textbf{INPUT:}}
\algnewcommand\INPUT{\item[\algorithmicinput]}
\algnewcommand\algorithmicoutput{\textbf{OUTPUT:}}
\algnewcommand\OUTPUT{\item[\algorithmicoutput]}

\usepackage{mathtools}

\DeclarePairedDelimiter\floor{\lfloor}{\rfloor}

\theoremstyle{plain}

\allowdisplaybreaks
\usepackage{bm}

\title{Risk Bounds for  Quantile  Trend Filtering}

\author[1]{Oscar Hernan Madrid Padilla}
\author[2]{Sabyasachi Chatterjee}

\affil[1]{\small Department of Statistics, University of California, Los Angeles}
\affil[2]{\small Department of Statistics, University of Illinois at Urbana-Champaign}

\begin{document}
	\maketitle
	
	\begin{abstract}
	We study quantile trend filtering, a recently proposed method for nonparametric quantile regression with the goal of generalizing  existing risk bounds known for the usual trend filtering estimators which perform mean regression. We study both the penalized and the constrained version (of order $r \geq 1$) of univariate quantile trend filtering. Our results show that both the constrained and the penalized version (of order $r \geq 1$) attain the minimax rate up to log factors, when the $(r-1)$th discrete derivative of the true vector of quantiles belongs to the class of bounded variation signals. Moreover we also show that if the true vector of quantiles is a discrete spline with a few polynomial pieces then both versions attain a near parametric rate of convergence. Corresponding results for the usual trend filtering estimators are known to hold only when the errors are sub-Gaussian. In contrast, our risk bounds are shown to hold under minimal assumptions on the error variables. In particular, no moment assumptions are needed and our results hold under heavy-tailed errors.  %On the other hand, we prove all our results for a Huber type loss which can be smaller than the mean squared error loss employed for showing risk bounds for usual trend filtering. 
Our proof techniques are general and thus can potentially be used to study other nonparametric quantile regression methods. To illustrate this generality we also employ our proof techniques to obtain new results for multivariate quantile total variation denoising and high dimensional quantile  linear regression. 
		%multipliers. Finally,  our experiments in real data  show the value of our proposed method for interpretability of the latent factors,  and for variable selection.
		\vskip 5mm
		\textbf{Keywords}: 	Total variation, nonparametric quantile regession, local adaptivity, fused lasso.
	\end{abstract}
	
	%\section{Introduction}

	\section{Introduction}
	
	\subsection{Introduction}
	
	In this paper we focus on the problem of nonparametric quantile regression for the \textit{quantile sequence model}. Specifically, let $y \in R^n$ be a vector of independent random variables and for a given quantile level $\tau \in (0,1)$,  let $\theta^*$, a vector  of $\tau$-quantiles of  $y$, be given by
	\[
	\theta_i^*   =  \arg \min_{a \in  R} E\:\{\rho_{\tau}(y_i - a) \}
	\]
	where $\rho_{\tau}(x) = \max\{  \tau x,(\tau-1)x \}$ is the usual check function used for quantile regression. Upon observing $y$, the problem is to estimate the vector of quantiles $\theta^*$. We call this the quantile sequence model. This generalizes the usual Gaussian sequence model where the quantile $\tau$ is taken to be $0.5$ and the distribution of $y$ is taken to be multivariate normal with the covariance matrix a multiple of identity. %For the moment, we do not impose any restrictions on the distribution of the components of $y.$ 
	
	Our main focus in this paper is on signals (quantile sequences) that have bounded $r$th order total variation. For a vector $\theta \in  R^n,$ let us define $D^{(0)}(\theta) = \theta, D^{(1)}(\theta) = (\theta_2 - \theta_1,\dots,\theta_n - \theta_{n - 1})^{\top}$ and $D^{(r)}(\theta)$, for $r \geq 2$, is recursively defined as $D^{(r)}(\theta) = D^{(1)}(D^{(r - 1)}(\theta)).$ Note that $D^{(r)}(\theta) \in  R^{n - r}.$ For simplicity, we denote the operator $D^{(1)}$ by $D.$ For any positive integer $r \geq 1$, let us now define the $r$th order total variation of a vector $\theta$ as follows:
	\begin{equation}
		\mathrm{TV}^{(r)}(\theta) = n^{r - 1} \|D^{(r)}(\theta)\|_{1}
	\end{equation}
	where $\|.\|_1$ denotes the usual $\ell_1$ norm of a vector. 
	\begin{remark}
		The $n^{r - 1}$ term in the above definition is a normalizing factor and is written following the convention adopted in the trend filtering literature; see for instance~\cite{guntuboyina2020adaptive}. If we think of $\theta$ as evaluations of a $r$ times differentiable function $f:[0,1] \rightarrow  R$ on the grid $(1/n,2/n\dots,n/n)$ then the Riemann approximation to the integral $\int_{[0,1]} \vert f^{(r)}(t)\vert  dt$ is precisely equal to $\mathrm{TV}^{(r)}(\theta).$ Here $f^{(r)}$ denotes the $r$th derivative of $f.$ Thus, for natural instances of $\theta$, the reader can imagine that $\mathrm{TV}^{(r)}(\theta) = O(1).$ 
	\end{remark}

	Let us now define the constrained quantile trend filtering (CQTF) estimator which is one of the main objects of study in this paper, and it is given as
	
	\begin{equation}
		\label{eqn:quantile_trend_filtering2}
		\hat{\theta}^{(r)}_V   \,=\,   \begin{array}{ll}
			\underset{  \theta \in R^n:  \mathrm{TV}^{(r)}(\theta) \leq \ V}{\arg \min  } &     \displaystyle  \sum_{i=1}^{n} \rho_{\tau}(y_i -   \theta_i).\\
		\end{array} 
	\end{equation}
	Here $V$ is a tuning parameter.

	The other estimator we focus on in this paper is the penalized quantile trend filtering estimator (PQTF) defined as follows:
	\begin{equation}
		\label{eqn:quantile_trend_filtering}
		\hat{\theta}^{(r)}_{\lambda}\,=\,  \underset{\theta  \in R^n}{\arg \min} \, \left\{  \sum_{i=1}^{n} \rho_{\tau}(y_i -   \theta_i)   \,+\,   \lambda\:\mathrm{TV}^{(r)}(\theta)  \right\},
	\end{equation}%\|  D^{(r)} \theta \|_1 
	%\]
	for a tuning parameter  $\lambda>0$. This is the quantile regression version of the standard trend filtering estimator proposed first by~\cite{kim2009ell_1}. %and and its statistical properties analyzed in~\cite{tibshirani2014adaptive}. %We will call the above estimator the penalized quantile trend filtering (PQTF) estimator. 

	The PQTF estimator has already been proposed in the literature. The PQTF estimator with $r=1$ appeared in \cite{li2007analysis}. When $r=1$, we refer to the PQTF  estimator as quantile fused lasso. More recently, \cite{brantley2019baseline}  proposed the general quantile trend filtering estimator (PQTF) of order $r \geq 1$. However, to the best of our knowledge, not much is known about the theoretical properties (such as risk bounds) of the PQTF and the CQTF estimators.
	
	Due to $\ell_1$ penalization both  the  CQTF and PQTF   estimators  enforce  $D^{(r)}(\hat{\theta}^{(r)}_{V})$ and $D^{(r)}(\hat{\theta}^{(r)}_{\lambda})$  respectively to be sparse. It is known that 
	for $\theta  \in R^ n$,  $  D^{(r)}( \theta)$  has  $k$ nonzero entries  if and only if  $\theta =  (f(1/n),f(2/n),\ldots,f(n/n) )^{\top}$   for a \textit{discrete spline} function $f$, consisting of  $(k+1)$ polynomials of degree $r-1$ (see Proposition $D.3$ in~\cite{guntuboyina2020adaptive}). For this reason, just like the usual trend filtering estimators, the CQTF and the PQTF estimators both fit discrete splines. For the precise definition of a discrete spline see Section 2 in~\cite{mangasarian1971discrete}.

	\subsection{Notation}
	\label{sec:notation}
	
	Let  $\{a_n\}$ and $\{b_n\}$ $\subset  \,\,R$ be two positive sequences. We write    $a_n = O(b_n)$  if there exists   constants $C>0$  and  $n_0 >0$ such that  $n \geq  n_0$ implies   that  $a_n \,\leq \,  C b_n$.   We also use the notation 
	$a_n = \tilde{O}(b_n)$ to indicate   that      $a_n    \,\leq\,  C  b_n g(\log n) $ for  $n\geq n_0$ where  $g(\cdot)$ is a polynomial function.
	Furthermore, if  $a_n =  O(b_n)$ and  $b_n =  O(a_n)$ then we write  $a_n =  \Theta(b_n)$ or  $a_n \asymp b_n$.
	For a sequence of random variables $X_n$ and a positive sequence $a_n$ we write  $X_n = O_{\mathrm{pr}}(a_n)$ if  for every $\epsilon>0$ there exists  $M > 0$ such that    $\mathrm{pr}\left(|X_n| \,\geq\,  M a_n\right) < \epsilon$ for all $n$. 
	%Throughout this paper, for every $n \in  N$ we consider  input data  $y \in R^n$ and then calculate a performance  metric  $c_n$  for a given estimator.  Importantly, there need no be a  connection between data for the models with different $n$'s.  We write    $c_n = O_{\mathrm{pr}}(d_n)$  for  a sequence  $\{d_n\} \subset R$  if  for every $\epsilon>0$ there exists  $C>0$ such that    $\mathrm{pr}\left(   c_n \,\geq\,  C d_n      \right) < \epsilon$ for all $n$.
	%Throughout we use the
	For any positive integer $n$, we denote the set of positive integers from $1$ to $n$ by $[n].$ For any vector  $v \in  R^m$ we denote its usual Euclidean or $\ell_2$ norm by $\|v\|.$ Furthermore, for  a  vector  $v \in  R^m$,  we define $\| v \|_0   \,=\, \vert \{   j   \in \{1,\ldots,m\}  \,:\,  v_j \neq 0  \}\vert $. Finally, for a set $K \subset  R^n$ we define its Rademacher width (or complexity)  as
	\[
	RW(K) =        E\left(    \underset{v \in K}{\sup}     \sum_{i=1}^{n}  \xi_i  v_i    \right),
	\]
	where  $\xi_1,\ldots, \xi_n$  are independent  Rademacher  random variables. %Now we state our general result. 
	%Furthermore, we hightlight that  Assumption \ref{as2} will hold for most common distributions  including the Cauchy distribution.

	\subsection{Summary of Our Results}

	The usual (mean regression) univariate trend filtering estimators are defined similarly to the PQTF and the CQTF estimators with the $\rho_{\tau}$ function replaced by the $x \rightarrow x^2$ function. These estimators were independently  introduced by \cite{steidl2006splines} and  \cite{kim2009ell_1}.

	A continuous version of these trend filtering estimators, where discrete derivatives are replaced by continuous derivatives, was proposed much earlier in the statistics literature 
	by~\cite{mammen1997locally} under the name {\em locally adaptive regression splines}. 
	By now, there exists a body of literature studying the risk properties of trend filtering under squared error loss. There exists two strands of risk bounds for trend filtering in the literature focussing on two different aspects.

	Firstly, for a given constant $V > 0$ and $r \geq 1$, $\Theta\{n^{-2r/(2r + 1)}\}$ rate is known to be the minimax rate of estimation over the space $\mathcal{BV}^{(r)}_{n}(V)$; (see e.g,~\cite{donoho1994ideal}) where for any  integer $r \geq 1$,
	\begin{equation*}
		\mathcal{BV}^{(r)}_{n}(V) = \{\theta \in  R^n: \mathrm{TV}^{(r)}(\theta) \leq V\}.
	\end{equation*}
	%{\color{red} Are we using this terminology anywhere in the paper?}

	A standard terminology in this field terms this $\Theta(n^{-2r/(2r + 1)})$ rate as the \textit{slow rate}. It is also known that a well tuned trend filtering estimator is minimax rate optimal over the parameter space $\mathcal{BV}^{(r)}_{n}(V)$ and thus attains the slow rate. This result has been shown in~\cite{tibshirani2014adaptive} and~\cite{wang2014falling} building on earlier results by~\cite{mammen1997locally}.

	Secondly, it is also known that an ideally tuned trend filtering (of order $r$) estimator can adapt to $\|D^{r}(\theta)\|_0$, the number of non zero elements in the $r$th order differences, \textit{under some assumptions on $\theta^*$}. Such a result has been shown in~\cite{guntuboyina2020adaptive} and~\cite{ortelli2019prediction}. In this case, the Trend Filtering estimator of order $r$ attains the $\tilde{O}\{\|D^{(r)}(\theta)\|_0/n\}$ rate which can be much faster than the $n^{-2r/(2r + 1)}$ rate. Standard terminology in this field terms this as the \textit{fast rate}. %We refer the reader to Section \ref{sec:notation}  where the notations $\tilde{O}(\cdot)$,$\Theta(\cdot)$ and the $\|\cdot\|_0$ are introduced.

	Our goal in this paper is to extend these two types of results for quantile trend filtering estimators under minimal assumptions on the distribution of the components of the data vector $y.$ We are able to do this to a large extent with two main differences from the existing results. To the best of our knowledge, the results for usual trend filtering all hold under sub-Gaussian noise and under mean squared error loss. Our results for quantile trend filtering estimators hold under an extremely mild assumption on the growth of the CDF's of the components of $y$ around the true quantiles; see Section \ref{sec:main_results} where this assumption is stated. In particular, our results hold even when the distribution of $y$ is heavy-tailed (with no moments existing) such as the Cauchy distribution. In this sense, our results are stronger than the existing results for the usual trend filtering estimators. On the other hand, our results hold under a Huber type loss which is in general smaller than the mean squared error loss. %Our proof technique necessarily requires us to use this Huber type loss precisely because we wanted to put minimal assumptions on the distribution of $y.$ %{\color{red} Discuss this last line}
	%{\color{red} Huber loss vs l2 loss. Mention in the discussion?}

	%Our results for quantile trend filtering estimators hold for a smaller Huber type loss function but also hold under an extremely mild assumption on the growth of the CDF's of the components of $y$ around the true quantiles; see~\eqref{blah} where this assumption is stated. In particular, our results hold even when the distribution of $y$ is heavy tailed (with no moments existing) such as the Cauchy distribution. 

	Our loss function is given by the function $\Delta_n^2  \,:\,R^n \rightarrow  R$ defined as 
	\begin{equation}
		\label{eqn:loss}
		\Delta_n^2(v) \,=\,   \frac{1}{n}\sum_{i=1}^{n}   \min\left\{  \vert v_i \vert,  v_i^2 \right\}
	\end{equation}
	which, up to constants, is a Huber loss, see \cite{huber1992robust}. We also write   $\Delta^2(v) =  n \Delta_n^2(v)$.  The main reason why our bounds are for the Huber loss is that this loss naturally appears as a lower bound to the quantile population loss; see~\eqref{eq:losslb} and Section \ref{sec:ideas} for a more detailed explanation. %Moreover, intuitively, it penalizes significantly less large estimation errors as compared to the usual mean squared error, making it more reasonable for estimation when the errors have heavy tails. 

	In our first result in Theorem \ref{thm5}, we show that the CQTF estimator satisfies 
	\begin{equation*}
		\label{eqn:rate1}
		\Delta_n^2(\hat{\theta}_{V}^{(r)} - \theta^*)   \, = \,  O_{\mathrm{pr}} \left\{n^{  -2r/(2r+1)}       \right\},  
	\end{equation*}
	where the notation  $O_{\mathrm{pr}}(\cdot)$  is defined in Section \ref{sec:notation}. Therefore, the CQTF estimator attains the minimax rate for estimating signals in $\mathcal{BV}^{(r)}_{n}(V).$ See  Section \ref{sec:discussion} where we state precisely in what sense  this is minimax rate optimal. Additionally, in Theorem \ref{thm4} we show that a similar result is satisfied by the PQTF estimator. These results generalize the \textit{slow rate} results for trend filtering to the quantile setting.  %(\ref{eqn:quantile_trend_filtering2}) attains the rate $ n^{  -2r/(2r+1)}  $   under the  error metric  $\Delta_n^2(\cdot)$.

	Now let us consider the case when  $\| D^{(r)} \theta^*\|_0 =s$  and the elements of $\{ j\,:\,  (D^{(r)} \theta^*)_j \neq 0   \}$ satisfy a minimal spacing condition. In Theorem \ref{thm2} we prove that,  with an ideal tuning parameter $V = \mathrm{TV}^{(r)}(\theta^*) = V^*$,  the CQTF estimator satisfies
	\begin{equation}
		\label{eqn:r1}
		\Delta_n^2(\hat{\theta}_{V^*}^{(r)} - \theta^*)  =O_{\mathrm{pr} }\left\{ \frac{ (s+1)}{n}  \log\left(\frac{en}{s+1}  \right) \right\}.
	\end{equation}
	Our result generalizes the \textit{fast rate} results of  \cite{guntuboyina2020adaptive}  to the quantile setting. We also show in Theorem~\ref{thm6} that the PQTF estimator of orders $r \in \{1,2,3,4\}$, when the tuning parameter $\lambda$ is chosen appropriately, attains the above \textit{fast rate}. This result generalizes the \textit{fast rate} results of  \cite{ortelli2019prediction}  to the quantile setting.
	%More recently, \cite{ortelli2019prediction}  showed a similar bound, with extra log factors, that holds  for the penalized trend filtering estimator for  $r \in \{1,2,3,4\}$. It is an interesting question as to whether the results of \cite{ortelli2019prediction} can be generalized for the PQTF estimators. We do not address this question in this paper. {\color{red}Mention this vdg ortelli point in the discussion section itself? There is a repetition.}

	In this paper we actually formulate a general quantile sequence problem under convex constraints. The setup is that we have a vector $y \in  R^n$ of independent random variables and $\theta^* \in R^n$ is a corresponding vector of $\tau \in (0,1)$ quantiles of $y.$ Suppose it is known that $\theta^* \in K \subset  R^n$ where $K$ enforces a constraint on  
	the vector $\theta^*.$ This is a generalization of the Gaussian sequence model with constraints on the mean vector. We call this the \textit{constrained quantile sequence problem}. A natural estimator for this problem is the following:
	
	\begin{equation}\label{eq:cqse}
		\hat{\theta}_{K}  \,=\,   \begin{array}{ll}
			\underset{  \theta \in K}{\arg \min  } &     \displaystyle  \sum_{i=1}^{n} \rho_{\tau}(y_i -   \theta_i) \\
		\end{array} 
	\end{equation}
	
	We call this estimator the \textit{constrained quantile sequence estimator} or the CQSE estimator. For example, if $K = \{\theta \in  R^n: \mathrm{TV}^{(r)}(\theta) \leq V\}$ for some integer $r \geq 1$ and some $V > 0$ then the above estimator is the CQTF estimator of order $r$ with tuning parameter $V$ as defined in~\eqref{eqn:quantile_trend_filtering2}.

	We prove a general result bounding the risk of $\hat{\theta}_{K}$ in terms of the Rademacher width of $K$; see Theorem~\ref{thm:basic} in Section~\ref{sec:ideas}.
	Therefore our proof  technique can potentially be used for other CQSE estimators with different constraint sets $K.$ In this context, we also consider two other related quantile estimation problems with different constraint sets $K$ and prove results for the corresponding CQSE estimators that appear to be new. The first problem we consider is two dimensional quantile total variation denoising which is the quantile version of the ubiquitous total variation denoising estimator (see~\cite{rudin1992nonlinear}) used in image processing. Here, in Theorem~\ref{thm:2dtv} we generalize existing results of \cite{hutter2016optimal,chatterjee2019new} to the quantile setting. To the best of our knowledge, quantile total variation denoising has not been proposed and studied before in the literature. Another setting we consider is high-dimensional quantile regression. We study the quantile version of lasso and in Theorem~\ref{thm:lasso} we prove a \textit{slow rate} for quantile lasso under the fixed design setup which holds without \textit{any} assumptions on the design matrix. Previous results in this problem show a fast  rate but with restricted eigenvalue conditions imposed on the design matrix as in \cite{belloni2011,fan2014adaptive}. %\textcolor{red}{i think they have a fast rate result}

	\subsection{Some Related Literature}

	In this section we mention some other existing works in the literature which are closely related to our work. Since its introduction  by \cite{koenker1978regression}, quantile regression has become a prominent tool in statistics. The attractiveness of quantile regression is  due to its  flexibility for modelling conditional distributions, construction of predictive models, and even outlier detection applications. The problem  of one-dimensional  nonparametric  quantile  regression  goes back at least to  \cite{utreras1981computing,cox1983asymptotics,eubank1988spline}  who focused on median regression. \cite{koenker1994quantile} introduced quantile smoothing splines in one dimension. These  are defined as the solution to problems of the form
	\[
	\displaystyle  \underset{g \in \mathcal{C}   }{\mathrm{minimize}}\, \left[ \sum_{i=1}^{n}  \rho_{\tau}\{ y_i -  g(x_i)\}   \,+\,  \lambda \left\{ \int_{0}^1 \vert   g^{\prime \prime}(x) \vert^p dx \right\}^{1/p}  \right],
	\]
	assuming that  $0 < x_i < \ldots < x_n < 1$,  where  $\lambda>0$ is  a tuning parameter,  $p \geq  1$, and   $\mathcal{C}$  a suitable  class of functions. When  $p=1$ this is related  to the quantile version of locally adaptive regression splines of order $2$ which appeared later in~\cite{mammen1997locally}. %{\color{red} Is this correct?}
	The theoretical properties of  quantile smoothing splines were studied in  \cite{he1994convergence}. Specifically, the authors  in  \cite{he1994convergence}  demonstrated that quantile smoothing splines attain the rate  $n^{-2r/(2r+1)}$,  for estimating   quantile  functions  in the class of  H\"{o}lder  functions  of exponent  $r$. %To the best of our knowledge, \cite{he1994convergence} seems to be the closest  theoretical quantile regression work to ours. 

	It is natural to believe that the connections between quantile version of adaptive regression splines proposed in~\cite{mammen1997locally} and quantile trend filtering would be similar to the connections between the mean regression counterparts. It is known that both attain similar rates over appropriate bounded variation function classes but trend filtering is computationally more efficient; see~\cite{tibshirani2014adaptive}.
	
	%{\color{red} Cite Ryan's papers here? The latest paper talks about something else? Discrete Splines. Where should we cite that?}
	%We refer the reader  to \cite{tibshirani2020divided} for a thorough discussion highlighting that  trend filtering  is  in fact a special  case of discrete splines.
	%\cite{tibshirani2020divided}

	In  the context  of median regression in one dimension,  the authors in \cite{brown2008robust} showed that  a wavelet-based quantile regression approach  attains  minimax rates  for  estimating the median function,  when the latter  belongs  to Besov spaces which is related to the sort of bounded variation classes considered in this paper. However, our focus in this paper is not on wavelet methods. Despite  the optimality of wavelet methods,  it is also known  that   total variation   based methods  can outperform wavelet methods in practice, see  \cite{tibshirani2014adaptive,wang2016trend}.  %Thus, we focus on trend filtering based estimators as in (\ref{eqn:quantile_trend_filtering}). 

	A precursor of trend filtering can be traced back in the machine learning literature to \cite{rudin1992nonlinear}  who proposed a two dimensional total variation penalized method for image denoising applications. To the best of our knowledge, the quantile version of this estimator has not been considered before in the literature. Due to the ubiquity of this image denoising method, we study the quantile 2D total variation denoising estimator in this paper; see Section~\ref{sec:2d}. %In its current univariate version, trend filtering  was independently  introduced by \cite{steidl2006splines} and  \cite{kim2009ell_1}. 
	
	%\cite{tibshirani2005sparsity}  then  defined difference based estimators of the form (\ref{eqn:trend_filtering}) that have led to further development of methods for trend filtering  such as in \cite{kim2009ell_1,tibshirani2011solution,tibshirani2012degrees}.

	On the computational front, it is known that the usual trend filtering estimator with $r=1$  can  be solved in $O(n)$ time, see for instance \cite{johnson2013dynamic}.  More recently, \cite{hochbaum2017faster} showed that  the  corresponding quantile fused lasso  estimator (PQTF) with  $r=1$, can  be computed in  $O(n \log n)$ time. For  $r>1$, \cite{brantley2019baseline}  proposed  an alternating direction method of multipliers (ADMM)  based algorithm for computing PQTF estimators of order $r.$

	%Hence $\hat{\theta}^{(r)}$  as defined in (\ref{eqn:quantile_trend_filtering}) is the  quantile version of  (\ref{eqn:trend_filtering}).   The case of  trend filtering with $r=1$ was introduced in \cite{rudin1992nonlinear}  for image denoising applications.

	% \cite{fan2018approximate}   studied an $\ell_0$  estimator  inspired by total variation regularization. 
	%{\color{red} If we talk about ell0 then we need to cite other works apart from fan}. %\cite{madrid2020adaptive}   proved that the fused lasso in geometric graphs attains minimax results for piecewise Lipchitz classes. \cite{ortelli2019synthesis}  studied connections between fused lasso on graphs and   the lasso estimator from \cite{tibshirani1996regression}.

	%In the weak sparsity  setting, with a fixed design, we show that $\ell_1$-constrained quantile regression can consistently estimate  the vector of regression coefficients, but without requiring a restricted eigenvalue condition as in \cite{belloni2011,fan2014adaptive}.

	\section{Main  Results}
	\label{sec:main_results}
	
	\subsection{Assumption}
	
	For all our theorems, unless stated otherwise, we consider any fixed quantile level $\tau \in (0,1)$, and any fixed integer $r \in \{1,2,\ldots\}$. The quantities  $\epsilon_i =  y_i -\theta_i^*$,  $i=1,\ldots,n$ which are unobservable are referred to as the errors. We also generically write  $V^* =   \mathrm{TV}^{(r)}  \left( \theta^*\right)$ where $\theta^*$ is the true signal. %Clearly,  $\theta^* \in K$, where
	%\begin{equation}
	%\label{eqn:bv_class}
	%K = \left\{ \theta \in R^n \,:\,   \mathrm{TV}^{(r)}  \left( \theta\right) \leq  V^* \right\}.
	%\end{equation}
	%Notice that  when $r=1$  the set  $K$ becomes the class of  bounded variation signals. The case of  $r>1$ corresponds to higher order  bounded variation classes, see \cite{tibshirani2014adaptive} for an overview.

	%Our first  assumption  stated  next  simply requires that $\theta^*$, the vector  of $\tau$-quantiles, has  $k$th discrete derivative which has bounded variation. We also require that the measurements  $y_i$ are independent.
	
	%\begin{assumption}
	%	\label{as1}
	%	We write $  \theta_i^*  = F_{y_i}^{-1}(\tau) $ for $i=1,\ldots,n$, and $V^* :=     n^k  \|D^{(r)}  \theta^*\|_1 $  satisfies  $V^* = O(1)$.  Here  $F_{y_i}$  is cumulative distribution function of $y_i$ for  $i=1,\ldots,n$. Also,   $y_1,\ldots,y_n$ are assumed to be  independent. 
	%\end{assumption}
	
	%Notice that  when $k=0$  the set  $K$ becomes the class of  bounded variation signals. Choices  of  $k>0$ corresponds to higher order  bounded variation classes, see \cite{tibshirani2014adaptive} for an overview.

	We state all of our results under the following assumption on the distribution of the components of $y$. %an assumption that basically guarantees that the quantiles of the components of $y$ are uniquely defined. %stated next requires that  for  each  $y_i$,   there exists  a   neighborhood around  the quantile such that  within such neighborhood  the cumlative distribution function of $y_i$ grows linearly away from $\theta^*_i$.
	%the  probability density function of $y_i$ is bounded by below.  A related assumption  appeared as D.1  in \cite{belloni2011}, and Condition 2 in \cite{he1994convergence}.
	%We state our assumption precisely below.
	
	\textbf{\noindent{Assumption A:}} There exist constants $L>0$ and $\underline{f}>0$ such that for any positive integer $n$ and any $\delta \in R^n$  satisfying  $\|\delta\|_{\infty} \leq L$  we have for all  $i= 1,\ldots, n$,
	\[
	\,   \vert   F_{y_i}(\theta_i^* + \delta_i)  -F_{y_i}(\theta_i^*) \vert\,  \geq \,  \underline{f}\,  \vert \delta_i \vert,
	\]
	where  $F_{y_i}$ is the CDF of  $y_i$.
	%$f_{y_i}$  is the probability density function  of $y_i$. %We write $  \theta^*  = F_{y_i}^{-1}(\tau) $, and   assume that  $\theta^* \in K$. 
	%\end{assumption}

	%Equivalently, for any sequence $\delta$ taking values in $[-L,L]$, the sequence of CDF's $F_{y_i}$ satisfies 
	%$$ \liminf_{n = 1}^{\infty}  \vert   F_{y_i}(\theta_i^* + \delta_i)  -F_{y_i}(\theta_i^*) \vert\,  > 0$$ 
	
	If the  cumulative distribution functions $F_{y_i}$ have probability  density functions $f_{y_i}$ with respect to Lebsgue measure then \textbf{\noindent{Assumption A}} is a weaker assumption than requiring that for any positive integer $n$,
	\[ 
	\underset{  \|\delta\|_{\infty} \leq L     }{\inf}\,\,\underset{i = 1,\ldots,n}{\min}\,\,f_{y_i}(  \theta_i^*  +\delta_i ) \geq \underline{f}, \] 
	which appeared as  Condition 2 in \cite{he1994convergence}, and is related to condition D.1  in \cite{belloni2011}. Such an assumption ensures that the quantile of $y_i$ is uniquely defined and there is a uniformly linear growth of the CDF around a neighbourhood of the quantile. An assumption of such a flavor (making the quantile uniquely defined) is clearly going to be necessary. We think this is a mild assumption on the distribution of $y_i$ as this should hold for most realistic sequences of distributions. For example, if the $y_i$'s are independent draws from   any density with respect to the Lebesgue measure that is bounded away from zero on any compact interval then our assumption will hold. In particular, no moment assumptions are being made on the distribution of the components of $y.$

	%{\color{red} Explain the asymptotic framework of bigo p. Triangular array.}

	\subsection{Results for CQTF Estimator}
	\label{sec:constrained}

	We now state our first result which is the \textit{slow rate} result for the CQTF estimator. %This shows that quantile trend filtering attains optimal rates for estimating signals in $K$. The proof of this result is deferred to the Supplementary material.
	
	\begin{theorem}
		\label{thm5} 
		Let $\{y_i\}_{i = 1}^{n}$ be any sequence of independent random variables which satisfies \textbf{Assumption A} and $\theta^*_i$ be the sequence of $\tau$ quantiles of $y_i.$ If $V$ is chosen such that  $V \geq  V^* = \mathrm{TV}^{(r)}(\theta^*)$ then 
		\[ \Delta_n^2(\hat{\theta}^{(r)}_V - \theta^*)    =O_{\mathrm{pr} }\left[   n^{   -2r/( 2r+1)   }  V^{2/(2r+1)} \max\left\{   1,        \left(  \frac{V }{n^{r-1}}\right)^{(2r-1)/(2r+1)}  \right\}  \right]. \]
		%	provided that
		%	\[
		%	 \max\{   V^*,(V^*)^{   2/(2r+1) } \}  =   O\left\{  n^{   (2r)/( 2r+1)   }\right\}.
		%	\]
	\end{theorem}
	
	%Notably, under the canonical scaling  $V^*= O(1)$, Theorem \ref{thm5}  shows that the  constrained quantile trend filtering estimator attains minimax rates for estimating  $\theta^*$ in the class of parameters $K$,  see \cite{mammen1997locally,tibshirani2014adaptive,guntuboyina2017spatial}. However, unlike previous results on trend filtering, our result holds without the strong assumption that the errors are sub-Gaussian. This explains  why the upper  depends on the loss $\Delta_n^2(\cdot)$ defined in (\ref{eqn:loss}). %Moreover, Theorem \ref{thm5}  shows that quantile trend filtering  is a robust estimator.

	\begin{remark}
		The above theorem holds for any $\tau \in (0,1)$. The role of $\tau$ is not made explicit on the right hand side in Theorem \ref{thm5}. The proof of Theorem \ref{thm5} reveals that the closer $\tau$ is to $\{0,1\}$, the larger the constants are in the upper bound in Theorem \ref{thm5}. %For symmetric  distributions, the closer $\tau$ is to $0.5$ the less  difficult it becomes to estimate  the vector of $\tau$-quantiles  $\theta^*$. {\color{red} Why symmetric distributions? Is the last line needed?}%However, although this is not reflected by the convergence rate since  $\tau $ remains fixed.
	\end{remark}

	\begin{remark}
		Theorem  \ref{thm5}  can be thought of as generalizing Theorem $2.1$ from \cite{guntuboyina2020adaptive}  to the quantile regression setting.  Aside from the different loss $\Delta_n(\cdot)$ and our result being a $O_{pr}$ statement, our result also differs from Theorem $2.1$ in~\cite{guntuboyina2020adaptive} in that our  upper bound has an extra term. This is the factor 
		\[
		\max\left\{   1,        \left(  \frac{V }{n^{r-1}}\right)^{(2r-1)/(2r+1)}  \right\}  
		\]
		which can  go to infinity if  $V$ grows faster than $n^{r-1}$. However, under the natural scaling $V^* = O(1)$ one can choose $V = O(1)$ as well and thus the above term is also $O(1).$ 
		%{\color{red} 
		%Mention extra factors as compared to existing results and Choice of V. This generalizes theorem blah from guntu}
	\end{remark}

	%Before  stating our \textit{fast rate} result for the CQTF estimator  we introduce some notation borrowed from \cite{guntuboyina2020adaptive}.
	
	We now state our \textit{fast rate} result for the CQTF estimator. 
	
	\begin{theorem}\label{thm:2}
		\label{thm2} 
		Let $\{y_i\}_{i = 1}^{n}$ be any sequence of independent random variables which satisfies \textbf{Assumption A} and $\theta^*_i$ be the sequence of $\tau$ quantiles of $y_i.$ Let $s =\|D^{(r) } \theta^*\|_0 $ and $S=\{  j \,:\,  (D^{(r) } \theta^*)_j  \neq 0   \}$.
		Let  $j_0 < j_1 < \ldots < j_{s+1}$   be such that $j_0 =1$,  $j_{s+1} = n-r$ and  $j_1,\ldots, j_s$  are the elements of $S$. With this notation define $\eta_{j_0} =  \eta_{j_{s+1}}=0$. Then  for  $j \in S$  define  $\eta_j$  to be $1$ if   $(D^{(r-1) } \theta^*)_j <    (D^{(r-1) } \theta^*)_{j+1}   $, otherwise set  $\eta_{j} =-1$. 	 Suppose  that $\theta^*$ satisfies the following minimum length assumption
		\begin{equation}\label{eq:minleng}
			\underset{ l \in [s] ,\,\,   \eta_{j_l} \neq  \eta_{j_{l+1}}   }{\min}   \,\,\, ( j_{l+1} -j_{l} )  \geq   \frac{c  n }{s+1}
		\end{equation}
		for some constant $c$ satisfying  $0\leq c\leq  1$.  Then we have that
		\[ \Delta_n^2(\hat{\theta}^{(r)}_{V^*} - \theta^*)  =O_{\mathrm{pr}}\left[     \max\left\{   \frac{V^*}{n^{r-1} }  , 1\right\}   \frac{ (s+1)}{n}  \log\left(\frac{en}{s+1}  \right)  \right]. \]
	\end{theorem}
	%{\color{red} This is incorrectly stated. eta is not 0.}
	
	\begin{remark}
		Theorem  \ref{thm2}  shows that  the constrained  quantile  trend filtering estimator attains, off by a logarithmic factor,  the  rate attained by an oracle estimator that knows the set $S $. Thus, Theorem \ref{thm2} can be thought of as generalizing Theorem 2.2 of \cite{guntuboyina2020adaptive} to the quantile setting. Our minimum length assumption is identical to the one assumed by \cite{guntuboyina2020adaptive}. In particular it requires that when two consecutive change points correspond to two  opposite changes in trend, then the two points should be sufficiently separated.	
	\end{remark}

	\begin{remark}
		Notice that  Theorem \ref{thm2}  provides an upper bound that depends on $V^*$. This  was not the case in Theorem 2.2 from  \cite{guntuboyina2020adaptive}  which gave an upper bound that is independent of  $V^*$.  Nevertheless, in the case $V^* = O(n^{r-1})$ (which covers the canonical regime) we do obtain the same rate from Theorem 2.2 in \cite{guntuboyina2020adaptive}. 
	\end{remark}

	%For the case of median regression with sub-Gaussian errors and with cannonical  scaling  $V^* = O(1)$, Theorem  \ref{thm2}  shows that  the constrained  quantile  trend filtering estimator attains, off by a logarithmic factor,  the  rate attained by an oracle estimator that knows the set  $S $, see \cite{guntuboyina2017spatial}. However, Theorem \ref{thm2}  holds for general distributions and quantiles going beyond  sub-Gaussian distributions.

	%\textcolor{red}{Finally, regarding the choice of  $V$, in practice  once can follow one of the approaches  discussed in Section 3 of \cite{brantley2019baseline}. There the authors put forward a Bayesian Information Criterion  (BIC) and an extended Bayesian Information Criteria.   }
	
	\subsection{Results for PQTF Estimator}
	
	From a computational point of view the penalized quantile trend filtering seems to present a more appealing method than its constrained counterpart. 
	The optimization problems corresponding to the CQTF  and the PQTF estimators are both linear programs that can be solved using any generic linear programming software. However, the PQTF optimization problem has special structures that enable more efficient computation. Existing works (e.g \cite{hochbaum2017faster,brantley2019baseline})  have studied different types of algorithms that can efficiently solve the penalized  quantile trend filtering problem. This is in contrast to the CQTF optimization problem that has not received similar attention from a computational perspective perhaps due to its inherent difficulty. This makes it important to also study the risk properties of the PQTF estimator. We now present our \textit{slow rate} result for the PQTF estimator.

	\begin{theorem}\label{thm4} 
		%\textcolor{red}{	 
		Let $\{y_i\}_{i = 1}^{n}$ be any sequence of independent random variables which satisfies \textbf{Assumption A} and $\theta^*_i$ be the sequence of $\tau$ quantiles of $y_i.$
		Suppose  that $V^* =  \Theta(1)$. Given any $\epsilon\in (0,1)$  there exists a positive constant  $c_{1,\epsilon}$ only depending on $\epsilon$ and $V^*$
		%Then  there exists a constant  $C$  such that 
		such that if $\lambda$ is chosen to be %choice  of $\lambda$  satisfying
		\[
		\lambda = c n^{  1/(2r+1)    } \left( \log n \right)^{ 1/(2r+1) } ,  
		%\Theta\left\{ n^{  1/(2r+1)    } \left( \log n \right)^{ 1/(2r+1) }    \right\},
		%\Theta\left\{ n^{  (2r-1)/(2r+1)    } \left( \log n \right)^{ 1/(2r+1) }   \|D^{(r)} \theta^*\|_1^{ - (2r-1)/(2r+1) }   \right\},
		%  \Theta\left\{ n^{  (2r-1)/(2r+1)    } \left( \log n \right)^{ 1/(2r+1) }   \|D^{(r)} \theta^*\|_1^{ - (2r-1)/(2r+1) }   \right\},
		\]
		for a constant  $c$ satisfying $c> c_{1,\epsilon}$ then
		%such that 
		\[ \Delta_n^2(\hat{\theta}^{(r)}_{\lambda} - \theta^*)  \leq  c_{2,\epsilon}  n^{  -2r/(2r+1) }  \left(  \log n \right)^{  1/(2r+1)	 }    \]
		with   probability at least  $1-\epsilon$. Here,  $c_{2,\epsilon}>0$ is a constant that only depends on $c$,$\epsilon$ and $V^*.$
		%}
		%O_{\mathrm{pr}}\left\{n^{  -(2r)/(2r+1) }  \left(  \log n \right)^{  1/(2r+1)	 }    \right\} .\]%n^{   -\frac{2(k+1)}{2(k+1)+1}  } \]
	\end{theorem}
	
	%{\color{red} Is the choice of $\lambda$ right here?}
	
	\begin{remark}
		Apart from an extra log factor, the bound in Theorem \ref{thm4} gives the same rate as the bound in Theorem~\ref{thm5}. %The proof  of Theorem  \ref{thm4} uses tools discussed in Section \ref{sec:ideas}   combined with a careful  construction of a restricted set  in  the spirit of \cite{belloni2011}, and  exploiting results from \cite{wang2016trend} and \cite{guntuboyina2017spatial}. Finally,  for simplicity, 
		As we mention above both the choice of $\lambda$ and our upper bound in Theorem \ref{thm4} depend on $V^*$ and it is possible to track down the dependence on $V^*$ by following our proof. However, this dependence on $V^*$ is not simple to state and thus for clarity of presentation we state the above theorem only under the natural scaling $V^* = \Theta(1).$ %we do not track  $V^*$ in Theorem \ref{thm4}. If $V^*$ is allowed to grow to infinity, then, both, the choice of $\lambda$ and the upper bound  in  Theorem  \ref{thm4} would need to be changed  as  functions of  $V^*$.  In particular, the latter would have an additional factor that increases as a function of  $V^*$.
	\end{remark}

	We now present our \textit{fast rate} result for the PQTF estimator.

	\begin{theorem}
		\label{thm6}
		Fix any $r \in \{ 1,2,3,4 \}$. Let $\{y_i\}_{i = 1}^{n}$ be any sequence of independent random variables which satisfies \textbf{Assumption A} and $\theta^*_i$ be the sequence of $\tau$ quantiles of $y_i.$ Consider the same notations as in Theorem  \ref{thm2} and the same minimum length assumption as in~\eqref{eq:minleng}. In addition, suppose that $\theta^*$ satisfies the following two conditions;
		\begin{itemize}
			\item $V^* = O(1)$. 
			\item $\frac{(s+1)}{n}  \log\left(\frac{en}{s+1}\right) (\log n) \log (s+1) = O(1)$ where we recall that $s =\|D^{(r)} \theta^*\|_0.$
		\end{itemize}
		
		Then given any $\epsilon \in (0,1)$ there exists  a  constant  $c_{\epsilon}>0$ only depending on $\epsilon$ and $V^*$ such that if $\lambda$ is chosen to be  
		%	such that if 	%choice  of $\lambda$  satisfying
		\[
		\lambda = c  \max\left\{      \frac{n^{r-1} (s+1)  \log n  \log (s+1)  \log \frac{n}{s+1}    }{V^*},    n^{r-1/2}\left(\frac{1}{s+1} \right)^{r-1/2} (\log n )^{1/2}\right\}
		\]
		for a constant  $c$ satisfying $c> c_{1,\epsilon}$ then
		%such that 
		\[ \Delta_n^2(\hat{\theta}^{(r)}_{\lambda} - \theta^*)  \leq  c_{2,\epsilon} \frac{ (s+1)}{n}  \log\left(\frac{en}{s+1}\right) (\log n)\{\log (s+1)\},\]
		with   probability at least  $1-\epsilon$. Here,  $c_{2,\epsilon}>0$ is another constant that depends on $c$, $\epsilon$ and $V^*$.
	\end{theorem}
	%	(s+1)^{1/2}(\log^{1/2} n ) \left\{  \log \left( \frac{n}{s+1}   \right)  \right\}^{1/2}  (\log ^{1/2} (s+1)), \] 
	%n^{  -(2r)/(2r+1) }  \left(  \log n \right)^{  1/(2r+1)	 }    \]
	
	\begin{remark}
		Theorem \ref{thm6}  shows  that the   PQTF estimator with appropriate tuning parameter,  up to log factors, attains  the same  \textit{fast rate} result   that CQTF  attains in Theorem \ref{thm2}. Theorem \ref{thm6}  can also be thought of as an extension of Corollary 1.2  from \cite{ortelli2019prediction}  to the quantile setting. We rely on the proof machinery developed in~\cite{ortelli2019prediction} whch explains why we can only prove the above theorem for $r \leq 4.$
	\end{remark}

	\begin{remark}
		Both Theorem~\ref{thm4} and Theorem~\ref{thm6} give fast rate results under a particular choice of the tuning parameter. In Theorem~\ref{thm2} we need to set $V = V^*$ which is hard to achieve in practice. Even for the mean regression case with gaussian noise, the best available result (see Corollary $2.3$ in~\cite{guntuboyina2020adaptive}) says that the tuning parameter $V$ should be such that $(V - V^*)^2$ scales like  $\tilde{O}(\frac{s + 1}{n})$ in order to achieve the fast rate. In Theorem~\ref{thm6}, we have more margin of error to choose $\lambda$ as we see from its proof, as long as $\lambda$ is chosen larger than the given threshhold, doubling $\lambda$ will at most double the MSE. In this sense, for attaining the fast rates the PQTF estimator seems to be more robust to the choice of tuning parameter. For our simulations, we have found that  the BIC based approach suggested in~\cite{brantley2019baseline} to choose the tuning parameter in a data driven way works well. 
	\end{remark}

	%\textcolor{red}{As for choosing $\lambda$, we find  the BIC approach  from \cite{brantley2019baseline} to work well in practice.}
	%Theorem  \ref{thm2}  shows that  the constrained  quantile  trend filtering estimator attains, off by a logarithmic factor,  the  rate attained by an oracle estimator that knows the set $S $. Thus, Theorem \ref{thm2} can be thought of as generalizing Theorem 2.2 of \cite{guntuboyina2020adaptive} to the quantile setting. Our minimum length assumption is identical to the one assumed by \cite{guntuboyina2020adaptive}. In particular it requires that when two consecutive change points correspond to two  opposite changes in trend, then the two points should be sufficiently separated.	
	\subsection{Result for Quantile Total Variation Denoising}\label{sec:2d}

	Total variation Denoising (TVD) in $2$ dimensions was proposed by~\cite{rudin1992nonlinear} which subsequently has become a standard image denoising method. In this subsection, we propose the quantile version of the TVD estimator
	and study its risk properties in general dimensions.

	%the  problem of quantile fused lasso in  $d$  dimensions.  In particular, we will exploit ideas from Section \ref{sec:ideas} combined with  results from \cite{hutter2016optimal} to obtain an upper bound, under the loss $\Delta_n^2(\cdot)$.

	Fix a dimension $d \geq 2.$ Let us denote the $d$ dimensional lattice with $n$ points by $L_{d,n} \coloneqq \{1,\dots,m\}^d$ where $n = m^{d}.$ We can also think of $L_{d,n}$ as the $d$ dimensional regular lattice graph with edges and vertices. Then, thinking of $\theta \in  R^{n}$ as a function on $L_{d,n}$ we define
	\begin{equation}\label{eq:TVdef}
		\mathrm{TV}(\theta) \coloneqq  \frac{1}{m^{d-1} }\sum_{(u,v) \in E_{d,n}} |\theta_{u} - \theta_{v}| 
	\end{equation}
	%{\red{}Check whether $1/d$ is required}
	where $E_{d,n}$ is the edge set of the graph $L_{d,n}.$ %The $1/n^{d - 1}$ factor is just a normalizing factor so that if $\theta = f(i_1/n,\dots,i_d/n)$ for some underlying differentiable function on $[0,1]^d$ then $\TV(\theta)$ is precisely the discretized Reimann approximation for $\int_{[0,1]^d} \big|\frac{\partial f(x_1,\dots,x_d)}{\partial x_1}\big| + \dots + \big|\frac{\partial f(x_1,\dots,x_d)}{\partial x_d}\big|.$
	One way to motivate the above definition is as follows. If we think  $\theta[i_1,\dots,i_n] = f(\frac{i_1}{n},\dots,\frac{i_d}{n})$ for a differentiable function $f: [0,1]^{d} \rightarrow  R$ then the above definition  is precisely the Reimann approximation for $\int_{[0,1]^d} \|\nabla f\|_1.$ Of course, the definition in~\eqref{eq:TVdef} applies to arbitrary arrays, not just for evaluations of a differentiable function on the grid. See \cite{sadhanala2016total} who calls this scaling the \textit{canonical scaling}. %for more discussion on this.

	%This notion of total variation extends the definition of \textit{variation} from differentiable functions on $[0,1]^d$ to arbitrary $d$ dimensional arrays. %This $\frac{1}{n^{d - 1}}$ scaling is termed as the \textit{canonical} scaling in~\cite{sadhanala2016total}. 

	We now define the Quantile Total Variation Denoising estimator (QTVD) as follows:
	$$\hat{\theta}_{V} = \arg \min_{\theta \in R^{n}: \mathrm{TV}(\theta) \leq V} \sum_{v \in L_{d,n}} \rho_{\tau}(y_v - \theta_v)$$ 
	where $V$ is a tuning parameter.

	For $d \geq 2$, this is the quantile version of the usual constrained TVD estimator where again the $\rho_{\tau}$ function is replaced by the $x \rightarrow x^2$ function. The risk properties of the usual constrained TVD estimator have been thoroughly studied in~\cite{chatterjee2019new}. The corresponding penalized version of the TVD estimator has also been studied in~\cite{hutter2016optimal},~\cite{sadhanala2016total}. These works show that a well tuned TVD estimator is nearly (up to log factors) minimax rate optimal over the class $\{\theta \in R^{n}: \mathrm{TV}(\theta) \leq V\}$ of bounded variation signals in any dimension. The following theorem extends this result to the quantile setting. 

\begin{theorem}
	\label{thm:2dtv}
	Suppose that \textit{Assumption A} holds. If   $V$ is chosen to satisfy $V \geq V^* := \mathrm{TV}(\theta^*)$  and $V^* = O(1)$, 
	then
	\[
	\Delta_n^2 (  \hat{\theta}_{V} -\theta^*  )  \,=\,  O_{\mathrm{pr  } }\left\{ \frac{ V (\log n)^2 }{   n^{1/d}  }  \right\},
	\]
	for   $d = 2$, and  
	\[
	\Delta_n^2 (  \hat{\theta}_{V} -\theta^*  )  \,=\,  O_{\mathrm{pr  } }\left(   \frac{   V\log n }{   n^{1/d} }  \right),
	\]
	for  $d>2$.
\end{theorem}
	
	%{\color{red} Is $V^* = O(1)$ needed here?}

	\begin{remark}
		%Notice that   Theorem   \ref{thm:2dtv}   requires that true signal  has total variation along $G_{d}$  which is of order $O(1)$. This is a standard setting for denoising in grid  graphs, in fact  \cite{sadhanala2017additive}  refers to this scaling of the total variation as ``canonical''. Under this condition and Assumption \ref{as2}, 
		Theorem   \ref{thm:2dtv}   shows that the QTVD estimator is minimax rate optimal  over the class $\{\theta \in  R^{n}: \mathrm{TV}(\theta) \leq V\}$ of bounded variation signals in any dimension $d \geq 2$, see discussion in Section \ref{sec:discussion}. This result can be thought of as generalizing  Theorem 2.1 from \cite{chatterjee2019new} to the quantile  regression setting.
		%shows  that   quantile fused lasso in $d$ dimensions attains  minimax rates under  the loss $\Delta_n^2(\cdot)$ provided that $V \asymp V^*$. These rates match those in \cite{chatterjee2019new} for the constrained fused lasso in two dimensions,  see also \cite{hutter2016optimal} for the corresponding result for the penalized estimator in $d$ dimensions. However,  unlike previous results, Theorem \ref{thm:2dtv} holds  with a different metric than the mean squared error and it holds under more general settings than sub-Gaussian errors.% {\color{red} Mention it matches the rates of existing works with the differences being.... Make the $V$ scale like big o 1?}
	\end{remark}

	\subsection{Result for High-dimensional Quantile Linear Regression}

	Now we consider high-dimensional linear quantile regression. We study the constrained version of the $\ell_1$-QR  estimator  defined in \cite{knight2000asymptotics} and  studied in \cite{belloni2011}. $\ell_1$-QR  is commonly used as a robust tool for variable selection and prediction with high-dimensional covariates and is the quantile version of the constrained lasso estimator proposed in~\cite{tibshirani1996regression}.

	%More specifically, suppose that  we are given  $\{(x_i,y_i)\}_{i=1}^n \subset R^p  \times  R$  with the  $\{x_i\}_{i=1}^n$ fixed, and  with  $y_1,\ldots,y_n$ independent random variables satisfying that the $\tau$ quantile of $y_i$ is a linear function evaluated at $x_i.$ That is, $\theta^*_i = x_i^{T} \beta^*$ where $\beta^* \in  R^p$ are the unknown vector of coefficients. {\color{red} Mention this is a standard assumption in the qtl reg literature?Here we are only assuming this for a single fixed quantile and not jointly}
	
	%\begin{assumption}
	%	\label{as2}

	%Next we consider a standard assumption in high dimensional quantile regression in the fixed design version of the model used in \cite{belloni2011}.
	
	%\textbf{\noindent{Assumption B:}} 
	Suppose that  we are given $\{(x_i,y_i)\}_{i=1}^n \subset R^p  \times  R$  with the  $\{x_i\}_{i=1}^n$ fixed, and  with  $y_1,\ldots,y_n$ independent random variables.  Let  $X \in  R^{n\times p}$, whose $i$th row is $x_i^{\top}$. We now consider the estimator
	
	%satisfying that the $\tau$ quantile of $y_i$ is a linear function evaluated at $x_i.$ That is, $\theta^*_i = x_i^{T} \beta^*$ where $\beta^* \in  R^p$ are the unknown vector of coefficients. In other words, the linear model is assumed to be correct here. 
	%\end{assumption}

	%the quantile  relation
	%\begin{equation}
	%\label{as3}
	%F_{y_i}^{-1}( \tau    )    =     x_i^{\top} \theta^*,
	%\end{equation}
	
	%where $ F_{y_i}$ is the cumulative distribution function of  $y_i$,  with   $\theta^*   \in R^p$  and   $\|\theta^*\|_1 = s$. With this setting, we focus on the goal of estimating  $\theta^*$. Towards that end, 

	\begin{equation}
		\label{eqn:l1_qr}
		\displaystyle \hat{\theta}   \,= \,
		% \begin{array}{ll}
		\underset{\theta \in R^n: \theta = X \beta, \|\beta\|_1 \leq V, \beta \in R^p}{\arg \min}  \left\{ \sum_{i=1}^{n}\rho_{\tau}(y_i - \theta_i) \right\},  \\
		%\text{subject to}&      \theta = X \beta, \|\beta\|_1 \leq L
		%\end{array}
	\end{equation}
	%\]
	where $V$ is a tuning parameter.

	%We now recall Assumption A and state it in this setting. 
	
	%\begin{assumption}
	%	\label{as4}
	%	The vector of quantiles  $\theta^* $ belongs to $K$. Moreover, 	there exists  a positive %constant  $L$  such that   for   $u\in R$  satisfying  $\vert u  \vert \leq L$  we have that
	%	\[
	%	\underset{i=1,\ldots,n}{\min}\,    \vert   F_{y_i}(x_i^{\top}\theta_i^* + u)  %-F_{y_i}(x_i^{\top}\theta_i^*) \vert\,  \geq \,  \underline{f}\,  \vert u\vert,
	%	\]
	%	for some  $\underline{f}>0$,  where   $f_{y_i}$  is the probability density function of %$y_i$.
	
	%\end{assumption}

	%{\color{red} Mention Assumption that the linear model holds in this setting. Do we need to write our main assumption again? Then mention such an assumption appeared in blah such as
	%The previous assumption is the version of Assumption \ref{as2}  for   the setting of high-dimensional regression. A related  condition  appeared in \cite{belloni2011}.}

	%Our next assumption  states that  the columns of the design matrix are normalized. This is a standard condition in high-dimensional regression, see \cite{rigollet2015high}  for a review. 
	%{\color{red} Give a name to this assumption and write a equation}

	%\begin{assumption}
	%	\label{as4.2}
	%	Let  $X \in R^{n \times p}$ be the matrix whose  $i$th  row is the  vector  $x_i^{\top}$. Denote  by  $X_{\cdot,j}$ the  $j$th column of  $X$. We  assume that  $\max_{j=1,\ldots,p}   \|X_{\cdot,j}\|  \leq  n^{1/2}$. 
	%We assume that the columns of the matrix  $X$ are normalized. 
	%\end{assumption}

	With the notation  from above, we now present our next result.

\begin{theorem}\label{thm:lasso}
	Suppose  that \textit{Assumption A} holds and  $\theta^* =   X \beta^*$ for some  $\beta^* \in R^p$.
	%and  let  $X \in  R^{n\times p}$, whose $i$th row is $x_i^{\top}$. 
	If $V$ is chosen such that $V \geq V^*: = \|\beta^*\|_1$ then there exists  a constant  $C>0$  such that \[ 
	E \left\{ \Delta_n^2\left( \hat{\theta} -\theta^*   \right)\right\}\, \leq\,  \frac{  C V      \left( \log p\right)^{1/2} \,\underset{j=,1\ldots,p}{\max}\|X_{\cdot,j}\|   }{n},
	%\left(  \frac{\log p}{n}  \right)^{1/2},
	\]
	where $X_{\cdot,j}$ is the $j$th column of  $X$ and $\hat{\theta }$  is the estimator defined in (\ref{eqn:l1_qr}).
	
	%	In addition, assume that the columns of  $X \in \mathbb{R}^{n\times p}$, whose $i$th row is $x_i^{\top}$, are normalized in the sense that
	%$$\max_{j=1,\ldots,p}   \|X_{\cdot,j}\|  \leq  n^{1/2}.$$
	
\end{theorem}

	%{\color{red} Why is this an expectation bound?Mention this is the standard slow rate bound for Lasso in high dim reg generalized to the qtl setting}
	
	\begin{remark}
		Theorem  \ref{thm:lasso}  implies that in the case that the columns of $X$ are normalized, which is the following standard assumption in high dimensional regression,  
		\[
		\underset{j=,1\ldots,p}{\max}\|X_{\cdot,j}\|   \leq n^{1/2},
		\]
		then we attain  a slow rate bound scaling like $\{\log p/n\}^{1/2}$.  It is well known that such a bound holds for the usual lasso without any assumptions on the design matrix $X$; see for instance~\cite{chatterjee2013assumptionless}. To the best of our knowledge, this slow rate bound for quantile lasso has not appeared in the literature before. Previous works \citep{belloni2011,fan2014adaptive,sun2019adaptive} make restricted eigenvalue type assumptions on the design matrix and attain fast rates of convergence.
		
	\end{remark}

	%\begin{remark}
	%	\textcolor{red}{Theorem  \ref{thm:lasso} holds without any assumptions on the design matrix $X$.  Previous works \citep{belloni2011,fan2014adaptive,sun2019adaptive} make restricted eigenvalue type assumptions on the design matrix and attain fast rates of convergence.
	%	}
	% \end{remark}

	%\textcolor{red}{change notation?}

	%Furthermore, Theorem \ref{thm6} generalizes  Theorem 2.4 from \cite{rigollet2015high} to the quantile regression setting.  Importantly, unlike previous works \citep{belloni2011,fan2014adaptive,sun2019adaptive}, 	Theorem \ref{thm6}   holds without conditions on the eigenvalues of the design matrix.
	%	Theorem \ref{thm6}   holds without conditions on the eigenvalues of the design matrix. This is a crucial difference from previous work in the literature that relies on  restricted eigenvalue conditions, see for instance \cite{belloni2011,fan2014adaptive,sun2019adaptive}. 

	%{\color{red} I would not write the price we pay. Extra sqroot L factor? really? 
	%However, the price we pay is that our upper bound is stated  in terms of the function $\Delta_n^2(\cdot)$ rather than the mean squared error. Furthermore,  our rate has an extra $s^{ 1/2 }$ factor as compared to that of Theorem 2 in \cite{belloni2011},  which holds under  stronger assumptions than the minimal assumptions in Theorem \ref{thm6}.}
	%	\end{remark}

	\section{Proof Ideas}
	\label{sec:ideas}
	
	\subsection{General Ideas}
	%In this section, we provide an overview of the main ideas underlying our proofs. Full proofs are given in Section~\ref{sec:proofs}.
	%In this paper we actually formulate a general quantile sequence problem under convex constraints. We have a vector $y \in  R^n$ of independent random variables and $\theta^* \in R^n$ is a correspoding vector of $\tau \in (0,1)$ quantiles of $y.$ Suppose it is known that $\theta^* \in K \subset  R^n$ where $K$ enforces a constraint on  
	%the vector $\theta^*.$ This is a generalization of the Gaussian sequence model with constraints on the mean vector. We call this the \textit{constrained quantile sequence problem}. A natural estimator for this problem is the following:
	
	%\begin{equation}
	%	\hat{\theta}_{K}  \,=\,   \begin{array}{ll}
	%		\underset{  \theta \in K}{\arg \min  } &     \displaystyle  \sum_{i=1}^{n} \rho_{\tau}(y_i -   \theta_i) \\
	%	\end{array} 
	%\end{equation}
	
	%We call this estimator the \textit{constrained quantile sequence estimator} or the CQSE estimator. For example, if $K = \{\theta \in  R^n: \TV^{(r)}(\theta) \leq V\}$ for some integer $r \geq 1$ and some $V > 0$ then the above estimator is the CQTF estimator of order $r$ with tuning parameter $V$ as defined in~\eqref{eqn:quantile_trend_filtering2}. 

	In this section, we provide an overview of the main ideas underlying our proofs. Full proofs (along with proof outlines for the major theorems) are given in the Appendix . We first prove a general result about the CQSE estimator (defined in~\eqref{eq:cqse}) when the constraint set $K$ is convex.

	\begin{theorem}\label{thm:basic}
		Let $K \subset R^n$ be a convex set. Let us define a function $\mathcal{R}: [0,\infty) \rightarrow  R$ as follows:
		$$\mathcal{R}(t) = RW(K \cap \{\theta: \Delta^2(\theta - \theta^*) \leq t^2\}).$$ Suppose the distributions of $y_1,\dots,y_n$ obey \textbf{Assumption A}. 
		Then the following inequality is true for any $t > 0$,
		\begin{equation*}
			\mathrm{pr}(\Delta^2(\hat{\theta}_{K} - \theta^*) > t^2) \leq C \frac{\mathcal{R}(t)}{t^2}.
		\end{equation*}
		where $C$ is a constant that only depends on the distributions of $y_1,\dots,y_n$.  
	\end{theorem}
	
	%{\color{red} Not clear how the big o p statement holds. Clear this up}
	As a consequence of  Theorem \ref{thm:basic} we obtain  two corollaries in Sections \ref{sec:cor1} and \ref{sec:cor2}
	that can be used to obtain asymptotic rates of convergence for CQSE estimators.
	
	\iffalse
	{\color{red} We dont need to mention these corollaries here.} 
	
	\begin{corollary}
		\label{cor:basic}
		Consider  the notation from  Theorem \ref{thm:basic}. If  $\{r_n\}$  is  a sequence such that 
		\begin{equation}
			\label{cor:as}
			\underset{t \to  \infty  }{\lim}\,   \underset{n \geq 1}{\sup} \,\frac{  \mathcal{R}( t  r_n  n^{1/2} )  }{  t^2   r_n^2 n}   \,=\, 0,
		\end{equation}
		then
		\[
		\frac{1}{n}\Delta^2(  \hat{\theta}_K - \theta^* ) \,=\,O_{  \mathrm{pr} }\left(r_n^2\right).
		\]
	\end{corollary}

	%{\color{red} Can we state all our results in terms of a tail bound? Otherwise we need to include the big o p statement in this theorem}
	
	The following result is another  simple corollary of Theorem~\ref{thm:basic}. 
	
	\begin{corollary}\label{cor:risk}
		Let $K \subset  R^n$ be a convex set. Suppose the distributions of $y_1,\dots,y_n$ obey Assumption A. Then the following expectation bound holds:
		\begin{equation*}
			E \{\Delta^2(\hat{\theta}_{K} - \theta^*)\} \leq C\:RW(K)
		\end{equation*}
		where $C$ is a constant that only depends on the distributions of $y_1,\dots,y_n$. 
	\end{corollary}
	
	%{\color{red} Has the proof of the corollary been given?}
	
	\begin{remark}
		The above corollary is only useful when the set $K$ is compact as otherwise $RW(K) = \infty$. In this paper, we use this corollary when we consider the case for quantile constrained lasso where $K$ is a linearly transformed $\ell_1$ ball, see Theorem \ref{thm:lasso}.
	\end{remark}
	\fi
	
	To prove Theorem~\ref{thm:basic} we view $\hat{\theta}_{K}$ as an M estimator as we now explain. We define $\hat{M}:  R \rightarrow  R$ and $\hat{M}_i:  R \rightarrow  R$ for each $i \in [n]$ satisfying 
	\[
	\displaystyle	\hat{M}(\theta) =    \sum_{i=1}^{n} \hat{M}_{i}(\theta_i),
	\]
	where
	\[
	\hat{M}_{i}(\theta_i) =  \rho_{\tau}(y_i -  \theta_i)    -  \rho_{\tau}(y_i -  \theta_i^*) .
	\]
	Also define the expected versions $M:  R \rightarrow  R$ and $M_i:  R \rightarrow  R$ for each $i \in [n]$ satisfying 
	\[
	\displaystyle	M(\theta) =    \sum_{i=1}^{n} M_{i}(\theta_i) .
	\]
	where $M_i(\theta_i) = E \{\hat{M}_{i}(\theta_i)\}$. With this notation, the CQSE estimator can also be written as 
	$$\hat{\theta}_{K} = \arg \min_{\theta \in K} \hat{M}(\theta)$$
	and a true quantile sequence $\theta^* = \arg \min_{\theta \in K} M(\theta).$ Therefore, the CQSE estimator is an M estimator or an instance of Empirical Risk Minimization.

	\begin{remark}
		Note that $|\hat{M}_i(\theta_i)| \leq |\theta_i - \theta^*_i|$ for all $i \in [n].$ Therefore, $E \hat{M}_i(\theta_i)$ is always well defined even if $y_i$ does not have any moments. 
	\end{remark}

	Since we are viewing the CQSE estimator as an M estimator, the natural loss function to measure its performance would be $M(\hat{\theta}_{K})$ and show that  $M(\hat{\theta}_{K})$ goes to $0$ as $n \rightarrow \infty.$  Using the M estimation viewpoint, we first prove the following result.

	\begin{proposition}\label{prop:basic}
		Let $K \subset  R^n$ be a convex set. Let us define a function $\mathcal{M}: [0,\infty) \rightarrow R$ as follows:
		$$\mathcal{M}(t) = RW(K \cap \{\theta: M(\theta) \leq t^2\}).$$ %\textcolor{red}{Suppose the distributions of $y_1,\dots,y_n$ obey \textbf{Assumption A}.} 
		Then the following inequality is true for any $t > 0$,
		\begin{equation*}
			\mathrm{pr}(M(\hat{\theta}_{K}) > t^2) \leq \frac{2\mathcal{M}(t)}{t^2}.
		\end{equation*}
		%	where $C$ is a constant that only depends on the distributions of $y_1,\dots,y_n$.
	\end{proposition}

	The above proposition is very similar to Theorem~\ref{thm:basic}, the only difference being that the loss function $\Delta^2(\hat{\theta} - \theta^*)$ is replaced with the function $M(\hat{\theta}).$ This proposition is shown by first reducing the task of bounding $\mathrm{pr}(M(\hat{\theta}_{K}) > t^2)$ to bounding $$\frac{E \sup_{\theta \in K \cap \{v: M(v) \leq t^2\}} [M (\theta) - \hat{M}(\theta)]}{t^2}.$$ Here the numerator in the bound is an expectation of suprema of a mean zero process. We then further bound this expected suprema by using symmetrization and contraction results commonly employed in empirical process theory,  see  Section 2.3 in \cite{van1996weak} and  Theorem  4.12  in \cite{ledoux2013probability}.

	However, handling $M$ in concrete problems such as quantile trend filtering is not convenient as it depends on the distribution of $y.$ Here, a particular property of $M$ comes in handy for us as one can show that if \textbf{Assumption A} holds then for all  $\delta \in R^n$, we have for a constant  $c_0>0$  that
	\begin{equation}\label{eq:losslb}
		M(\theta^*+\delta) \geq   c_0 \Delta^2(\delta) .%   :=\sum_{i=1}^{n}    d(\delta_i)
	\end{equation}
	This is the content of Lemma \ref{lem2} in the Appendix. This makes it possible for us to convert the result in Proposition~\ref{prop:basic} to Theorem~\ref{thm:basic}. Lemma \ref{lem2}  is the reason why we use $\Delta_n^2(\cdot)$ as the loss function throughout this paper.

	The estimator $\hat{\theta}_K$ can be thought of as the quantile version of constrained least squares in the Gaussian sequence model. The study of convex constrained least squares in the Gaussian sequence model has a long history and is, by now, well established (see e.g., 
	\cite{van1990estimating,van1996weak,hjort2011asymptotics,chatterjee2015risk}). The general theory says that risk bounds (under the squared error loss) for the convex constrained least squares estimator can be deduced from the localized Gaussian width term 
	\begin{equation}\label{eq:gw2}
		G_2(t) = GW(K \cap \{\theta:\|\theta - \theta^*\| \leq t\}).
	\end{equation}
	Theorem $3.1$ should be thought of as a quantile version of such a result. In our case, the localized Rademacher width $RW(K \cap \{\theta: \Delta^2(\theta - \theta^*) \leq t^2)$ determines an upper bound on the loss function $\Delta^2(\hat{\theta} - \theta^*)$. Since Rademacher width is upper bounded by a constant times Gaussian width; see Lemma \ref{radamacher_width} in the Appendix, the main difference in our result versus results for convex constrained least squares is that the $\ell_2$ norm is replaced by the loss function $\Delta.$

	\subsection{Theorems  \ref{thm5}, \ref{thm2},  and \ref{thm:2dtv}}
	Theorems \ref{thm5}, \ref{thm2},  and \ref{thm:2dtv} are all bounding the risk for a particular instance of the CQSE estimator. For example, in Theorem~\ref{thm5} the constraint set $K = \{\theta \in  R^n: \mathrm{TV}^{(r)}(\theta) \leq V\}$ where $V \geq V^* = \mathrm{TV}^{(r)}(\theta^*)$ and in Theorem \ref{thm2} we consider the same $K$ with $V = V^*$.

	The starting point for proving Theorems $\ref{thm5}, \ref{thm2}, \ref{thm:2dtv}$ is Theorem~\ref{thm:basic} which behooves us to bound the local Rademacher width term $\mathcal{R}(t)$ for any $t \geq 0.$ Since Rademacher width is upper bounded by Gaussian width, it suffices to bound the local Gaussian width term $G_1(t) = GW(K \cap \{\theta: \Delta^2(\theta - \theta^*) \leq t^2\}).$ Now, tight bounds for the related local Gaussian width term $G_2$ (defined in~\eqref{eq:gw2}) exists in the literature and in particular we use Lemmas  B.1--B.3  from~\cite{guntuboyina2020adaptive}. However, to bound $G_1(t)$ by $G_2(t)$ one needs to bound convert the $\ell_2$ norm in $G_2(t)$ to the $\Delta$ function in $G_1(t).$ The majority of our proof executes this conversion which constitutes one of the technical contributions of this work. We have written more detailed proof outlines for Theorems $\ref{thm5}, \ref{thm2}$ in the Appendix.

	For Theorem \ref{thm:2dtv}, the relevant constraint set is $K = \{\theta \in R^{n}: \mathrm{TV}(\theta) \leq V\}$. Here also, we bound the local Gaussian width term $G_1(t)$ by reducing the problem to bounding $G_2(t)$ which then can be further bounded by using existing results from~\cite{hutter2016optimal}.

	%We use the following elementary inequality for this purpose. 
	
	%\begin{lemma}\label{lem:elem}
	%	\label{lem18}
	%	For all  $v   \in    \mathbb{R}^n $ the following inequality holds: 
	%	\begin{equation}
	%	\label{eqn:ine}
	%	\|v\|^2   \,\leq \,  \max\{ \|v\|_{\infty},1  \} \Delta^2(v).
	%	\end{equation}
	%\end{lemma}

	%To bound $G_1(t)$ for the constraint set $K$, we write it as a sum of $G_1(t)$ for the sets and . We show the second term is small by a direct argument. 
	%The main idea to bound the first term now in our proofs of Theorems $1.1$ and $1.2$ is to show that if $\theta \in K$ then the $\|(\theta - P \theta)\|_{\infty}$ is not too large and then invoking Lemma~\ref{lem:elem} and theorems blah from guntuboyina to bound $G_(t)$. 
	
	\subsection{Theorem   \ref{thm:lasso} }
	In this case, the constraint set is $$K = \{\theta \in R^n: \theta = X \beta, \|\beta\|_1 \leq L, \beta \in R^p\}$$ for a given fixed design matrix $X \in  R^{n \times p}$. Since this is a compact set, we directly use Corollary~\ref{cor:risk} to get an expectation bound. This means that we need to simply bound $RW(K)$ which can be done using standard existing results.

	\subsection{Theorem~\ref{thm4} and Theorem  \ref{thm6}}
	The PQTF estimator is not an instance of the CQSE estimator. Therefore, the proof here is necessarily different. Here, we still continue to use the general idea of viewing the PQTF estimator as a penalized M estimator and using appropriately modified versions of the symmetrization and contraction results. 
	
	However, due to the presence of the penalty term, the proofs of Theorem~\ref{thm4} and Theorem~\ref{thm6} are longer and contain additional preliminary localization arguments compared to the proofs of Theorem~\ref{thm5} and Theorem~\ref{thm2} respectively. In Theorem~\ref{thm6}, we have extensively used the recent ideas developed in~\cite{ortelli2019prediction} and adapted their argument to our quantile setting. 
	
	For the convenience of the reader, we have included a proof outline (before the formal proof) for each of our Theorems $1-4$ in the Appendix. We hope that these proof outlines convey the main ideas of our proofs and make reading our proofs easier.
	
	%\subsection{Theorem  \ref{thm6}}
	
	%\textcolor{red}{The proof of Theorem \ref{thm6} follows, at a high level, a similar  argument to that of Theorem \ref{thm4}. Some differences are the following. First, in \textit{Step 1} instead of using tools from \cite{wang2016trend} we exploit ideas from \cite{ortelli2019prediction}. Second, in \textit{Step 4} we use modified versions of   symmetrization and Talagrand-Ledoux inequality combined with ideas from \cite{ortelli2019prediction}.}
	
	\section{Experiments}

	We now proceed to  illustrate with simulations the empirical performance of quantile trend filtering. As benchmark methods, we consider the usual (mean regression) trend filtering  estimator of order $r=1$ and $r=2$  denoted  as TF1 and TF2 respectively, and quantile smoothing splines (QS) (introduced in~\cite{koenker1994quantile}) which we implement using the R package ``fields" . Notice that TF1 and TF2 only provide estimates for  $\tau =0.5$. As   for quantile  trend filtering, we consider the penalized estimator  (\ref{eqn:quantile_trend_filtering}) with orders $r=1$  and $r=2$ which we denote as PQTF1 and PQTF2 respectively. These are implemented in R via ADMM, similarly to 
	\cite{brantley2019baseline}, using the R package  ``glmgen". 
	%We also compared against quantile random forest  using the R package ``quantregForest" but we omit the results due to poor performance. {\color{red} Why is the last line necessary?}
	
	For the different trend filtering based methods we consider values of $\lambda$
	such that $\log_{10}( \lambda n^{r-1})$  is in a grid of 300 evenly spaced points  between 1 and 4.5 and we  choose their corresponding  penalty parameter to be the value that minimizes the average mean squared error over 100 Monte Carlo replicates. Here, for each instance of an estimator $\hat{\theta}$ we consider the mean squared error $\frac{1}{n}\sum_{i=1}^{n} (  \hat{\theta}_i -\theta^*_i )^2$ as a measure of its performance
	with $\theta^*$  the true vector of quantiles. Additional simulation results reporting the $\Delta_n^2(\theta^* -\hat{\theta})$ values instead of the MSE  are presented in Section \ref{sec:additonal} in the Appendix.
	%\textcolor{red}{}
	%the average MSE over 100 Monte Carlo replicates.
	%{\color{red} Should we report MSE or our loss}

	\begin{table}[t!]
		\centering
		\caption{\label{tab1}  Average mean squared error times  10,   $\frac{10}{n} \sum_{i=1}^n( \theta_i^*-\hat{\theta}_i )^2 $,  averaging over 100 Monte carlo simulations for the different methods considered. Captions are described in the text.  }
		\medskip
		\setlength{\tabcolsep}{14pt}
		\begin{small}
			\begin{tabular}{ rrrrrrrr}
				\hline
				$n$ & Scenario            &$\tau$                     & PQTF1           & PQTF2          &QS                  &TF1               & TF2     \\  
				\hline	
				10000 &1                     &    0.5                   &       0.023    &   0.08   &    0.21     &   \textbf{0.016} & 0.4\\			
				5000 &1                     &    0.5                     &  0.046       &     0.12    &  0.23       &  \textbf{0.034}&0.65  \\	
				1000 &1                     &    0.5                     &     0.18  &    0.29      &  0.32        &    \textbf{0.12}& 0.94 \\			
				\hline				
				10000 &2                    &    0.5                     &   \textbf{ 0.037} &   0.11 &   0.13 & 4917385.2  & 5743.119 \\			
				5000 &2                    &    0.5                      & \textbf{0.066 }&  0.15    &  0.17 &   25215.87          &286.45  \\	
				1000 &2                     &    0.5                      &  \textbf{0.29 }   & 0.43   & 0.45      &    354693.6      &   11522.6        \\
				\hline	
				10000 &3                    &    0.5                      & \textbf{0.015}&  0.063    &   0.17   &   2.26  &    0.95  \\			
				5000 &3                   &    0.5                      &   \textbf{0.029} &   0.092              &     0.18       &    0.14       &    0.65             \\	
				1000 &3                     &    0.5                      &      \textbf{0.13} &  0.24      & 0.26     &      2.23          &     1.04            \\				
				\hline		         
				10000 &4                    &    0.5                      &  0.045   &  \textbf{0.009} &     0.015   &      0.065      &     0.016      \\			 
				5000 &4                   &    0.5                      &   0.075     &    \textbf{0.019}      &     0.027          &   0.24             &      0.031            \\	
				1000 &4                     &    0.5                      &0.30       &    \textbf{0.082}      &   0.098       &       0.29   &           0.31         \\	
				\hline	
				10000 &5                   &    0.5                      &    0.13      &  0.056    &   \textbf{0.041}  &   61625.82         &      134.80    \\			
				5000 &5                   &    0.5                      &    0.24  &     0.099    &   \textbf{0.086}     &     1063110.0           &     877.85    \\	
				1000 &5                   &    0.5                      & 1.92    &      \textbf{0.35}    &  \textbf{0.35}     &        1443060.0         &     11531.79   \\		
				\hline			
				%10000 &6                     &    0.9                      &                &        \textbf{}          &             &         *        &            *            \\	               
				10000 &6                   &    0.9                     &      0.18     &\textbf{0.070}  &    0.075         &         *       &                  *      \\	   
				5000&6                   &   0.9                      &    0.29        &  \textbf{0.13}  &   0.14          &             *    &                 *       \\			         
				1000&6                   &   0.9                      &      1.19      &  \textbf{0.39}  &  0.41           &             *    &                 *       \\			
				10000 &6                   &   0.1                      &  0.16    & \textbf{0.065}  &  0.070        &          *       &      *                  \\			
				5000 &6                   &    0.1                     & 0.31   &\textbf{0.13}&   0.14&            *     &         *               \\	
				1000 &6                     &    0.1                      & 1.27    &  \textbf{0.46}    & 0.47    &           *      &              *          \\            
				\hline  
			\end{tabular}
		\end{small}
	\end{table}

	\begin{figure}[t!]
		\begin{center}%1.42/1.8
			\includegraphics[width=1.42in,height=1.5in]{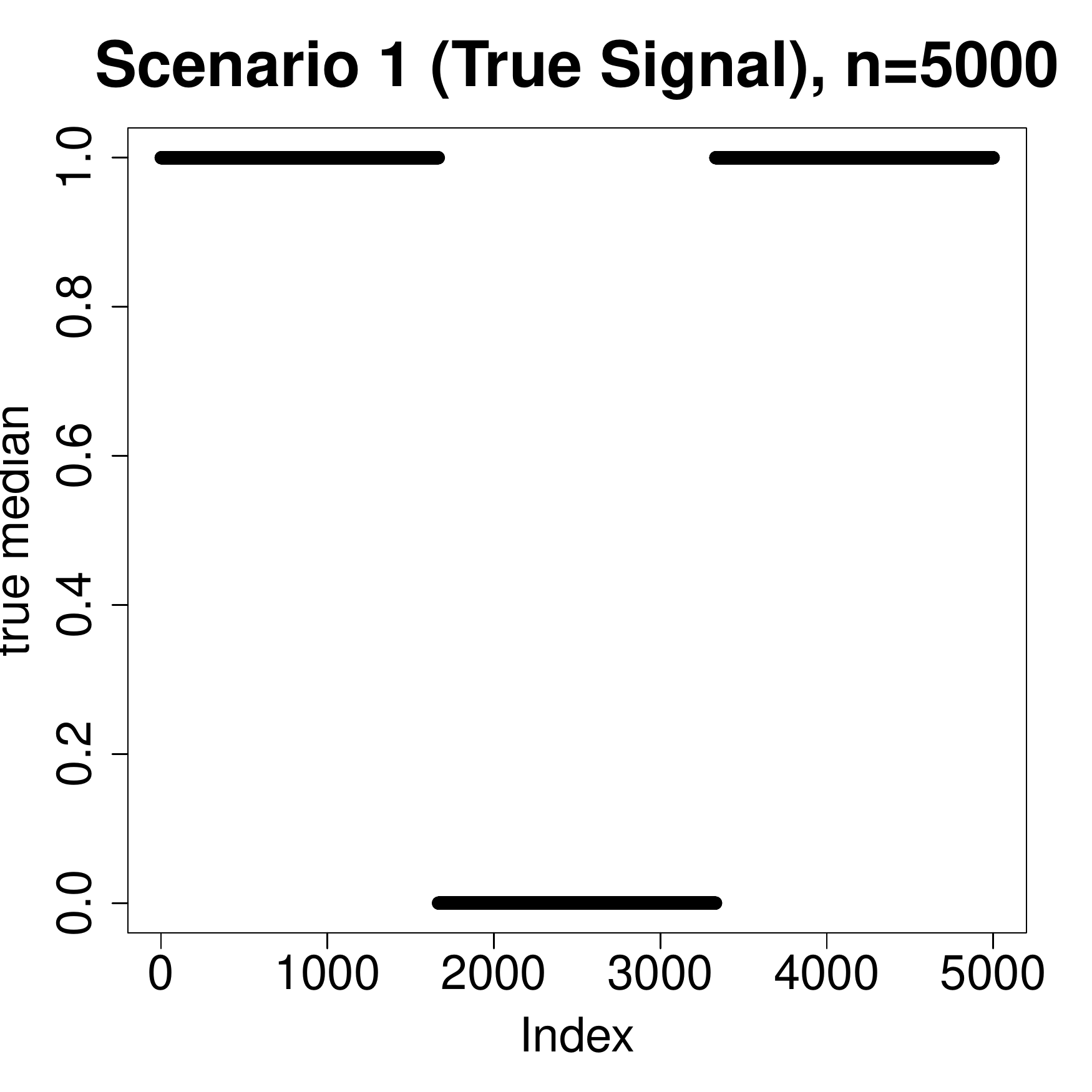} %scenario_p21n1media
			\includegraphics[width=1.42in,height=1.5in]{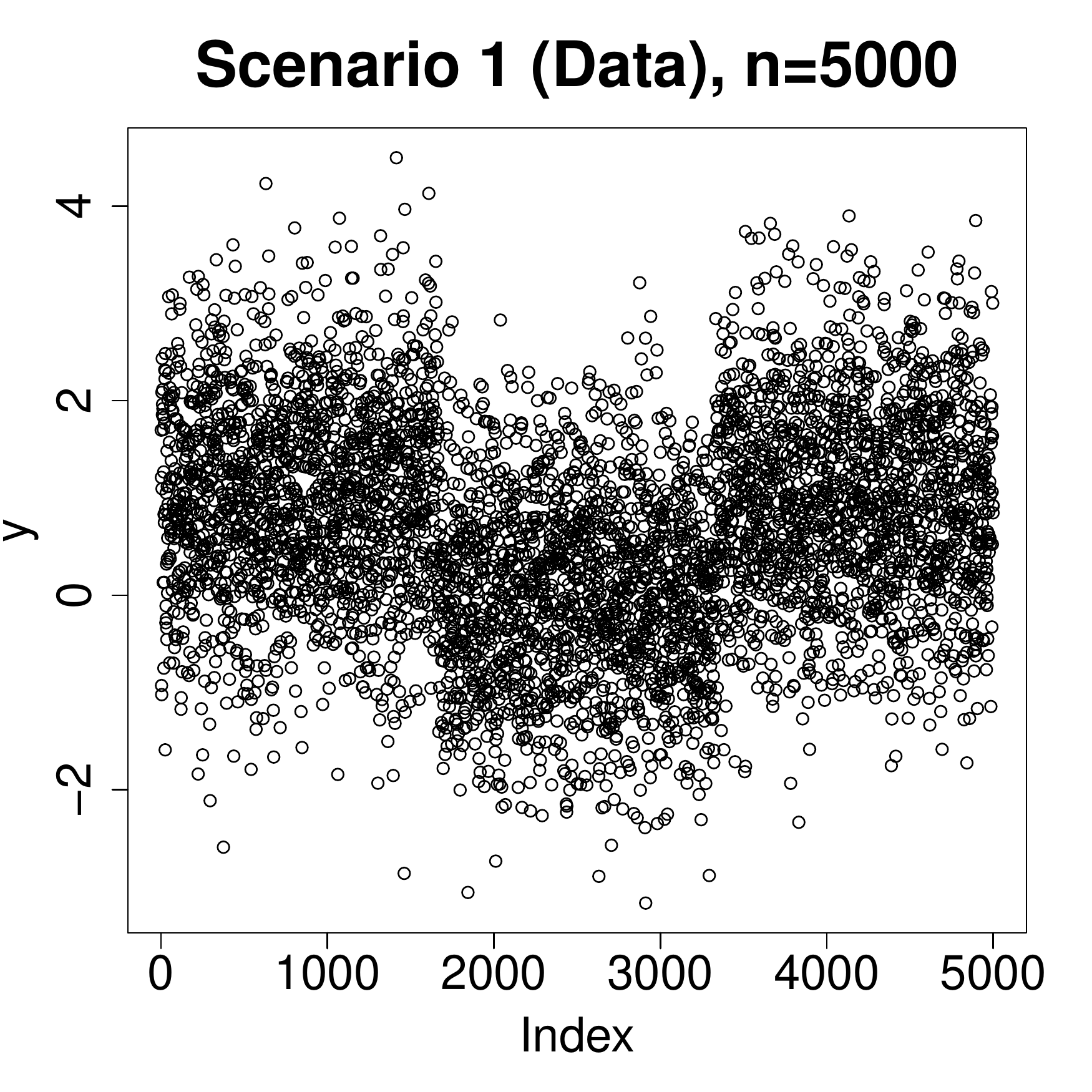} 
			\includegraphics[width=1.42in,height=1.5in]{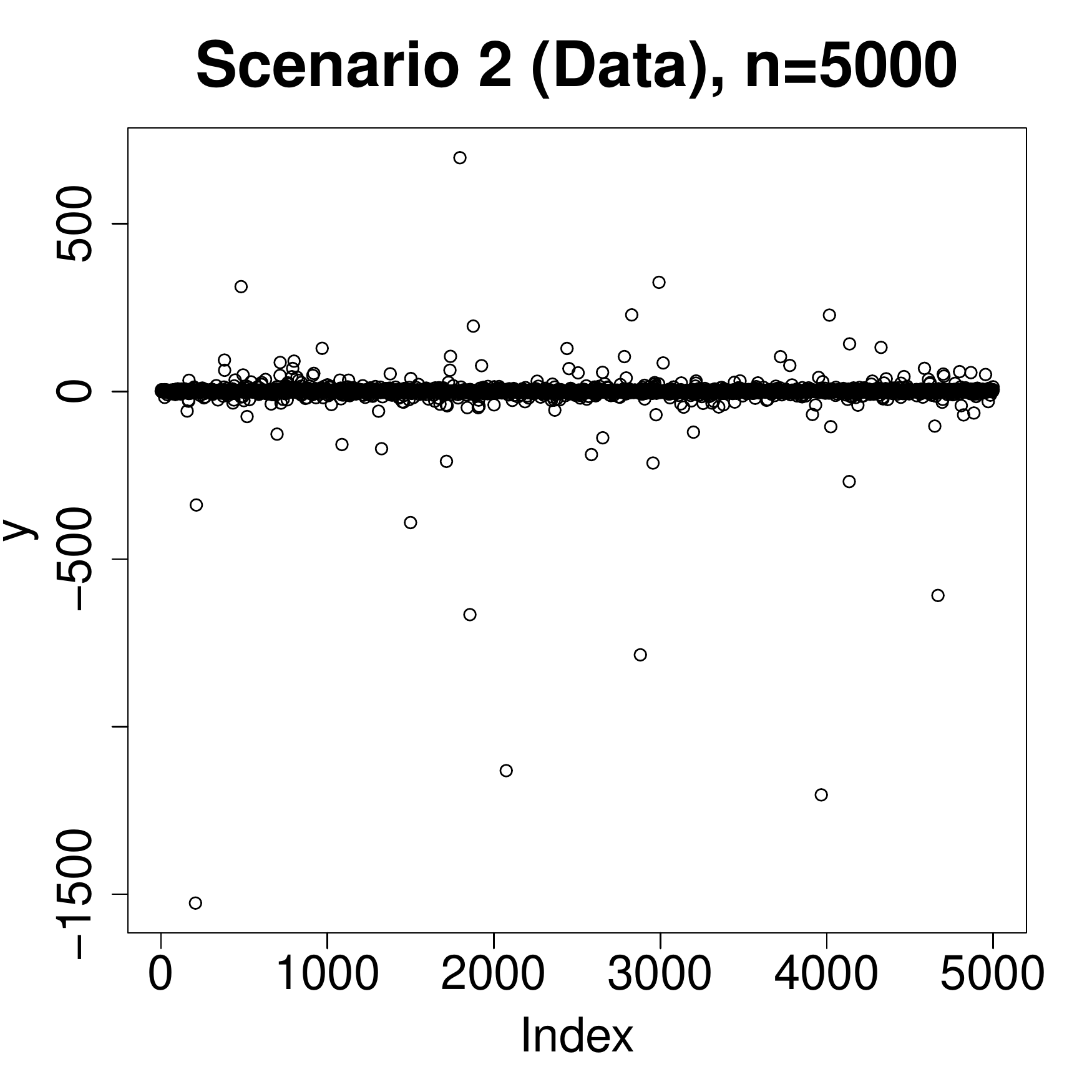} 
			\includegraphics[width=1.42in,height=1.5in]{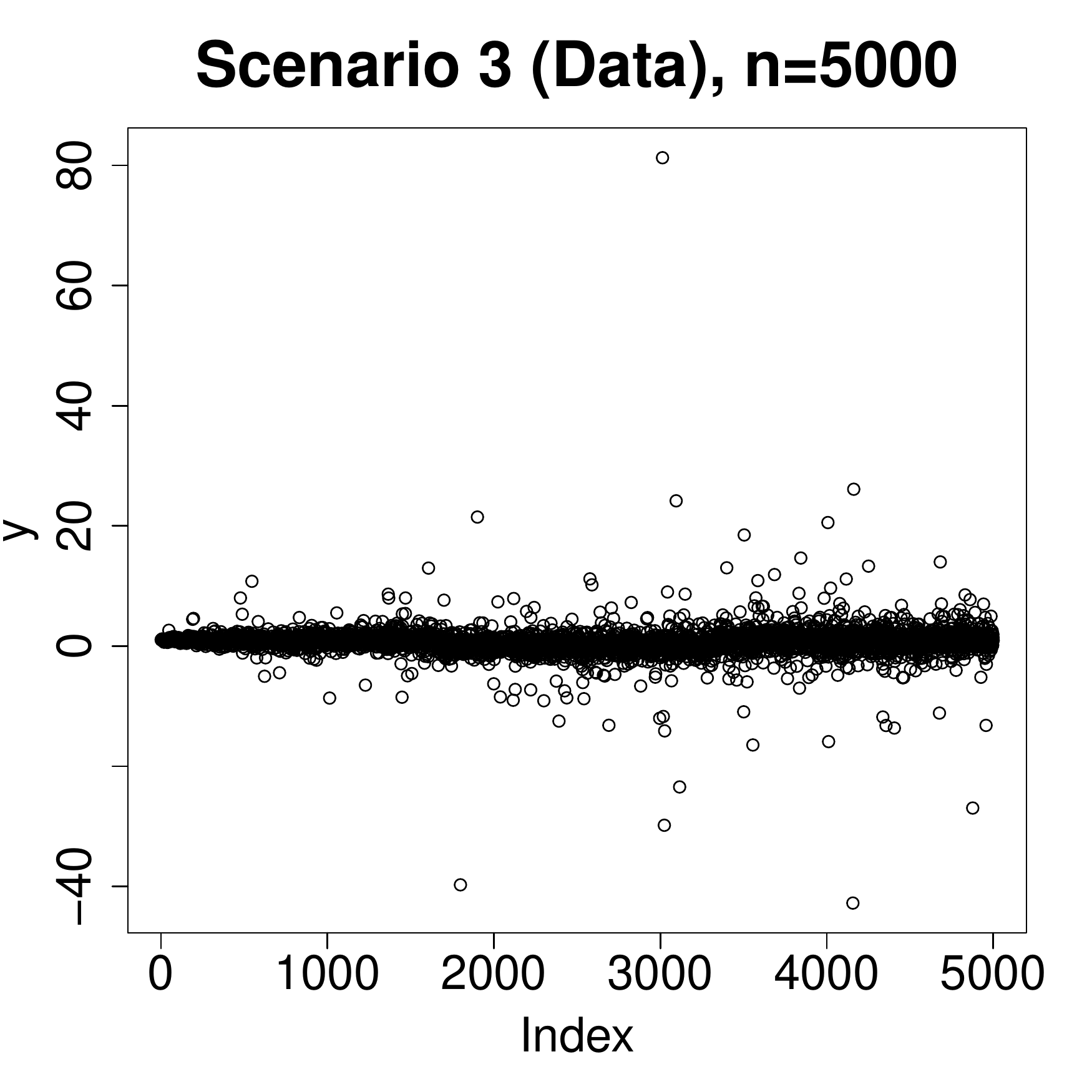} 
			\includegraphics[width=1.42in,height=1.5in]{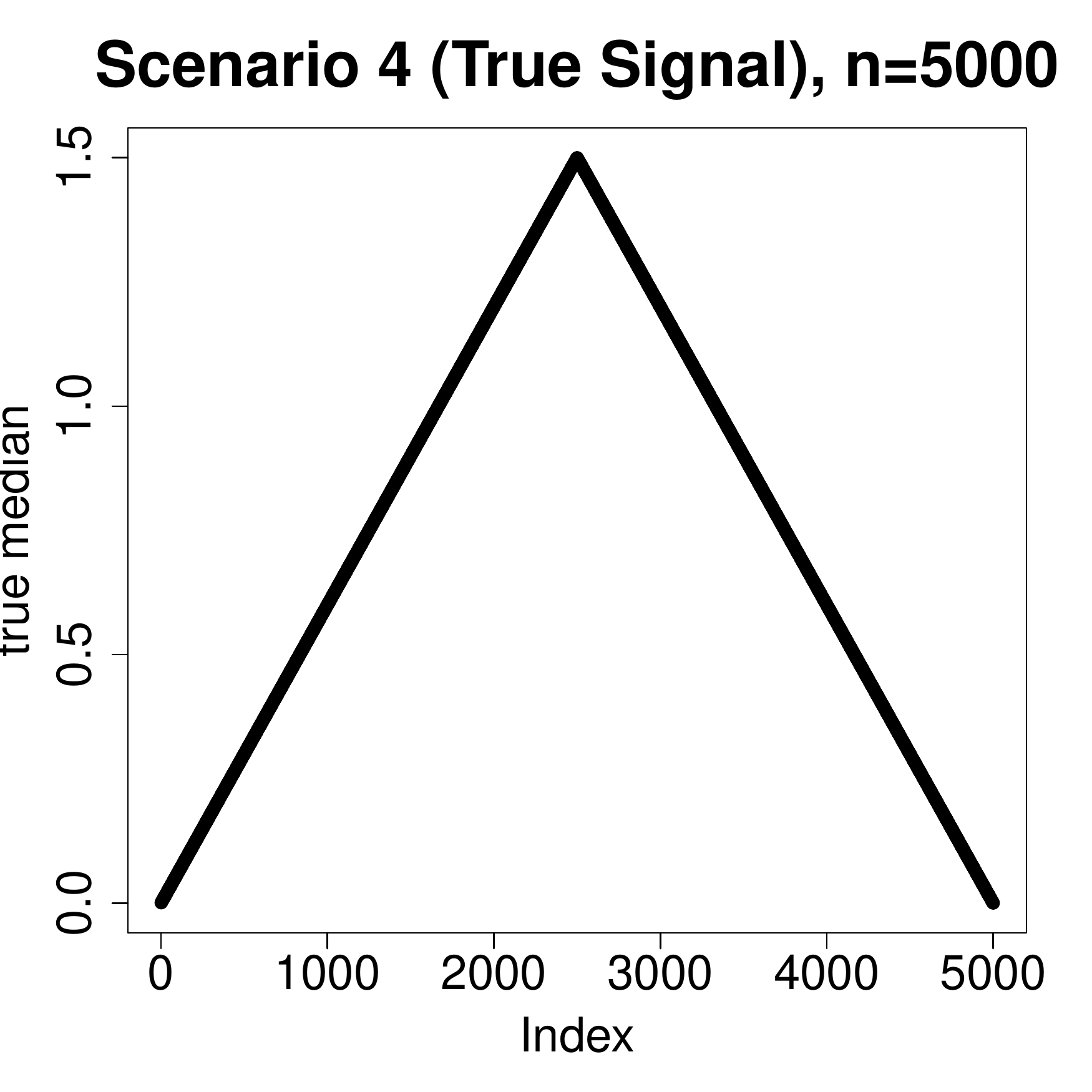} %scenario_p21n1media
			\includegraphics[width=1.42in,height=1.5in]{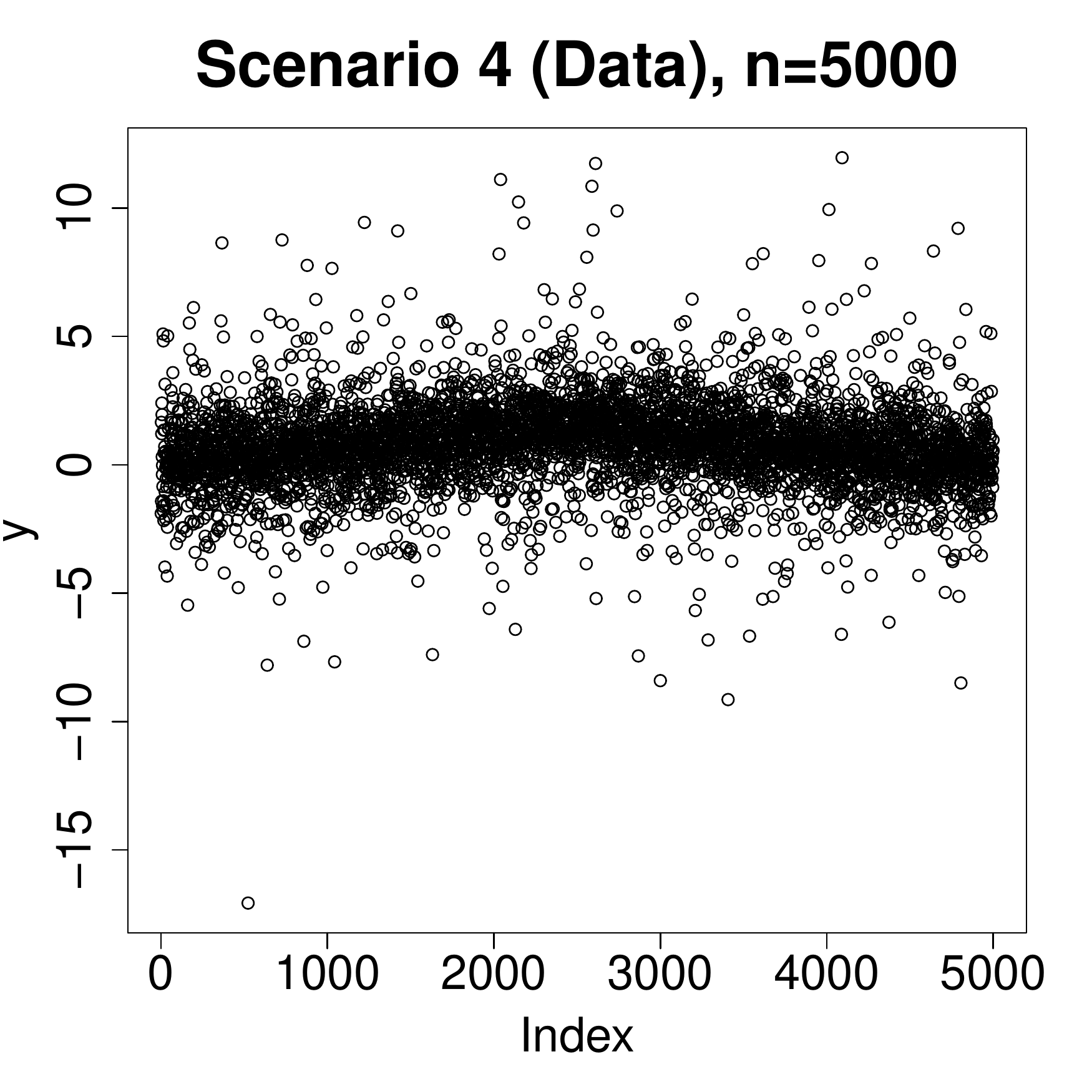} 
			\includegraphics[width=1.42in,height=1.5in]{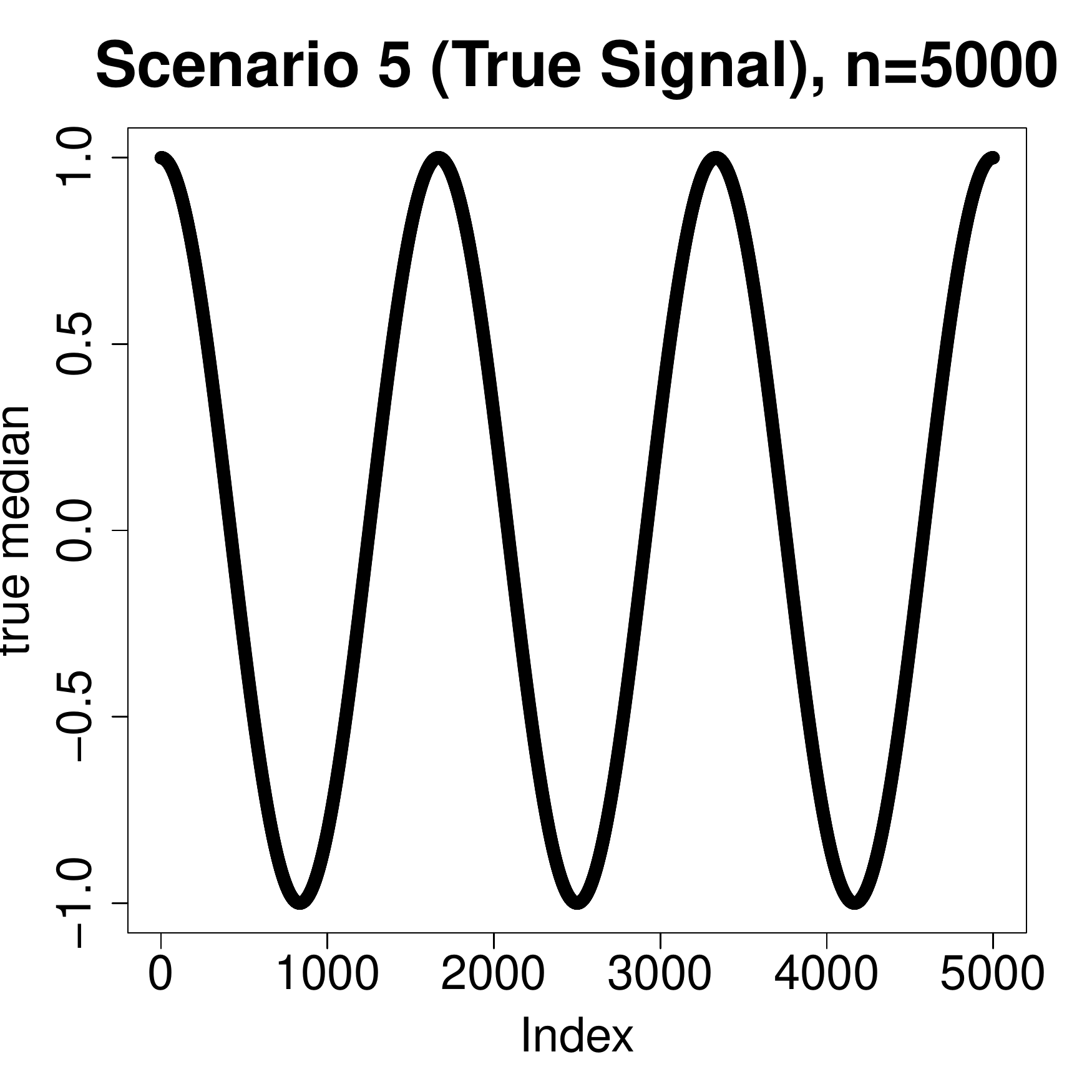} %scenario_p21n1media
			\includegraphics[width=1.42in,height=1.5in]{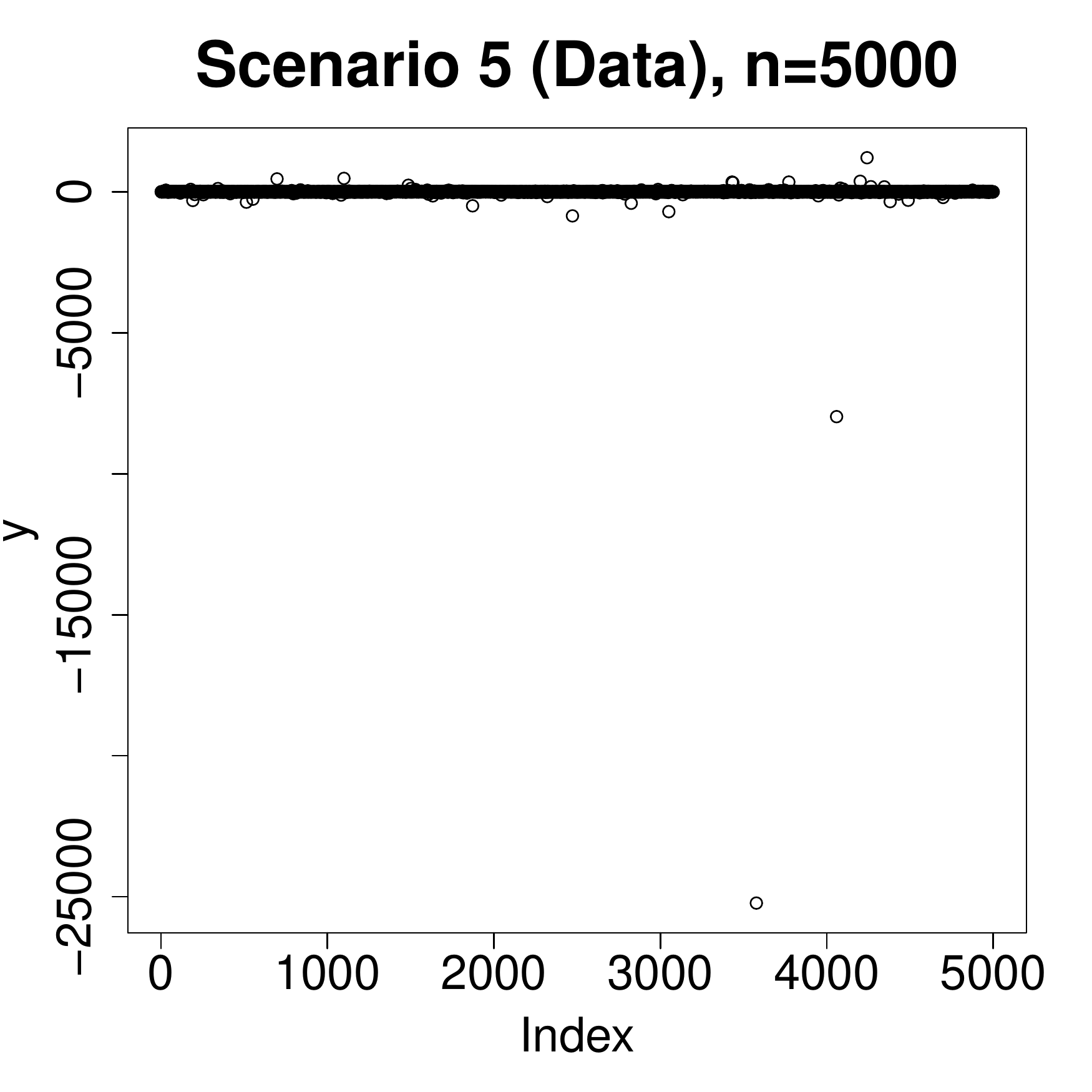} 
			\includegraphics[width=1.42in,height=1.5in]{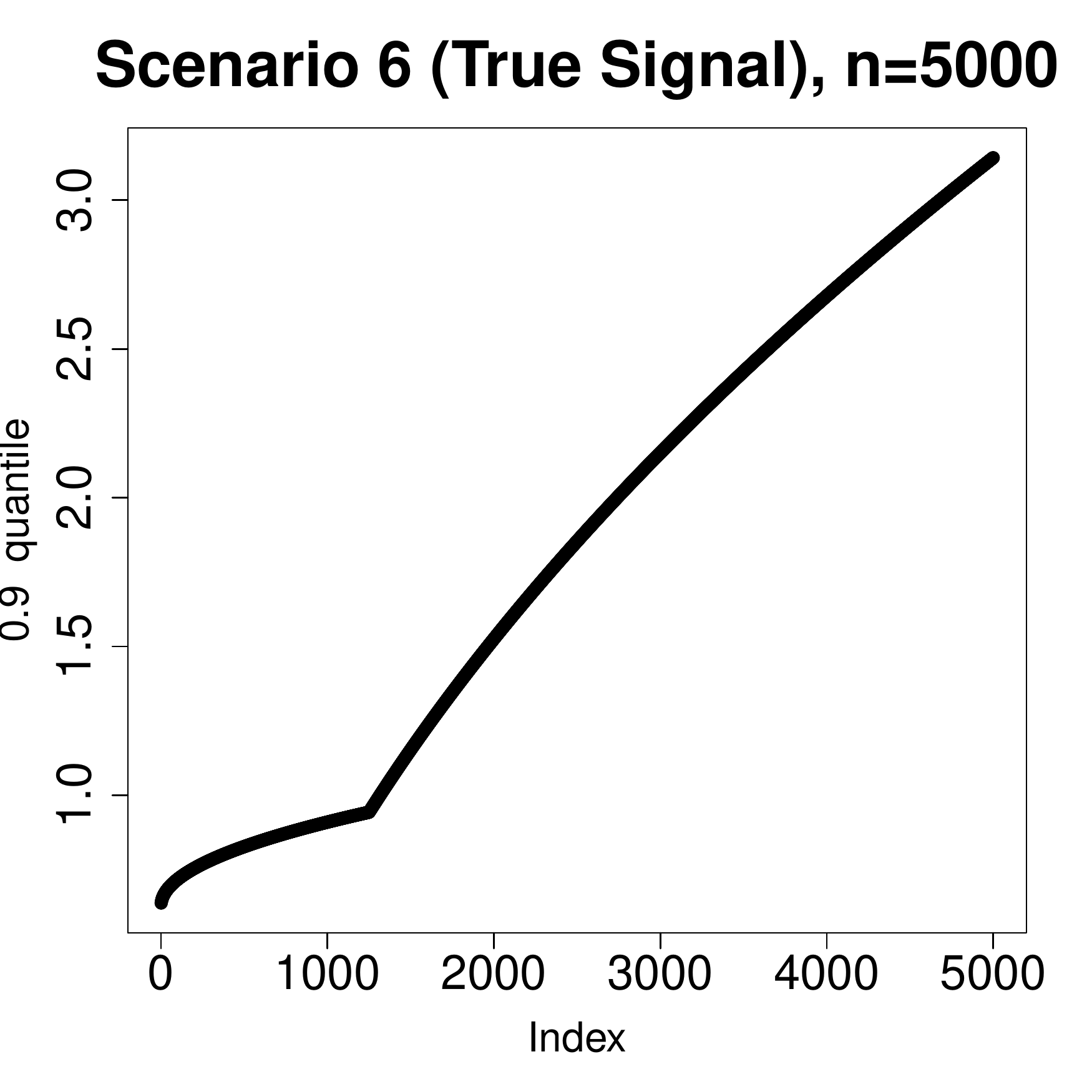} %scenario_p21n1media
			\includegraphics[width=1.42in,height=1.5in]{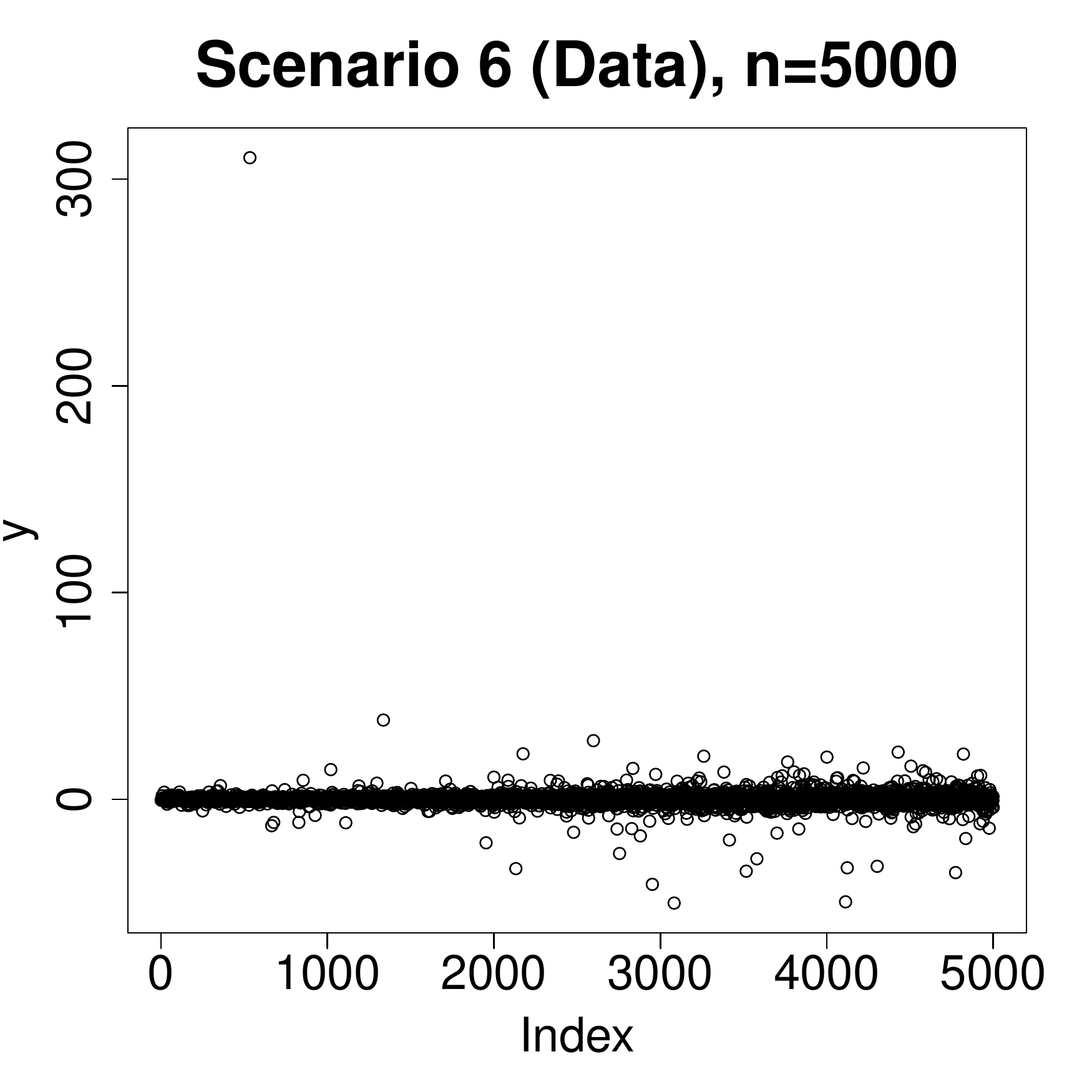} 
			%scenario_p21n1median
			%		\includegraphics[width=2in,height=2.2in]{figures/v2_scenario_p23n1median.pdf} %  		scenario1n2median
			%\includegraphics[width=2in,height=2.2in]{figures/scenario_p53n3q10.pdf} %  		
			%\includegraphics[width=2in,height=2.2in]{figures/scenario_p53n2q10.pdf} %  		
			%\includegraphics[width=2in,height=2.2in]{figures/scenario_p53n1q10.pdf} %  		
			\caption{ 		\label{fig1} The top left panel shows $\theta^*$, the true median, for Scenarios 1,2, and 3.   The next three panels in the top row correspond to data generated according to Scenarios 1, 2 and 3.  Similarly, the  middle panels show the true median curve and  instances of data for  Scenarios 4 and 5. Finally, the bottom  row shows the true quantile curve   for Scenario  6 associated with  $\tau =0.9$, and an instance of data generated according to Scenario 6. }
		\end{center}
	\end{figure}

	Next we describe  the generative models or scenarios. For each scenario we generate 100 data sets for different values of $n$ in the set $\{1000,5000,10000\}$. We then report  the average mean squared error, based on optimal tuning, of the different competing methods. In each scenario the data are generated as 
	\begin{equation}
		\label{eqn:model}
		%	 \[
		y_i    =    \theta^*_i  +  \epsilon_i,\,\,\,\,\,\,i =1,\ldots,n,
		% 	\]   
	\end{equation}
	where  $\theta^*  \in R^n$, and the errors  $\{\epsilon_i\}_{i=1}^n$ are independent with  
	$\epsilon_i  \sim   F_i$  for  some distributions $F_i$ with $i\in [n]$.  We now explain the different choices of $\theta^*$  and  $F_i$'s that we consider.
	
	\emph{Scenario 1(Piecewise Constant Quantiles, Normal Errors)}  In  this  case  we  take  $\theta^*$  to satisfy  $\theta_i^* = 1$  for  $i \in   \{  1,\ldots,\floor{n/3   } \} \cup  \{ n- 2\floor{n/3   }+1,\ldots,n \}$ and  $\theta_i^* = 0$  otherwise.   We take the  $F_i$'s  to be $N(0,1)$. Since the errors are normal and the true signal is piecewise constant, it is natural to expect that TF1 will be the best method. This is verified in Table \ref{tab1}. However, we also see that PQTF1  is a close competitor.
	
	%{\color{red} Are the the indices correct?}
	
	\emph{Scenario 2(Piecewise Constant Quantiles, Cauchy Errors)}  This is the same  as Scenario 1,  where we replace  $N(0,1)$  with $\mathrm{Cauchy}(0,1)$ errors. In this situation the errors  have no mean. As a result, TF1 and TF2  completely breakdown as shown in Table \ref{tab1}. In contrast, as expected, the quantile  methods are robust and can still provide reasonable estimates. In fact, we see that PQTF1 is the best method. This is reasonable since the true median curve is piecewise constant. The second best method is PQTF2.
	
	\emph{Scenario 3(Piecewise Constant Quantiles, Heteroscedastic t Errors)} Once again, we take  $\theta^*$ as in Scenario 1. With regards to the $F_i$'s, we set  $\epsilon_i =    i^{1/2}/n^{1/2}     v_i$, where the $v_i$'s    are independent draws from  $  t(2)$. Here  $t(2)$  denotes  the t-distribution with  $2$ degrees of freedom. The empirical performances here are similar to that of Scenario 2. Table \ref{tab1} suggests that PQTF1 is the best method followed by PQTF2.  Interestingly,  TF1 and TF2  are  not so unreasonable but their behavior seems erratic as the MSE does not  decrease with $n$. A possible explanation for this is that the errors  have mean but do not have variance.   
	
	\emph{Scenario 4(Piecewise Linear Quantiles, t Errors)} We set $\theta_i^* = 3(i/n)$,  for $i \in \{ 1,\ldots, \floor{n/2}\}$,  and $\theta_i^* = 3(1- i/n)$ for $\{\floor{n/2} +1,\ldots,n\}$.  The errors are then independent draws from $t(3)$. Since the true median curve is piecewise linear, Scenario 4  offers  a model that seems more amenable for PQTF2.  This intuition is confirmed in Table \ref{tab1}  where   PQTF2 outperforms the competitors followed by QS. %We also see in Table \ref{tab1} that  TF1 and TF2 have MSEs  that seem to decrease with $n$, perhaps due to the fact that errors have finite mean and variance. However, t

	\emph{Scenario 5(Sinusoidal Quantiles, Cauchy Errors)} The signal is taken as $\theta_i^*  = \mathrm{cos}(6\pi i/n)$  for  $i \in \{1,\ldots,n\}$.  We then generate  $\epsilon_i \sim^{ind} \mathrm{Cauchy}(0,1)$ for $i=1,\ldots,n$. Here the true median  curve   is  infinitely differentiable. Table \ref{tab1} shows that the best performance is given by QS and PQTF2 is a close second. As with the other scenarios that have Cauchy errors, TF1 and TF2  provide poor estimates.
	
	\emph{Scenario 6(Piecewise smooth quantiles, Heteroscedastic Errors)}  For our last scenario we   generate data  as $y $ as
	\[
	y_i  = \begin{cases}
		\frac{v_i( 0.25\sqrt{  (i/n)    }  +1.375)}{3}    &   \text{if}  \,\,\,  i\in \left\{1,\ldots,\floor{n/2} \right\}  \\
		\frac{v_i(7 \sqrt{  (i/n)    }  -2)    }{3}&\text{if}  \,\,\,  i\in \left\{\floor{n/2}+1,\ldots,n \right\},  \\
		%(sqrt(x[1:(n/4)])+1)*qt(tau,2)/3
	\end{cases}
	\] 
	where  the $v_i '$s are  independent draws from  $ t(2)$. Unlike the previous scenarios, Scenario 6 presents a case where the  median is constant  but the  other quantiles  change. For instance, as illustrated in Figure \ref{fig1},  the $0.9$th quantile  is  piecewise smooth  but continuous.  By the nature of Scenario 6, one would expect PQTF2 to be the  best method as the pieces of the $0.1$ and $0.9$th quantile curves can be well approximated by  linear functions. This is indeed what we find in Table \ref{tab1}.

	Finally, Figure \ref{fig1}  illustrates the true signals and one data set example for each of the  different scenarios that we consider. Overall, we see that the PQTF estimator performs well across different scenarios and under the presence of heavy tailed errors thereby supporting our theoretical findings.

	%The results in Table \ref{tab1}  show that, overall,  PQTF1 and PQTF2 outperform the competitors.   For Scenario 1 which consists of  a piecewise constant signal with Gaussian errors, as expected, we can see that TF1  is the best method. For Scenarios 2--3.  which have a piecewise constant median but heavy tail errors, the best method is  PQTF1. For Scenario 5, a model with a smooth median, the best method is quantile splines. Finally,
	%for Scenarios 4 and 6 we observe that PQTF2 outperforms the competitors. This  is reasonable since in such  scenarios $\theta^*$  is or can be well approximated by  a piecewise linear signal.
	
	%{\color{red} Why did we choose these plots?}

	\section{Discussion}
	\label{sec:discussion}

	To summarize, in this paper we have studied  quantile trend filtering and some other quantile regression methods. Our risk adaptive bounds generalize  previous work to quantile setting. The main advantage of our results is that they hold under very general conditions without requiring moment conditions and allowing for heavy-tailed  distributions. We now discuss some issues related to our work in this paper.

	Unlike trend filtering with sub-Gaussian errors, our risk bounds are based on   $\Delta(\cdot)$ instead of the squared error loss function $\|\cdot\|$.  In general, it is the case that the former is smaller. It is a natural question whether our results also hold under squared error loss. One thing we can say is that when the set $K$ in the constrained quantile sequence model is contained in an $\ell_{\infty}$ ball whose radius does not grow with $n$, then our convergence rates based on $\Delta(\cdot)$  also hold under $\|\cdot\|$. This is because $\|\cdot\|$ and  $\Delta(\cdot)$ are equivalent up to constants when evaluated in a compact set.

	We have often stated that our convergence rates are minimax rate optimal. To clarify on this let us consider the case of Theorem \ref{thm5}. As  \cite{nussbaum1985spline}  showed (see the discussion in \cite{tibshirani2014adaptive}), there exists a constant $c > 0$ such that
	\begin{equation}
		\label{eqn:minmax_lower_bound}
		\underset{ \hat{\theta}  }{\inf}\,\,\underset{\theta^*   \,:\,  \mathrm{TV}^{(r)}(\theta^*)\leq V,\,\,\,  \|\theta^*\|_{\infty}\leq 1    }{\sup}\, E\left(     \frac{1}{n} \|  \hat{\theta}   -  \theta^* \|^2   \right)  \,\geq\, c  \left(\frac{V}{n}\right)^{2r/(2r+1)} 
	\end{equation}
	where the inifimum is taken over all estimators  and the $y_i$'s are independent draws from $N(\theta^*_i,  \sigma^2 )$ for a known $\sigma$.  Since the parameter space is within the $\ell_{\infty}$ ball of radius $1$, therefore the left hand side in (\ref{eqn:minmax_lower_bound}) equals the following up to a constant
	\[
	\underset{ \hat{\theta}  }{\inf}\,\,\underset{\theta^*   \,:\,   \mathrm{TV}^{(r)}(\theta^*)\leq V,\,\,\,  \|\theta^*\|_{\infty}\leq 1    }{\sup}\, E\left\{   \Delta_n^2\left(\hat{\theta}   -  \theta^*\right)   \right\}.    
	\]
	It now follows that  the rates in Theorem \ref{thm5} and Theorem \ref{thm4} (up to log factors) are minimax  in the sense that they match the rate in (\ref{eqn:minmax_lower_bound}). Similarly, it can be seen that the rate in Theorem  \ref{thm:2dtv} is minimax up to log factors in that sense that it matches the minimax rates of mean  estimation with sub-Gaussian noise in the class of 2D bounded variation signals, see \cite{hutter2016optimal,chatterjee2019new}.
	
	%	s minimax in the sense that such rate holds for abitrary  distributions and it matches the rate in (\ref{eqn:minmax_lower_bound}).}

	One natural extension of our work is to consider  estimation of  multiple quantiles with  trend filtering. This  can be formulated as follows.
	%\begin{remark}
	%\label{remark1}
	Let  $\Lambda \subset  (0,1) $  be a finite  set and  consider the  estimator
	\begin{equation}
		\label{eqn:multiples}
		\begin{array}{llll}
			\{\hat{\theta}^{(r)}_\tau \}  &  \,=\,&   	\underset{  \{\theta(\tau)\}_{\tau  \in \Lambda } \subset R^n }{\arg \min  } &     \displaystyle   \sum_{\tau \in   \Lambda} \sum_{i=1}^{n} \rho_{\tau}\{y_i -   \theta_i(\tau)\} ,   \\
			&&	\text{subject to } &   \mathrm{TV}^{(r)}\left\{ \theta(\tau) \right\} \leq  V(\tau),\,\,\,\forall  \tau \in  \Lambda \\
			&	&&    \theta(\tau) \leq  \theta(\tau^{\prime})   ,\,\,\,\,\,\forall \tau < \tau^{\prime} ,   \,\,\,\,\tau , \tau^{\prime} \in \Lambda,
		\end{array} 
	\end{equation}
	where $\{V(\tau)\}$ are tuning parameters for  $\tau \in \Lambda$. Let $\theta_i^*(\tau)   $  be a true $\tau$th quantile sequence for each $\tau \in \Lambda$. If \textbf{Assumption} A holds for each $\theta^*(\tau)$ instead of  $\theta^*$,  then  one can show using similar arguments as in the proof of Theorem \ref{thm5} that
	\[
	\sum_{\tau \in \Lambda}\Delta_n^2\left\{  \theta^*(\tau)  -\hat{\theta}^{(r)}_\tau   \right\}    =O_{\mathrm{pr} }\left\{    n^{   -2r/( 2r+1)   }    \right\},
	\] 
	provided that $V(\tau)\geq    \mathrm{TV}^{(r)}\left( \theta(\tau)^* \right)  = O(1)$ and $V(\tau) = O(1)$ . This is an extension of the upper bound in  Theorem \ref{thm5} to the case where we are estimating finitely many quantiles simultaneously. However, it might be of interest to consider the case when the number of quantiles to be estimated is allowed to grow with $n$. We leave this for future investigation. %we are not aware of how to handle the case where size of $\Lambda$ grows with $n$.  We also  notice that without  the monotonicity constraint the   estimator (\ref{eqn:multiples})   would be decomposed into independent problems. 
	
	%	This shows that the upper bound in  Theorem \ref{thm5} translates to the multiple quantile setting. A similar phenomenon also holds for all other constrained estimators  results that we provide. Hence, for simplicity we focus on the analysis of single quantiles.
	%In fact, a similar  result holds  for all the reminder constrained estimators upper bounds that we provide. 
	%\end{remark}

	%In this paper we have proved the \textit{fast rate} bound for the CQTF estimator but not the PQTF estimator. Recently, \cite{ortelli2019prediction}  showed that the \textit{fast rate} bound holds (with extra log factors) for the usual penalized trend filtering estimator for  $r \in \{1,2,3,4\}$. It is an interesting question as to whether the results of~\cite{ortelli2019prediction} can be generalized for the PQTF estimators. We think that our proofs need to be modified significantly for this and therefore this is out of scope of the current paper. 
	
	With regards to our results on both the CQTF and PQTF estimators, all of our bounds are $O_{\mathrm{pr}}(\cdot)$ statements. It would be interesting to attempt to translate these results to expectation or high probability bounds on the estimation error measured with $\Delta_n^2(\cdot)$. Our guess is that if we allow heavy tailed errors such as the Cauchy distribution then the expectation of $\Delta_n^2(\hat{\theta} - \theta^*)$ may not even exist. More investigations need to be done on how heavy the errors can be while ensuring in expectation or high probability bounds for $\Delta_n^2(\hat{\theta} - \theta^*).$

	Since we give a general bound for the convex constrained quantile sequence estimation problem it would also be interesting to investigate whether our proof technique can be used in shape constrained quantile problems such as isotonic regression (see~\cite{chatterjee2015risk}) and convex regression (see~\cite{guntuboyina2015global}).
	
	%However, there are some limitations to our work. For instance,  two frameworks  where our machinery falls short are  isotonic and convex regression. In such settings, we are not  able to extend  to quantile regression the results from   and . %Another setting where our work does not extend existing 

	%Recently, \cite{ortelli2019prediction}  showed a similar bound, with extra log factors, that holds  for the penalized trend filtering estimator for  $r \in \{1,2,3,4\}$. It is an interesting question as to whether the results of \cite{ortelli2019prediction} can be generalized for the PQTF estimators. We do not address such question in this paper. 
	
	%Furthermore,  we emphasize that  when it comes to fast rates of convergence, estimation in the class of piecewise polynomial signals, we have only presented  nearly minimax guarantees for the constrained version of quantile trend filtering.  One potential way to prove the same for the penalized estimator could be to exploit some of the results from \cite{ortelli2019prediction}. However, our proof technique would have to be significantly modified and it goes beyond the scope of this paper.

	It is worthwhile to mention that trend filtering can be generalized for general graphs as was proposed by \cite{wang2016trend} which included theoretical and computational developments. In the particular case of the fused lasso on general graphs, several recent works \cite{padilla2017dfs,madrid2020adaptive,ortelli2019synthesis} have studied its risk properties. It will be interesting to investigate whether these types of results can be extended to the quantile setting.

	Finally, the \textit{Dyadic CART} estimator; orginally proposed in~\cite{donoho1997cart}, has been shown to enjoy computational and certain statistical advantages over trend filtering while nearly maintaining all its known theoretical guarantees; see~\cite{chatterjee2019adaptive}. It would be also be interesting to develop quantile versions of Dyadic CART as an alternative to quantile trend filtering. 
	% {\color{red} Move this para to the discussion setting?}

\section*{Acknowledgement}

The authors  thank Ryan Tibshirani for  helpful and stimulating conversations.
%Acknowledgements should appear after the body of the paper but before any appendices and be as brief as possible
%subject to politeness. Information, such as contract numbers, of no interest to readers, must
%be excluded.
	\appendix

\section{Proof of Proposition~\ref{prop:basic}}\label{sec:proofs}

\subsection{Lemmas Required for Proof of Proposition \ref{prop:basic} }

We first recall the following well known fact bounding Rademacher Width by Gaussian Width; e.g see Page 132 in \cite{wainwright2019high}. %well known fact that will be used  in our proofs. 

\begin{lemma}
	\label{radamacher_width}
	% and also \cite{tomczak1989banach,bartlett2002rademacher}). For a set  $K \subset R^n$
	%, the Rademacher  width of $K$ is defined as
	%	
	%	\[
	%	RW(K) =        E\left(    \underset{v \in K}{\sup}     \sum_{i=1}^{n}  \xi_i  v_i    \right),
	%\]
	%where  $\xi_1,\ldots, \xi_n$  are independent  Rademacher  random variables. Similarly, the Gaussian width  of  $K$ is defined as 
	%\[
	%GW(K) =        E\left(    \underset{v \in K}{\sup}     \sum_{i=1}^{n}  z_i  v_i    \right),
	%\]
	%where  $z_1,\ldots, z_n$  are independent  standard normal random variables.  With this notation, 
	We have 
	\begin{equation}
		\label{eqn:Gaussian_width}
		RW(K)     \leq     \left(   \frac{\pi}{2}   \right)^{1/2} GW(K) 
	\end{equation}
	where 
	\[
	GW(K) =        E\left(    \underset{v \in K}{\sup}     \sum_{i=1}^{n}  z_i  v_i    \right),
	\]
	for $z_1,\ldots, z_n$   independent  standard normal random variables.  
\end{lemma}

%We now prove lemmas that hold  for a general constraint set  $K$.  Throughout this section  $K$  is a subset of $R^n$. We start  by introducing some notation. 
%First, we denote by $F_{y_i}$  the cumulative distribution function of $y_i$. Next, we recall some definitions.
%{\color{red} Where is $F_{y_i}$ used?}

We now recall some definitions. 
\begin{definition}
	The function  $\Delta^2  \,:\, R^n\rightarrow  R$  is defined as%
	%	\[%
	%	 and where
	\[
	\displaystyle 	  \Delta^2(\delta) \,:=\,  \sum_{i=1}^{n} \min\{ \vert  \delta_i \vert, \delta_i^2 \}.
	\]
	We also write  $\Delta(\delta)  =   \{\Delta^2(\delta) \}^{1/2}$.
	%	\] 
\end{definition}

\begin{definition}
	
	\label{def1}
	
	%	\label{def1}
	We define the empirical loss function \[
	\displaystyle	\hat{M}(\theta) =    \sum_{i=1}^{n} \hat{M}_{i}(\theta_i),
	\]
	where
	\[
	\hat{M}_{i}(\theta_i) =  \rho_{\tau}(y_i -  \theta_i)    -  \rho_{\tau}(y_i -  \theta_i^*) .
	\]
	Setting $M_{i}(\theta_i) = E\{\rho_{\tau}(y_i -  \theta_i)    -  \rho_{\tau}(y_i -  \theta_i^*) \}$,  the population  version of $\hat{M}$ becomes
	\[
	\displaystyle	M(\theta) =    \sum_{i=1}^{n} M_{i}(\theta_i) .
	\]
\end{definition}

Notice that in the previous definition the functions  $M$ and $\hat{M}$ depend  on $n$ and  $\theta^*$ but we omit making this dependence explicit for simplicity.

With the notation from Definition \ref{def1}, we  consider the $M$-estimator

\begin{equation}
	\label{eqn:constrained_estimator}
	\displaystyle \hat{\theta}   =  \begin{array}{ll}
		\underset{  \theta \in R^{n} }{\arg \min}&    \hat{M}(\theta)  \\
		\text{subject to}&      \theta \in K,
	\end{array}
\end{equation}
and  $\theta^* \in \arg\min_{\theta \in  R^n}   M(\theta)$.  Throughout, we assume that $\theta^* \in K   \subset R^n$ where $K$ denotes a general constraint set throughout this section.

\begin{lemma}
	\label{lem1}
	With the notation from  before,
	\begin{equation}
		\label{eqn:fbound}
		M(\hat{\theta}) \, \leq\,    \underset{v \in K}{\sup}\,\left\{M(v) - \hat{M}(v)\right\}.
	\end{equation}
	%M(\theta^*)
\end{lemma}

\begin{proof}
	\[
	\begin{array}{lll}
		M(\hat{\theta}) &=& 	M(\hat{\theta}) -    \hat{M}(\hat{\theta})  + \hat{M}(\hat{\theta})  \\
		&\leq& M(\hat{\theta}) -    \hat{M}(\hat{\theta})  \\
		%  &=  & M(\hat{\theta}) -    \hat{M}(\hat{\theta})  +  \hat{M}(\theta^*)     - M(\theta^*) \\
		& \leq &     \underset{v \in K}{\sup}\,\,\left\{M(v) - \hat{M}(v)\right\},
	\end{array}
	\]
	where the first  inequality follows since   $\hat{M}(\hat{\theta})   \leq 0 $.
\end{proof}

Next,  we proceed to bound the right hand side of Equation  \ref{eqn:fbound} by the standard technique of symmetrization.

\begin{lemma}
	\label{lem3}
	%\label{eqn:symmetrization}
	(Symmetrization).	It holds that
	\[
	\displaystyle E\left[  \underset{v \in K}{\sup}\,\,\left\{M(v) - \hat{M}(v)\right\}\right]\leq   2 \,E\left\{   \underset{v \in K}{\sup}\,\,  \sum_{i=1}^{n}  \xi_i \hat{M}_{i}(v_i)   \right\}, 
	\]
	where  $\xi_1,\ldots,\xi_n$ are   independent  Rademacher variables  independent  of  $\{y_i\}_{i=1}^n$.
\end{lemma}

\begin{proof}
	Let  $\tilde{y}_1,\ldots,\tilde{y}_n$  be an independent and identically distributed   copy  of  $y_1,\ldots,y_n$, and  let  $\tilde{M}_{i}$  the version of  $\hat{M}_{i}$ corresponding to  $\tilde{y}_{1},\ldots,\tilde{y}_n$.
	Then,
	\begin{equation*}
		\label{eqn:m}
		\begin{array}{lll}
			\displaystyle  E\left(   \underset{v \in K}{\sup}\,\,  \sum_{i=1}^{n}   [  E\{\hat{M}_{i}(v_i) \}       - \hat{M}_{i}(v_i) ]     \right)  & = &\displaystyle E\left(   \underset{v \in K}{\sup}\,\,  \sum_{i=1}^{n}   [  E\{ \tilde{M}_{i}(v_i) \}   - \hat{M}_{i}(v_i) ]     \right). 
		\end{array}
	\end{equation*}
	Condition on $y_1,\ldots,y_n$  and let 
	\[
	\displaystyle X_v =   \sum_{i=1}^{n}   \left\{\tilde{M}_{i}(v_i) -    \hat{M}_{i}(v_i) )     \right\}  .
	\]
	Then
	\[
	\begin{array}{lll}
		\underset{v \in K}{\sup}\,\,   E_{   \tilde{y}_1,\ldots,\tilde{y}_n  | y_1,\ldots,y_n   }  X_v  \,\leq \, E_{   \tilde{y}_1,\ldots,\tilde{y}_n  | y_1,\ldots,y_n   } \underset{v \in K}{\sup}\,\,    X_v.
	\end{array}
	\]
	We  can  take the expected  value with respect to  $y_1,\ldots,y_n$  to get
	\[
	\begin{array}{lll}
		\displaystyle	E\left[ \underset{v \in K}{\sup}\,\,  \sum_{i=1}^{n}\left\{  M_i(v_i)   -\hat{M}_{i}(v_i)  \right\}  \right]&\leq &  \displaystyle  E \left[\underset{v \in K}{\sup}\,\, \sum_{i=1}^{n}   \left\{  \tilde{M}_{i}(v_i) -    \hat{M}_{i}(v_i) \right\}     \right]\\
		& = & \displaystyle E \left[\underset{v \in K}{\sup}\,\, \sum_{i=1}^{n}  \xi_i  \left\{  \tilde{M}_{i}(v_i) -    \hat{M}_{i}(v_i) \right\}      \right]\\
		& \leq&\displaystyle E \left\{\underset{v \in K}{\sup}\,\, \sum_{i=1}^{n}  \xi_i  \tilde{M}_{i}(v_i)      \right\}  \,+\,\\
		&&  \displaystyle E \left\{\underset{v \in K}{\sup}\,\, \sum_{i=1}^{n}  -\xi_i    \hat{M}_{i}(v_i)       \right\}\\
		& = &2 \displaystyle E \left\{\underset{v \in K}{\sup}\,\, \sum_{i=1}^{n}  \xi_i    \hat{M}_{i}(v_i)       \right\},\\
		%&   = & \displaystyle\mathbb{E} \left(\underset{v \in K}{\sup}\,\, \sum_{i=1}^{n} (  (  \tilde{M}_{i}(v) -   \mathbb{E}\tilde{M}_{i}(v)   )  +( \mathbb{E}\hat{M}_{i}(v)   - \hat{M}_{i}(v) ) )     \right)\\
		%& \leq &\displaystyle\mathbb{E} \left(\underset{v \in K}{\sup}\,\, \sum_{i=1}^{n}  (  \tilde{M}_{i}(v) -   \mathbb{E}\tilde{M}_{i}(v)   )      \right) + \\
		%& &\displaystyle\mathbb{E} \left(\underset{v \in K}{\sup}\,\, \sum_{i=1}^{n} (  \mathbb{E}\hat{M}_{i}(v)   - \hat{M}_{i}(v) )      \right)\\
		%& =& 2	 \displaystyle\mathbb{E} \left(\underset{v \in K}{\sup}\,\, \sum_{i=1}^{n} (  \mathbb{E}\hat{M}_{i}(v)   - \hat{M}_{i}(v) )      \right),
	\end{array}
	\]
	where the first  equality follows because  $( \xi_1  (  \tilde{M}_{n,1}(v_1) -    \hat{M}_{n,1}(v_1) ) , \ldots, \xi_n  (  \tilde{M}_{n,n}(v_n) -    \hat{M}_{n,n}(v_n) )  )$ and  $( \tilde{M}_{n,1}(v_1) -    \hat{M}_{n,1}(v_1)  ,\ldots, \tilde{M}_{n,n}(v_n) -    \hat{M}_{n,n}(v_n)  )$  have the same distribution. The second equality  follows because  $-\xi_1,\ldots,-\xi_n$ are also independent Rademacher variables.

\end{proof}

\begin{lemma}
	\label{lem4}
	%Suppose  that Assumption \ref{as1.2}
	%holds. Then
	(Contraction principle).	With the notation from before we  have that
	\[
	E\left\{  \underset{v \in K}{\sup}\,\, \sum_{i=1}^{n}   \xi_i   \hat{M}_{i} (v_i)  \right\}   \,\leq
	\, 2	RW\left(   K -   \theta^* \right) = 2	RW\left(   K \right).
	%  E\left\{    \underset{v \in K}{\sup}\,\, \sum_{i=1}^{n}   \xi_i (v_i -\theta_i^*) \right\}   .
	\]
\end{lemma}

\begin{proof}
	Recall that   $\hat{M}_{i} (v_i) =  \rho_{\tau}(y_i - v_i  )- \rho_{\tau}(y_i - \theta_i^*  )$. Clearly, these are 1-Lipschitz continuous functions. Therefore,  
	\[
	\begin{array}{lll}
		\displaystyle 	E\left\{  \underset{v \in K}{\sup}\,\, \sum_{i=1}^{n}   \xi_i   \hat{M}_{i} (v_i)  \right\}   &  = & \displaystyle 	 E\left(E\left\{  \underset{v \in K}{\sup}\,\, \sum_{i=1}^{n}   \xi_i  \hat{M}_{i} (v_i)      \bigg|  y\right\} \right)\\
		&\leq  &\displaystyle 	 E\left(E\left\{  \underset{v \in K}{\sup}\,\, \sum_{i=1}^{n}   \xi_i v_i      \bigg|  y\right\} \right)\\
		&= &\displaystyle 	E\left\{  \underset{v \in K}{\sup}\,\, \sum_{i=1}^{n}   \xi_i (v_i-\theta_i^*  )  \right\} +      E \left(  \sum_{i=1}^{n}   \xi_i \theta_i^*    \right)   \\
		& =&  \displaystyle  E\left\{    \underset{v \in K}{\sup}\,\, \sum_{i=1}^{n}   \xi_i (v_i -\theta_i^*) \right\} , 
	\end{array}
	\]
	where the inequality follows from, the contraction principle for Rademacher complexity, see  Theorem  4.12  in \cite{ledoux2013probability}.
	%Conditioning  on $y_1,\ldots,y_n$,  we notice  that  the  functions  $x \rightarrow M_i(\theta_i^*+x)$  are  1-Lipschitz continuous. The claim  follows  from  Theorem  4.12  in \cite{ledoux2013probability}.
\end{proof}

The following corollary can be used for proving upper bounds  for general constraint  estimators as in (\ref{eqn:constrained_estimator}) when the set $K$ is compact.

\begin{corollary}
	\label{thm1}
	With the notation from before, 
	\[
	%E	\left\{\Delta^2(  \hat{\theta} - \theta^* ) \right\}  \,\leq \, 
	E\left\{M( \hat{\theta}) \right\} \,\leq \,  2 \,RW\left(   K \right),
	%E\left\{    \underset{v \in K}{\sup}\,\, \sum_{i=1}^{n}   \xi_i  (v_i -\theta_i^*)  \right\}.
	\]
	where  the  right most inequality  holds for a general  set $K$.%, and the  left most  requires  that  Assumption A     holds.
	%	If in addition  (\ref{as2})   holds,  then
	%	\[
	%	E	\left\{\Delta^2(  \hat{\theta} - \theta^* ) \right\}  \,\leq \,  2\,RW\left(   K -   \theta^* \right).
	%	% E\left\{    \underset{v \in K}{\sup}\,\, \sum_{i=1}^{n}   \xi_i  (v_i -\theta_i^*)\right\}.
	%	\]
\end{corollary}

\begin{proof}
	This follows  from Lemmas \ref{lem1}--\ref{lem4}.
\end{proof}

\subsection{Proof of Proposition   \ref{prop:basic}  }

We now prove Proposition \ref{prop:basic} whose statement we now recall here.

%\begin{proposition}%\label{prop:basic}
\textbf{Proposition 1.}	\textit{Let $K \subset  R^n$ be a convex set. Let us define a function $\mathcal{M}: [0,\infty) \rightarrow  R$ as follows:
	$$\mathcal{M}(t) = RW[K \cap \{\theta: M(\theta) \leq t^2\}].$$ %Suppose the distributions of $y_1,\dots,y_n$ obey Assumption A. 
	Then the following inequality is true for any $t > 0$,}
\begin{equation*}
	\mathrm{pr}\{M(\hat{\theta}) > t^2\}\leq \frac{2\mathcal{M}(t)}{t^2}.
\end{equation*}
%	where $C$ is a constant that only depends on the distributions of $y_1,\dots,y_n$.
%\end{proposition}
\begin{proof}
	Suppose  that
	\begin{equation}
		\label{eqn:otherD}
		M(  \hat{\theta} ) > t^2.%=  c_1^2  r_n^2,  
		%n^{   -( 2k+2)/(2k+3)  },  
	\end{equation}
	
	First, notice that $M$ is continuous. To see this,  let  $\theta,\tilde{\theta} \in  R^n$. Then
	\[
	\begin{array}{lll}
		\vert  M(\theta) -M(\tilde{\theta}) \vert &=& \displaystyle \left\vert    \sum_{i=1}^{n}     E\left\{   \rho_{\tau}(y_i-   \theta_i)  - \rho_{\tau}(y_i-   \tilde{\theta}_i) \right\}  \right\vert\\
		& \leq &  \displaystyle \sum_{i=1}^{n}  E\left\{ \vert \rho_{\tau}(y_i-   \theta_i)  - \rho_{\tau}(y_i-   \tilde{\theta}_i)   \vert   \right\}\\
		& \leq& \displaystyle \sum_{i=1}^{n}    \vert \theta_i -\tilde{\theta}_i\vert, 
	\end{array}
	\]
	where the second inequality follows  by the fact that $\rho_{\tau}(\cdot)$ is a 1-Lipschitz function. Hence,  $M$ is continuous.

	Next,  let $q^2 :=   M( \hat{\theta} )$. Then  define  $g \,:\, [0,1] \rightarrow R $  as  $g(u) =    M\{ (1-u)\theta^*+  u \hat{\theta} \}$.  Clearly,   $g$ is a continuous function  with  $g(0) =0$, and  $g(1)   = q^2$. Therefore,  there exists   $u_{  \hat{\theta} } \in [0,1]$ such that  $g(u_{ \hat{\theta} }) = t^2$.  Hence, letting  $\tilde{ \theta} = (1-u_{  \hat{\theta} } )\theta^*   +u_{  \hat{\theta} } \hat{\theta}$  we observe that  by the convexity of $\hat{M}$ and the basic inequality
	\[
	\hat{M}(\tilde{ \theta}  ) = \hat{M}\{   (1- u_{  \hat{\theta} } )\theta^* +   u_{  \hat{\theta} } \hat{\theta}  \} \leq  (1- u_{  \hat{\theta} } ) \hat{M}(\theta^*) +   u_{  \hat{\theta} }  \hat{M}(\hat{\theta}) \leq 0.
	\]
	Furthermore,  $\tilde{\theta}\in K$  by convexity of  $K$, and  $M(\tilde{\theta}) = t^2$   by construction.   This implies  that
	\[
	\underset{ v \in K\,:\,  M(v )   \leq  t^2   }{\sup}\,\,M(v ) -\hat{M}(v)   \,\geq  \,  	     M(\tilde{ \theta} ) -\hat{M}(\tilde{ \theta}) \,\geq \,   M(\tilde{ \theta}) .%\, \geq\, c_0 t^2.
	\]
	Therefore, 
	\[
	\begin{array}{lll}
		\mathrm{pr}\left\{ M(   \hat{\theta}  )    >t^2\right\}  & \leq & \displaystyle  \mathrm{pr}\left\{ \underset{v \in K   \,:\,    M(v)    \leq  t^2  }{\sup}\,\,M(v) -\hat{M}(v)    \,\geq \,  t^2 \right\} \\%\geq  c_0 t^2   \right\}\\
		& \leq&    \displaystyle  \frac{1}{t^2}E\left\{ \underset{v \in K   \,:\,   M(v)    \leq  t^2  }{\sup}\,\,M(v) -\hat{M}(v)   \right\}\\
		& \leq& \displaystyle \frac{2}{t^2}  RW\left[ \{    v \in   K \,:\, M(v) \leq t^2     \}          \right]\\
		&=&  \displaystyle 2\mathcal{M}(t) /t^2 ,
	\end{array}
	\]
	where the second inequality follows from  Markov's inequality, and the third as in Lemmas  \ref{lem3} and \ref{lem4}.
	This  completes the proof.

\end{proof}

\subsection{Proof of Theorem \ref{thm:basic} and Associated Corollaries}

We start by recalling \textit{Assumption A}. %Recall that we denote by $F_{y_i}$  the cumulative distribution function of $y_i$. 

\textit{\noindent{Assumption A:}} There exists  a constant $L>0$ and $\underline{f}>0$ such that for any positive integer $n$ and any $\delta \in R^n$  satisfying  $\|\delta\|_{\infty} \leq L$  we have that 
\[
\,   \vert   F_{y_i}(\theta_i^* + \delta_i)  -F_{y_i}(\theta_i^*) \vert\,  \geq \,  \underline{f}\,  \vert \delta_i \vert,
\]
for all $i= 1,\ldots, n$,  where we recall that $F_{y_i}$ is the CDF of  $y_i$.

Theorem  \ref{thm:basic}  follows with the same argument in the proof of Proposition \ref{prop:basic}  combined with the following lemma.

%Before  stating our  next lemma  we recall Assumption \ref{as2}  from the paper.

%\begin{assumption} There exists  a constant  $L>0$  such that   for   $\delta \in R^n$  satisfying  $\|\delta\|_{\infty} \leq L$  we have that 
%	\[%
%	\underset{i=1,\ldots,n}{\min}\,    \vert   F_{y_i}(\theta_i^* + \delta_i)  -F_{y_i}(\theta_i^*) \vert\,  \geq \,  \underline{f}\,  \vert \delta_i \vert,
%\]
%for some  cinstant $\underline{f}>0$, and where   $F_{y_i}$ is the cumulative distribution function of  $y_i$.
%%%%	%$f_{y_i}$  is the probability density function  of $y_i$. %We write $  \theta^*  = F_{y_i}^{-1}(\tau) $, and   assume that  $\theta^* \in K$. 
%\end{assumption}

%We are now ready  to  prooceed with a lemma characterizing the behavior of the function $M$ arround $\theta^*$.

\begin{lemma}
	\label{lem2}
	Suppose  that Assumption  A holds. Then there exists  a constant   $c_0$ such that  for all  $\delta \in R^n$, we have
	\[
	\displaystyle  M(\theta^*+\delta) \geq   c_0 \Delta^2(\delta) .%   :=\sum_{i=1}^{n}    d(\delta_i)
	\]
	%	where 
	%	\begin{equation}
	%	content...
	%	\label{eqn:distance}
	%	d(x) =\begin{cases}
	%	c_0     x^2   & \text{if}    \,\,\, \vert x\vert \leq L,\\       
	%	c_0     \vert x\vert   & \text{if}    \,\,\, \vert x\vert >L,\\ 
	%	\end{cases} 
	%	\end{equation}
	%	for some constant  $c_0>0$.
	%  that scales like $\min\{   \underline{f}L,\underline{f}  \}$.
\end{lemma}

\begin{proof}
	First, we notice that by Equation B.3 in  \cite{belloni2011}, we have that
	\begin{equation}
		\label{eqn:b3}
		M_{i}(\theta_i^* +\delta_i ) - M_{i}(\theta_i^* ) \,=\,\displaystyle  \int_0^{\delta_i}    \left\{ F_{y_i}(\theta_i^*+z) -   F_{y_i}(\theta_i^*)   \right\}dz,\\
	\end{equation}
	for all $i = 1,\ldots,n$.    Hence,  supposing  that  $\vert\delta_i\vert\leq L$,  we obtain from Assumption  A that
	\[
	\begin{array}{lll}
		M_{i}(\theta_i^* +\delta_i ) - M_{i}(\theta_i^* ) \,\geq \,  \displaystyle\frac{\delta_i^2  \underline{f} }{2}.
		%&=&\displaystyle  \int_0^{\delta_i}    \left[ F_{y_i}(\theta_i^*+z) -   F_{y_i}(\theta_i^*)   \right]dz\\
		%& =&\displaystyle  \int_0^{\delta_i}      f_{y_i}(u(\theta_i^*,z))  z   dz\\
		%&\geq &\displaystyle\frac{\delta_i^2  \underline{f} }{2},
	\end{array}
	\]	
	%where  $u(\theta_i^*,z)$ is a point between  $\theta_i^* +z$ and $\theta_i^*$, and the inequality    follows  from Assumption  \ref{as2}. 
	
	Suppose  now that  $\delta_i> L$. Then by  (\ref{eqn:b3}), we obtain
	\[
	\begin{array}{lll}
		M_{i}(\theta_i^* +\delta_i ) - M_{i}(\theta_i^* ) %&=&\displaystyle  \int_0^{\delta_i}    \left\{ F_{y_i}(\theta_i^*+z) -   F_{y_i}(\theta_i^*)   \right\}dz\\
		& \geq &  \displaystyle  \int_{L/2}^{\delta_i}    \left\{ F_{y_i}(\theta_i^*+z) -   F_{y_i}(\theta_i^*)  \right\}dz\\
		&\geq &  \displaystyle  \int_{L/2}^{\delta_i}    \left\{ F_{y_i}(\theta_i^*+L/2) -   F_{y_i}(\theta_i^*)   \right\}dz\\
		&=&  \displaystyle  \left(\delta_i-\frac{L}{2}\right) \{F_{y_i}(\theta_i^*+L/2) -   F_{y_i}(\theta_i^*)\}   \\
		& \geq &  \frac{\delta_i}{2}  \frac{L \underline{f}}{2}\\
		& =: & \vert\delta_i\vert c_0,  
	\end{array}
	\]
	where the first  two inequalities follow because  $F_{y_i}$ is monotone, and the third inequality by  Assumption A.
	
	The case  $\delta_i <-L$ can be handled similarly. The conclusion follows combining the  three different cases.
\end{proof}

%Next we state  to conditions that   generalize Assumptions \ref{as1}--\ref{as2} in the paper.

%\begin{assumption}
%	\label{as1.2}
%	We write $  \theta_i^*  = F_{y_i}^{-1}(\tau) $, and   assume that  $\theta^* \in K$. Here  $F_{y_i}$  is cumulative distribution function of $y_i$ for  $i=1,\ldots,n$. We require   $y_1,\ldots,y_n$ to be independent.
%\end{assumption}

%\begin{assumption}
%	\label{as2.2}
%	There exists  a constant  $L$  such that   for   $\delta \in R^n$  satisfying  $\|\delta\|_{\infty} \leq L$  we have that%
%	\[
%	\underset{i = 1,\ldots,n}{\min}\,\,f_{y_i}(  \theta_i^*  +\delta_i ) \geq \underline{f},
%	\]
%	for some  $\underline{f}>0$, and where   $f_{y_i}$  is the probability density function of $y_i$.
%We write $  \theta^*  = F_{y_i}^{-1}(\tau) $, and   assume that  $\theta^* \in K$.
%\end{assumption}

%\newpage

%\section{General lemmas}\label{sec-comp}

\subsubsection{ Corollary  \ref{cor:basic}  }
\label{sec:cor1}

\begin{corollary}
	\label{cor:basic}
	Consider  the notation from  Theorem \ref{thm:basic}. If  $\{r_n\}$  is  a sequence such that 
	\begin{equation}
		\label{cor:as}
		\underset{t \to  \infty  }{\lim}\,   \underset{n \geq 1}{\sup} \,\frac{  \mathcal{R}( t  r_n  n^{1/2} )  }{  t^2   r_n^2 n}   \,=\, 0,
	\end{equation}
	then
	\[
	\frac{1}{n}\Delta^2(  \hat{\theta}_K - \theta^* ) \,=\,O_{  \mathrm{pr} }\left(r_n^2\right).
	\]
\end{corollary}

\begin{proof}
	Let  $\epsilon>0$ be given.  Notice that  for any $c_1>0$ we have that 
	\[
	\begin{array}{lll}
		\mathrm{pr}\left\{  \frac{1}{n}\Delta^2(  \hat{\theta}_K - \theta^* )     >  c_1^2 r_n^2 \right\}   & \leq&  		   C \frac{\mathcal{R}(  c_1 r_n  n^{1/2} )}{c_1 ^2 r_n^2  n}\\ 
		& <& \epsilon,
		%P(\Delta^2(\hat{\theta}_{K} - \theta^*) > t^2) \leq C \frac{\mathcal{R}(t)}{t^2}.
	\end{array}
	\]
	where the first inequality holds by Theorem \ref{thm:basic} and the last  by choosing  $c_1$  large enough exploiting (\ref{cor:as}).
	%}
\end{proof}

\subsubsection{Corollary   \ref{cor:risk} }
\label{sec:cor2}

\begin{corollary}\label{cor:risk}
	Let $K \subset  R^n$ be a convex set. Suppose the distributions of $y_1,\dots,y_n$ obey Assumption A. Then the following expectation bound holds:
	\begin{equation*}
		E \{\Delta^2(\hat{\theta}_{K} - \theta^*)\} \leq C\:RW(K)
	\end{equation*}
	where $C$ is a constant that only depends on the distributions of $y_1,\dots,y_n$. 
\end{corollary}

\begin{proof}
	This follows by   combining  Theorem \ref{thm1} with  Lemma \ref{lem2}.
\end{proof}
%\begin{corollary}
%\label{cor1}
%Under  Assumption A, if  $\theta^* \in K$ we have that 
%\begin{equation}
%		\label{eqn:lower_bound}
%	E\left\{\Delta^2(  \hat{\theta} -\theta^* )\right\}   \,\leq \,  E\left[\underset{v \in K}{\sup}\,\,\left\{M(v) - \hat{M}(v)\right\}\right]. 
%\end{equation}
%\end{corollary}
%\newpage

%\newpage

%{\color{red} Make subsubsections for the corollaries?}

% See Lemma  \ref{lem22} for a characterization of $\mathcal{R}  ^{\perp}$.
% Notice

%Furthermore, we recall that  $\mathrm{TV}^{(r)}(\theta) = n^{r-1} \|D^{(r)} \theta\|_1$  for  $\theta\in R^n$.

% that  when $r =1$,  we have that $\mathcal{R}^{\perp}  = \text{span}\{   \boldsymbol{1} \}$, where  $\boldsymbol{1} = (1,\ldots,1)^{\top} \in R^n$.  

\section{Proof of Theorem~\ref{thm5} and Theorem~\ref{thm2}}

We first provide a sketch of our proofs for the sake of convenience of the reader. This sketch is meant to convey the overall proof structure. 
\subsection{Proof sketch of Theorem~\ref{thm5} and Theorem~\ref{thm2} }\label{sec:sketch}
%{\color{red} Merge the above here}.

As discussed in Section \ref{sec:ideas},  we must upper bound the quantity 
\[
RW\big[K_{V} - \theta^* \:\:\cap \{\delta: \Delta^2(\delta) \leq t^2\}\big]
\]
for $t>0$. Here, $K_{V} =    \{   \theta    \,:\, \|D^{(r) }  \theta \|_1   \leq   V n^{1-r} \}$,  where $V^* =  n^{r-1} \|D^{(r) }  \theta^* \|_1$ and  $V \geq  V^*$ when we are proving Theorem~\ref{thm5} and $V = V^*$ when we are proving Theorem~\ref{thm2}. This differs from the usual least squares setting where the quantity of interest is 
\[
RW\big[K_{V} - \theta^* \:\:\cap \{\delta: \|\delta\|^2 \leq t^2\}\big].
\]
To proceed in the proof of Theorems \ref{thm5} and  \ref{thm2}, we start by writing
\[
RW\big[K_{V} - \theta^* \:\:\cap \{\delta: \Delta^2(\delta) \leq t^2\}\big] \leq  T_1  + T_2,
\]
where 
\[
T_1 \,:=\,  \,E\left\{   \underset{  \delta \in   K_{V} - \theta^*\,:\, \Delta^2(\delta) \leq t^2     }{\sup} \,\,  \xi^{\top} P_{  \mathcal{R}^{\perp}}\delta    \right\}
\]
and 
\[
T_2   \,:=\,   \,E\left\{   \underset{  \delta \in   K_{V} - \theta^*\,:\, \Delta^2(\delta) \leq t^2     }{\sup} \,\,  \xi^{\top} P_{  \mathcal{R}}\delta    \right\}.
\]
where $\mathcal{R}$ is the subspace spanned by the rows of $D^{(r)}$ (recall its definition fon Page 2 in the paper), $\mathcal{R}  ^{\perp}$ is the orthogonal  complement of  $\mathcal{R}  $ and $P_{  \mathcal{R}},P_{  \mathcal{R}^{\perp}}$ are the orthogonal projection matrices for the corresponding subspaces. %We will show that $T_2$ is the dominating term an % Thus,

Next we consider different  steps.

\textbf{Step 1}: Bounding  $T_1$. We attain this by writing
%\[
\[
\displaystyle  P_{\mathcal{R}^{\perp}} \delta  =  \sum_{j=1}^{r}  (\delta^{\top} v_j )v_j,
\]
where $\{v_1,\dots,v_r\}$ form an orthonormal basis of  $\mathcal{R}^{\perp}$. We then  upper bound  $\vert \xi^{\top}  P_{\mathcal{R}^{\perp}} \delta  \vert$ using Lemmas \ref{lem23}  and  \ref{lem5}, exploiting the fact that  $\delta \in   K - \theta^*$ and $ \Delta^2(\delta) \leq t^2  $ as in the definition of  $T_1$.

%lem23
%lem5
%\]
\textbf{Step 2}: Bounding  $T_2$. The bound for $T_2$ is going to be the leading order term. The key observation we use here is Lemma \ref{lem18}
which states that
\begin{equation}
	\label{eqn:aux1} 	\|  \delta\|^2   \,\leq \,  \max\{ \| \delta\|_{\infty},1  \} \Delta^2( \delta),
\end{equation}
for all $\delta \in R^n$. Then Lemmas \ref{lem5} and \ref{lem6}  provide upper bounds on $\|P_{   \mathcal{R}  }\delta  \|_{\infty} $ and  $\Delta^2(   P_{   \mathcal{R}  }\delta  )$ respetively. This together with (\ref{eqn:aux1})  leads to 
\[
T_2 \,\leq \,  \,E\left(   \underset{  \delta \in   K_{V} - \theta^*\,:\, \|\delta\| \leq \tilde{t}     }{\sup} \,\,  \xi^{\top} P_{  \mathcal{R}}\delta    \right)
\]
for some  $\tilde{t}$ that is of the same order of magnitude as $t$.

\textbf{Step 3}:

We obtain that the bound on $T_1$ is a lower order term. Hence, the main task is to get good bounds on $T_2.$

\begin{itemize}
	\item Proof of Theorem~\ref{thm5}: The bound for $T_2$ given in the last display is exactly the local Gaussian Width of the set $K_{V} - \theta^*$ for $V \geq V^*$ and a bound for this local Gaussian width is available in Lemma  B.1  from \cite{guntuboyina2020adaptive}. 
	
	\item {Proof of Theorem~\ref{thm2}:} The main difference with Theorem \ref{thm5} is in the way we handle $T_2$.
	Similarly as in \textbf{Step 2}, we show that for $V = V^*$
	\[
	T_2  \,\leq \, \displaystyle  E\left\{  \underset{  \delta    \,: \,   K- P_{   \mathcal{R}  }\theta^*,\,\, \|\delta \|\leq  \tilde{t}       }{\sup} \,\,     \xi^{\top} \delta    \right\}
	\leq \,E\left\{   \underset{  \delta   \in T_{K_{V} }( P_{   \mathcal{R}  }\theta^*), \,\,   \|\delta\|\leq \tilde{t}}{\sup} \,\,     \xi^{\top} \delta    \right\}\]
	for some  $\tilde{t}$ that is of the same order of magnitude as $t$ and $T_{K_{V}}(P_{   \mathcal{R}  }\theta^*)$ is the tangent cone at $P_{   \mathcal{R}  }\theta^*$ with respect to the convex set $K_V$; see~\eqref{eqn:cone} for the precise definition. The Gaussian width of such a tangent cone is again available in Appendix B.2 in \cite{guntuboyina2020adaptive} and we directly employ this result to finish the proof. %{\color{red} Cite the lemma}
\end{itemize}

%The bound for $T_2$ given in the last display is of the 

%\textbf{Theorem \ref{thm5}.}  This is obtained by arriving at  an upper bound to $T_2$ based on \textbf{Step 2} and   Lemma  B.1  from \cite{guntuboyina2020adaptive}. 

\subsection{Proofs of Theorem~\ref{thm5} and Theorem~\ref{thm2}}
We now start our formal proofs. We first state some lemmas that we will require. 

\subsubsection{Intermediate results required for Proofs of Theorem  \ref{thm5} and Theorem~\ref{thm2}}

\begin{lemma}
	\label{lem22}	
	It  holds that
	\[
	\mathcal{R}^{\perp}  \,=\,   \Pi:=   \mathrm{Span}\left\{   v \in   R^{n} \,:\,   v_i =  p(i/n), \,\text{for}\,\,\, i = 1,\ldots,n,\,\,\, \text{and} \,\,p(\cdot) \, \,\text{a polynomial of degree  at most }\,\,r-1  \right\}.
	\]
\end{lemma}
\begin{proof}
	First, notice that 	$\Pi$  equals to the column space of  the matrix
	\[
	\left(  \begin{matrix}
		1   &    \frac{1}{n }  &   \cdots  &   \left( \frac{1}{n}\right)^{r-1}\\
		1  &    \frac{2}{n }  &   \cdots  &   \left( \frac{2}{n}\right)^{r-1}\\
		\vdots &\vdots &   \vdots   &  \vdots\\
		1  &    1 &   \cdots  &  1\\
	\end{matrix} \right)  \in R^{n \times r},
	\]
	which is a Vandermonde matrix  of  rank  $r $, if  $r< n$.   Furthermore, $v \in \Pi$  implies that $D^{(r)} v =0$, which holds by an iterative application of the mean value theorem  and the fact that the  $r$th  derivative of a polynomial of  is  constant and equals to  $0$. Therefore,  $\Pi \subset  \mathcal{R}^{\perp} $   and  $\text{dim}(\mathcal{R}^{\perp} ) =  \text{dim}(\Pi)$.  Hence, the claim follows.
\end{proof}
\begin{lemma}
	\label{lem18}
	Let  $\delta   \in   R^n $. Then 
	\begin{equation}
		\label{eqn:ine}
		\|  \delta\|^2   \,\leq \,  \max\{ \| \delta\|_{\infty},1  \} \Delta^2( \delta).
	\end{equation}
\end{lemma}
\begin{proof}
	We notice that
	\[
	\begin{array}{lll}
		\|\delta\|^2 &=& \sum_{i \,:\,  \vert  \delta_i \vert \leq 1 } \vert \delta_i \vert^2   +   \sum_{i \,:\,  \vert  \delta_i \vert > 1 } \vert \delta_i \vert^2  \\
		&\leq &  \sum_{i \,:\,  \vert  \delta_i \vert \leq 1 } \vert \delta_i \vert^2   +   \|\delta \|_{\infty}\sum_{i \,:\,  \vert  \delta_i \vert > 1 } \vert \delta_i \vert \\
		&\leq&  \max\{ \|\delta\|_{\infty},1\} \left( \sum_{i \,:\,  \vert  \delta_i \vert \leq 1 } \vert \delta_i \vert^2   +   \sum_{i \,:\,  \vert  \delta_i \vert > 1 } \vert \delta_i \vert\right) \\
		& =& \max\{ \|\delta\|_{\infty},1\}\Delta^2( \delta).
	\end{array}
	\]
	%	This follows  immediately from the definition of $\Delta^2(\cdot)$.
\end{proof}

\begin{lemma}
	\label{lem19}
	Let  $v \in \mathcal{R}^{\perp}$ such that  $\|v\| = 1$. Then
	\[ 
	\|v\|_{\infty} \leq  \frac{b_r}{n^{1/2} },	    
	\]
	for a positive constant  $b_r$ that only depends on $r$.
\end{lemma}
\begin{proof}
	Let $\{q_m\}_{m=0}^{r-1} $ be  the  normalized Legendre polynomials  of degree at most $r-1$ which have domain in $[-1,1]$ and   satisfy \[
	\displaystyle \int_{-1}^{1} q_m(x) q_{  m^{\prime} }(x) dx  \,=\,    1_{  \{ m =   m^{\prime} \}  },
	\]
	%	where  $q_m$  has 
	%Let  $q_m (x) = p_m(x)\cdot ( m+1/2   )^{1/2}    $. 
	Next notice that, by Lemma \ref{lem22}, $v$  can be written as
	\[
	v_i   =  \sum_{j=0}^{r-1} a_jq_j(x_i),
	\]
	where  $x_i =  -1+2i/n$ for $i=1,\ldots,n$, and where  $a_0,\ldots,a_{r-1} \in R$. Let  $g \,:\, R \rightarrow R$ be defined as
	\[
	g(x)   =  \sum_{j=0}^{r-1} a_jq_j(x).
	\]
	The notice that for  $A_1 = (-1,x_1)$, and $A_i =  (x_{i-1},x_i) $ for all $i >1$, we have that  
	\[
	\begin{array}{lll}
		\displaystyle       \left \vert   \sum_{j=0}^{r-1} a_j^2 - \frac{2}{n}  \right \vert  &  = & \displaystyle\left \vert   \int_{-1}^1  [g(x)]^2dx  - \frac{2}{n} \sum_{i=1}^{n}  [g(x_i) ]^2 \right \vert\\
		& \leq&  \displaystyle   \sum_{i=1}^{n}    \int_{A_i} \left \vert     [g(x)]^2 -   [g(x_i) ]^2 \right \vert dx\\
		& \leq&  \displaystyle   \sum_{i=1}^{n}    \int_{A_i}  \| [ g^2]^{\prime}\|_{\infty}\left \vert    x -   x_i \right \vert dx\\
		& \leq& \displaystyle \frac{4}{n} \| [ g^2]^{\prime}\|_{\infty}.
	\end{array}	
	\]
	However,
	\[
	\begin{array}{lll}
		\{[g(x)]^2\}^{\prime}  &\,=\,& \displaystyle  \left\{  \sum_{j=0}^{r-1}   a_j^2  [q_j(x)]^2    +     \sum_{j \neq  j^{\prime} }   a_j a_{ j^{\prime} }    q_j(x) q_{j^{\prime}   }(x)  \right\}^{\prime}.
	\end{array}
	\]
	Therefore,
	\[
	\left \vert   \sum_{j=0}^{r-1} a_j^2 - \frac{2}{n}  \right \vert   \,\leq \,    \frac{   c_r \|a\|_{\infty}^2  }{n}   \leq       \frac{   c_r  }{n} \sum_{j=0}^{r-1} a_j^2,
	\]
	for some constant $c_r>0$   that only depends on $r$.  Hence, for large enough $n$,
	\[
	\sum_{j=0}^{r-1} a_j^2     \,\leq \,    \frac{\tilde{c}_r   }{n},
	\]
	for a constant  $\tilde{c}_r>0$ that depends on $r$.  As a result
	\[
	%	 \begin{array}{lll}
	\|v\|_{\infty} \,\leq\,  \underset{i=1,\ldots,n}{\max} \,  \sum_{j=0}^{r-1}  \vert q_j(x_i) \vert  \vert a_j\vert  \,\leq \,   \|  a\|_{\infty}    \underset{x \in [-1,1] }{\max} \, \sum_{j=0}^{r-1}  \vert q_j(x) \vert  \,\leq \,      \frac{   \tilde{c}
		_r^{1/2} }{n^{1/2}  }  \underset{x \in [-1,1] }{\max} \, \sum_{j=0}^{r-1}  \vert q_j(x) \vert, 
	%	 \end{array}
	\]
	and the claim follows.
\end{proof}

\begin{lemma}
	\label{lem20}
	
	If   $\delta \in R^n$  and  $\mathrm{TV}^{(r )}(\delta)  \leq   V $, then
	\[
	\|  P_{  \mathcal{R} } \delta \|_{\infty}   \leq    \tilde{C}_r \frac{V}{n^{r-1}},  
	\]
	for a constant  $\tilde{C}_r >1$  that depends on $r$.
\end{lemma}

\begin{proof}
	Let  $M :=  \left\{D^{(r)} \right\}^{ +}   \in  R^{   n \times (n-r) }$ be the Moore–Penrose inverse of  $D^{(r)}$. 
	%	Then \[
	%\|M\|_{\infty} =	\underset{i =1,\ldots, n,\,\,\,j = 1,\ldots,n-r }{\max}\,  \vert  M_{i,j}\vert     \,=\, O(1).
	%\]
	First, we notice  that by  Lemma  13 in \cite{wang2016trend}, we have that $M =  P_{  \mathcal{R} } H_2/(r-1)! $  where  $H_2$  consists of  the last $n-r$ columns of the $(r-1)$th
	order falling factorial basis matrix. Here, as in \cite{wang2014falling},  we have that for  $i \in \{1,\ldots,n\}$ and  $j \in  \{ 1,\ldots,n-r\}$,
	\[
	(H_2)_{i,j}   =  h_j(i/n),
	\]
	where 
	\[
	h_j(x) =   \prod_{l=1}^{r-1}   \left(   x -  \frac{j+l}{n} \right)1_{  \left\{ x\geq  \frac{j+r-1}{n} \right\}   }.
	\]
	%	Next we define $G\in R^{n  \times (n-r)}$ such that  for  $i \in \{1,\ldots,n\}$ and  $j \in \{1,\ldots,n-r\}$, we have 
	%	\[
	%	G_{i,j} = g_j(i/n),
	%	\]
	%	where
	%	\[
	%	g_j(x) \,=\, (x-t_j)^{r-1}1_{ \left\{    x\geq  t_j  \right\}  },
	%	\]
	%	and with
	%	\[
	%	  \left\{   t_1,\ldots,t_{n-r} \right\} \,=\, \begin{cases}
	%	 \left\{     \frac{  2+(r-1)/2  }{n},\ldots,\frac{n-(r-1)/2 }{n}     \right\}  &     \text{if} \,\,r \,\,\text{is  odd},\\
	%	  \left\{     \frac{  2+r/2  }{n},\ldots,\frac{n-r/2 }{n}     \right\}&     \text{if} \,\,r \,\,\text{is  even}.
	%	  \end{cases}
	%	\]
	%	Then,  by Lemma  4 in   \cite{wang2014falling}, we  have that
	%   \begin{equation}
	%	\label{eqn:failling_basis}
	%    \underset{i=,1\ldots,n,\, j= 1,\ldots,n-r  }{\max}\,\,\left\vert  G_{i,j} - (H_2)_{i,j}  \right\vert \leq  \frac{(r-1)^2}{n}.
	% \end{equation}
	Then for  $e_i $ an element of the canonical basis in  $R^{n-r}$  we have that
	\[
	\begin{array}{lll}
		\|  e_i^{\top} M\|_{\infty}  & \leq&     \| P_{    \mathcal{R}  } e_i\|_{1} \|    H_2\|_{\infty}/(r-1)!\\ 
		& \leq&  \left(  \|e_i \|_{1} +     \| P_{    \mathcal{R}^{\perp}  }e_i \|_{1}  \right)    \| H_2\|_{\infty}/(r-1)!  \\
		& \leq& \left[  1  +     \| P_{    \mathcal{R}^{\perp}  }e_i \|_{1}  \right]  /(r-1)!  \\
		%       & \leq&\left[  1  +     \| P_{    \mathcal{R}^{\perp}  }e_i \|_{1}  \right]  /(r-1)! 
	\end{array}
	\]
	where the  first  inequality follows from H\"{o}lder's inequality, the second from the triangle inequality and the last by   the definition of  $H_2$. 
	
	Next let  $v_1,\ldots, v_{r}$ be an  orthonormal basis of  $P_{\mathcal{R}^{\perp} }$. Then
	\[
	\displaystyle 	 \| P_{    \mathcal{R}^{\perp}  }e_i \|_{1}    \,=\,   \left\|  \sum_{j=1}^{r}    (e_i^{\top} v_j) v_j \right\|_1 \,\leq\, \sum_{j=1}^{r}    \vert   (e_i^{\top} v_j)   \vert \|v_j\|_1  \,\leq \,  \sum_{j=1}^{r}    \|  v_j  \|_{\infty} \|v_j\|_1 \,\leq\,  \sum_{j=1}^{r}    \|  v_j  \|_{\infty} n^{1/2}.
	\]
	Hence,  from Lemma  \ref{lem19} we obtain that
	%	Then 
	\begin{equation}
		\label{eqn:pseudo_inv}
		\|M\|_{\infty} =	\underset{i =1,\ldots, n,\,\,\,j = 1,\ldots,n-r }{\max}\,  \vert  M_{i,j}\vert     \,=\, O(1).
	\end{equation}
	
	Finally,  if   $\delta \in R^n$  and  $\mathrm{TV}^{(r )}(\delta)  \leq   V $, then
	\[
	\| P_{  \mathcal{R} }\delta\|_{\infty}\,=\,  \|   \{ D^{(r)}\}^{  +} D^{(r) } \delta   \|_{\infty}\,\leq\,    \|   M\|_{\infty} \|D^{(r) } \delta   \|_{1} =   O\left\{   \|D^{(r) } \delta   \|_{1}    \right\}\,
	\]
	where  the first  inequality  follows from H\"{o}lder's inequality, and the last from  (\ref{eqn:pseudo_inv}). The claim follows.
	%\[
	% M =  P_{R}
	%\]
\end{proof}

\begin{lemma}
	\label{lem23}
	Let  $\delta \in  R^n$  and  $v \in \mathcal{R}^{\perp}$ such that  $\|v\|=1$. Then
	\[
	\vert  \delta^{\top} v  \vert \,\leq \,  \frac{b_r}{n^{ 1/2  } }  \Delta^2(\delta)  +    \Delta(\delta),
	\]
	where  $b_r >0$ is the constant from Lemma \ref{lem19}.
\end{lemma}

\begin{proof}
	Notice that 
	\[
	\begin{array}{lll}
		\vert \delta^{\top} v\vert     & \leq  &   \displaystyle  \sum_{i=1}^{n }  \vert\delta_i \vert   \vert v_{i}\vert \\
		& = &   \displaystyle  \sum_{i=1}^{n }  \vert\delta_i \vert   \vert v_{i}\vert   1_{    \{\vert  \delta_i \vert >1\}   }   \,+\,  \sum_{i=1}^{n }  \vert\delta_i \vert   \vert v_{i}\vert   1_{    \{\vert  \delta_i \vert \leq 1\}   }  \\
		& \leq &  \displaystyle   \|v\|_{\infty}    \sum_{i=1}^{n }  \vert\delta_i \vert     1_{    \{\vert  \delta_i \vert >1\}   }   \,+\,   \|v\|  \left(\sum_{i=1}^{n }  \delta_i^2     1_{    \{\vert  \delta_i \vert \leq 1\}   } \right)^{ 1/2  }\\
		& \leq &\frac{b_r}{n^{ 1/2  } }\Delta^2(\delta)  +    \Delta(\delta),
	\end{array}
	\]
	where  the first inequality follows from  the triangle inequality, the second from  H\"{o}lder and  Cauchy–Schwarz inequalities, and the last  by the definition of $\Delta^2(\cdot)$ and Lemma  \ref{lem19}. The claim follows.%  combining  (\ref{eqn:ineq})  with (\ref{eqn:ineq2}).
\end{proof}

\begin{lemma} 
	\label{lem5}
	Let  $\delta   \in R^n$ with  $\Delta^2(\delta) \leq  t^2$. Then
	\[
	\|  P_{\mathcal{R}^{\perp}     }\delta\|_{\infty}       \,\leq \,    \gamma(t,n) :=     \tilde{b}_r \left(      \frac{t}{  n^{  1/2  } } +   \frac{t^2}{n}  \right),
	\]
	where $\tilde{b}_r>0 $ depends on $r$ only.
\end{lemma}

\begin{proof}
	Let  $v_1,\ldots,v_{r}$  an  orthonormal  basis     of   $\mathcal{R}^{\perp}$. Then
	\[
	\displaystyle  P_{\mathcal{R}^{\perp}} \delta  =  \sum_{j=1}^{r}  (\delta^{\top} v_j )v_j.
	\]
	Hence,
	\begin{equation}
		\label{eqn:ineq}
		\displaystyle \vert   (P_{\mathcal{R}^{\perp}} \delta)_i   \vert \,\leq \,   r \left(\underset{j =1 ,\ldots,r}{\max}\,\|v_j\|_{\infty}\right) \left( \underset{j =1 ,\ldots,r}{\max}\,  \vert \delta^{\top} v_j\vert \right)\,\leq \,  r \frac{b_r}{n^{1/2} }\left( \underset{j =1 ,\ldots,r}{\max}\,  \vert \delta^{\top} v_j\vert \right),
	\end{equation}
	where  the last  inequality  follows from  Lemma \ref{lem19}. 	Now, for  $j \in \{1,\ldots,r\}$, we have by Lemma \ref{lem23} that 
	\begin{equation}
		\label{eqn:ineq2}
		\begin{array}{lll}
			\vert \delta^{\top} v_j\vert     & \leq  &   %\displaystyle  \sum_{i=1}^{n }  \vert\delta_i \vert   \vert v_{j,i}\vert \\
			%& = &   \displaystyle  \sum_{i=1}^{n }  \vert\delta_i \vert   \vert v_{j,i}\vert   1_{    \{\vert  \delta_i \vert >1\}   }   \,+\,  \sum_{i=1}^{n }  \vert\delta_i \vert   \vert v_{j,i}\vert   1_{    \{\vert  \delta_i \vert \leq 1\}   }  \\
			%& \leq &  \displaystyle   \|v_j\|_{\infty}    \sum_{i=1}^{n }  \vert\delta_i \vert     1_{    \{\vert  \delta_i \vert >1\}   }   \,+\,   \|v_j\|  \left(\sum_{i=1}^{n }  \delta_i^2     1_{    \{\vert  \delta_i \vert \leq 1\}   } \right)^{ 1/2  }\\
			%& \leq &
			\frac{b_r}{n^{ 1/2  } }t^2  +   t.
		\end{array}
	\end{equation}
	The claim follows  combining  (\ref{eqn:ineq})  with (\ref{eqn:ineq2}).
\end{proof}

\begin{proposition}
	\label{lem9}
	Under Assumption A we have that
	\[
	\begin{array}{lll}
		%\displaystyle E\left\{    \underset{   \delta    \,: \,   \|D^{(r)} \delta\|_1\leq  \frac{2  V^*}{n^{r-1}  }\,\,   \Delta^2(\delta )\leq  t^2       }{\sup} \,\,    \sum_{i=1}^{n}   \xi_i \delta_i    \right\} 
		RW\left[\left\{\delta    \,: \,    \mathrm{TV}^{(r)}(\delta)\leq  2  V, \,\,   \Delta^2(\delta )\leq  t^2 \right\}  \right]	&\leq &  C_r  \{m(t,n)\}^{1-1/(2r)}(  n^{1/2} V)^{1/(2r)} + a(t,n)
		%m(t,n)  \left\{  \frac{  n^{ 1/2  }   V }{ m(t,n) } \right\}^{  1/(2r)    }  +  a(t,n)
		%C_r \left\{  \frac{t^2 }{n^{ 1/2   }}  +   t\right\} +   C_r  m(t,n)  \left\{  \frac{  n^{ 1/2  }   V }{ m(t,n) } \right\}^{  1/(2r)    }  \\ %1/(2+2k)
		%& &  + C_r  m(t,n)  \{\log ( en)\}^{ 1/2  },
	\end{array}
	\]
	where $a(t,n)$ is a lower order term defined as 
	\[
	a(t,n): = C_r \left\{  \frac{t^2 }{n^{ 1/2   }}  +   t\right\}   + C_r  m(t,n)  \{\log ( en)\}^{ 1/2  },
	\]
	and 
	\[
	m(t,n)  := 
	\tilde{c}_r\max\left\{  \left(V/n^{r-1} \right)^{  1/2    } ,1  \right\}   \left(\left[1+ \left\{\gamma(t,n)\right\}^{ 1/2   } \right]t   +  \frac{ t^2 }{ n^{  1/2 } } \right),
	\]
	for some positive  constants   $C_r, \tilde{c}_r$.%, and  with $\xi_1,\ldots,\xi_n$ independent   Rademacher  random variables.
\end{proposition}

\begin{remark}
	For the choice of $t$ that we make within the proof of Theorem~\ref{thm1}, the term $m(t,n)$ is $\Theta(t)$ and hence the reader can safely think of $m(t,n)$ in the right hand side above as $t.$ 
\end{remark}

\begin{proof}
	First,  we observe that

	\begin{equation}
		\label{eqn:t1_t2}
		\begin{array}{lll}
			\displaystyle E\left\{  \underset{v \in K   \,:\,    \Delta^2(v-\theta^*)    \leq  t^2  }{\sup}\,\,      \sum_{i=1}^{n}   \xi_i  (v_i -\theta_i^*) \right\}  &\leq &   \displaystyle E\left\{   \underset{  \delta    \,: \,   \mathrm{TV}^{(r)}(\delta)\leq  2V,\,\,   \Delta^2(\delta )\leq  t^2        }{\sup} \,\,  \xi^{\top} P_{ \mathcal{R}^{\perp}}\delta    \right\}  \\
			& & + \displaystyle E\left\{    \underset{  \delta    \,: \,  \mathrm{TV}^{(r)}(\delta)\leq  2V,\,\,   \Delta^2(\delta )\leq  t^2        }{\sup} \,\,     \xi^{\top} P_{ \mathcal{R}  }\delta    \right\}\\
			&  =:& T_1 +T_2.
		\end{array}
	\end{equation}
	Hence, we proceed to bound $T_1$ and  $T_2$.  
	
	%An  upper bound for  $T_1$ is given in the following lemma.
	
	\textbf{Bounding $T_1$.}
	
	Let 
	$v_1,\ldots,v_{r}$  an orthonormal basis of $\mathcal{R}^{\perp}$. Then  by Lemma \ref{lem19},  it holds that   $\| v_j\|_{\infty}  \leq   b_r/n^{  1/2 }$, for  $j = 1,\ldots,r$.   Hence,  for any $\delta \in R^{n}$  with  $\Delta^2(\delta) \leq  t^2$,
	\begin{equation}
		\label{eqn:calculation}
		\begin{array}{lll}
			\xi^{\top} P_{  \mathcal{R}^{\perp}}\delta   &  \leq &  \displaystyle    \left\vert \sum_{j=1}^{r}        \delta^{\top }v_j   \cdot   \xi^{\top} v_j          \right\vert \\
			& \leq & \displaystyle    \sum_{j=1}^{r}        \left\vert \delta^{\top }v_j   \right\vert  \cdot    \left\vert\xi^{\top} v_j          \right\vert \\
			& \leq &  \displaystyle     r\left( \underset{j = 1,\ldots, r}{\max}  \vert   \xi^{\top} v_j    \vert   \right)\left(\underset{j = 1,\ldots, r}{\max}  \vert   \delta^{\top} v_j    \vert   \right)\\
			& \leq& \displaystyle   r \left( \underset{j = 1,\ldots, r}{\max}  \vert   \xi^{\top} v_j    \vert   \right)\left(    \frac{b_rt^2 }{n^{  1/2 }} +   t\right)
		\end{array}
	\end{equation}
	where the last inequality   follows from Lemma \ref{lem23}.  Therefore,
	\begin{equation}\label{eqn:t1}
		\displaystyle   T_1 \,\leq \, r \left(    \frac{b_rt^2 }{n^{  1/2 }}  +   t\right) \sum_{j=1}^{r}   \,E\left(  \vert    \xi^{\top} v_j  \vert    \right)  \,\leq  \, C_r \left(\frac{t^2}{n^{  1/2  }}  +   t\right), 
	\end{equation}
	for  some positive  constant  $C_r>0$, and  where the last inequality  follows  since $\xi^{\top} v_j    $  are sub-Gaussian  random variables with variance 1.

	\textbf{Bounding $T_2$.}
	
	We now proceed to bound $T_2$. 	Towards that end  we  first prove a lemma.

	\begin{lemma}
		\label{lem6}
		
		Let  $\delta  \in R^n$  with  $\Delta^2(\delta)\leq  t^2 $. Then,
		\[
		\Delta^2(   P_{   \mathcal{R}  }\delta  ) \leq  h(t,n)  :=  c_r \left\{   t^2 +   t^2 \gamma(t,n)+   n \left(  \frac{t^2}{n} +  \frac{t^4}{n^2} \right)    \right\},
		% 2\left\{   t^2 +   2t^2 \gamma(t,n)+   16n r^2\left(  \frac{t^2}{n} +  \frac{t^4}{n^2} \right)    \right\},
		%2\max\{L,L^2\}\left(1+ 2 \gamma(t,n)   +  16(k+1)^2 \right)t^2   +  \frac{16(k+1)^2  t^4 }{n}.
		\]
		%	where
		%	\[
		%	  m(t,n)  := 
		%	\tilde{c}_r\max\left\{  \left(V/n^{r-1} \right)^{  1/2    } ,1  \right\}   \left(\left[1+ \left\{\gamma(t,n)\right\}^{ 1/2   } \right]t   +  \frac{ t^2 }{ n^{  1/2 } } \right),
		%	\]
		with  $\gamma(t,n)$  as in Lemma \ref{lem5},  and for  some constant  $c_r>0$.
	\end{lemma}
	
	\begin{proof}
		Set  $\tilde{\delta}   = P_{\mathcal{R} }\delta  $. By Lemma  \ref{lem5}  we have that $\|\tilde{\delta} -\delta\|_{\infty} \leq \gamma(t,n)$. Also,
		\[
		\begin{array}{lll}
			\Delta^2(\tilde{\delta})   &  = &\displaystyle  \sum_{i=1}^{n }  \min\{ \vert  \tilde{ \delta}_i\vert, \delta_i^2 \}\\
			%\vert\tilde{\delta}_i \vert     1_{    \{\vert  \tilde{\delta}_i \vert >1\}   }   \,+\,      \sum_{i=1}^{n }  \tilde{\delta}_i^2     1_{    \{\vert  \tilde{\delta}_i \vert \leq 1\}      }\\
			%	&  \leq  &\displaystyle  \sum_{i=1}^{n }  \vert\tilde{\delta}_i \vert     1_{    \{\vert  \delta_i \vert >1   +2 \gamma(t,n) \}   }   \,+\,      \sum_{i=1}^{n }  \tilde{\delta}_i^2     1_{    \{\vert  \delta_i \vert \leq 1-2 \gamma(t,n) \}      }    \\
			%	& & \displaystyle +   \sum_{i=1}^{n }   \min\left\{ \vert\tilde{\delta}_i \vert,    \tilde{\delta}_i^2\right\}     1_{    \{\vert  \delta_i \vert \in    (1-2 \gamma(t,n),1+2 \gamma(t,n)) \}   }  \\
			& \leq& \displaystyle  \sum_{i=1}^{n } \vert\tilde{\delta}_i \vert     1_{    \{\vert  \delta_i \vert >1\}   }   \,+\,      \sum_{i=1}^{n }  \tilde{\delta}_i^2 1_{    \{\vert  \delta_i \vert \leq 1\}      },\\
		\end{array}
		\]
		and so 
		\[
		\begin{array}{lll}
			\Delta^2(\tilde{\delta})   &  \leq  &\displaystyle  \sum_{i=1}^{n }   \left\{\vert\delta_i \vert  + \gamma(t,n)\right\} 1_{    \{\vert  \delta_i \vert >1\}   }   \,+\,      \sum_{i=1}^{n }  \left[2\delta_i^2   +2\{\gamma(t,n) \}^2  \right] 1_{    \{\vert  \delta_i \vert \leq 1\}      }\\
			& \leq&   \left[   2t^2   +  2t^2 \gamma(t,n)+  2 n \{\gamma(t,n)\}^2   \right]\\
			&\leq& 2\left\{   t^2 +   2t^2 \gamma(t,n)+   4n \tilde{b}_r^2\left(   \frac{t^2}{n} +  \frac{t^4}{n^2} \right)    \right\},\\
			% & =&
		\end{array}
		\]
		where the second inequality follows form the fact that  
		\[
		\vert \left\{  i\,:\,    \vert  \delta_i\vert >1  \right\}\vert   \,\leq \,  t^2.
		\]
	\end{proof}
	%we observe  that
	%	Notice that    
	%	\[
	%	\begin{array}{lll}
	%	\displaystyle E\left\{   \underset{  \delta    \,: \,   \|D^{(r)}\delta \|_1\leq  \frac{2  V^*}{n^k},\,\,   \Delta^2(\delta )\leq  t^2        }{\sup} \,\,    \sum_{i=1}^{n}   \xi_i \delta_i    \right\} &\leq &   \displaystyle E\left\{   \underset{  \delta    \,: \,   \%|D^{(r)}\delta \|_1\leq  \frac{2  V^*}{n^k},\,\,   \Delta^2(\delta )\leq  t^2        }{\sup} \,\,  \xi^{\top} P_{ \mathcal{R}^{\perp}}\delta    \right\}  \\
	%& & + \displaystyle E\left\{    \underset{  \delta    \,: \,   \|D^{(r)}\delta \|_1\leq  \frac{2  V^*}{n^k}\,\,   \Delta^2(\delta )\leq  t^2        }{\sup} \,\,     \xi^{\top} P_{ \mathcal{R}  }\delta    \right\}\\
	%	&  =:& T_1 +T_2.
	%	\end{array}
	%	\]
	%	We now proceed to  bound   $T_1$  and  $T_2$.  
	Next, let  $\delta   \in R^n$,  $\tilde{\delta}=    P_{  \mathcal{R} }\delta $ and suppose that $\Delta^2(\delta)\leq  t^2$, and  $ \mathrm{TV}^{(r)}(\delta) \leq 2V$.  Then from Lemmas  \ref{lem18}, \ref{lem20}  and  \ref{lem6}, we obtain  that 
	\begin{equation}
		\label{eqn:embedding}
		\begin{array}{l}\|\tilde{\delta} \| \leq      m(t,n)  :=  \\
			\tilde{c}_r\max\left\{  \left(V/n^{r-1} \right)^{  1/2    } ,1  \right\}   \left(\left[1+ \left\{\gamma(t,n)\right\}^{ 1/2   } \right]t   +  \frac{ t^2 }{ n^{  1/2 } } \right),
		\end{array}
	\end{equation}
	for  a  positive  constant $\tilde{c}_r$  that depends on $r$. As a result  from (\ref{eqn:t1_t2}), we obtain	
	\begin{equation}
		\label{eqn:t2}
		T_2  \,\leq \,  \displaystyle  E\left\{  \underset{  \delta    \,: \,   \mathrm{TV}(\delta)\leq  2 V,\,\, \|\delta \|\leq  m(t,n)        }{\sup} \,\,     \xi^{\top} \delta    \right\}.\\
	\end{equation}
	Therefore,  by  Lemma  B.1  from \cite{guntuboyina2020adaptive} and    Lemma \ref{radamacher_width}, 
	%(\ref{eqn:Gaussian_width}),
	\begin{equation}
		\label{eqn:t22}
		T_2  \leq    C_r  m(t,n)  \left\{  \frac{n^{1/2}  V }{ m(t,n) } \right\}^{   1/(2r)   }   
		+ C_r  m(t,n)  \{\log ( en)\}^{  1/2 }.
	\end{equation}
	for  a positive constant  $C_r$  that depends on $r$.
	The conclusion follows.
\end{proof}

%\newpage

%Combining Lemmas \ref{lem18}--\ref{lem7}  we arrive to the following corollary which provides the embedding in Equation (\ref{eqn:embedding_tv}).

%\begin{lemma}
%	\label{lem8}
%	Let  $\delta   \in R^n$,  $\tilde{\delta}=    P_{  \mathcal{R} }\delta $ and suppose that $\Delta^2(\delta)\leq  t^2$, and  $\|D^{(r)} \delta\|_1 \leq    \frac{ 2V^* }{n^{r-1  }}  $. Then, with the notation from Lemmas \ref{lem18}--\ref{lem7}, 
%	\[
%	\begin{array}{l}\|\tilde{\delta} \| \leq      m(t,n)  :=  \\
%	\max\left\{  \left(\tilde{C}_r V^* \right)^{  1/2    } ,1  \right\}   2^{ 1/2  }   \left(\left[1+ \left\{2\gamma(t,n)\right\}^{ 1/2   } +  4r \right]t   +  \frac{4r  t^2 }{ n^{  1/2 } } \right).
%	\end{array}
%	\]
%\end{lemma}

%We are now ready to show the embedding in (\ref{eqn:embedding_tv}).
%\newpage

%	\begin{proof}
%	Notice  that
%	\[
%	\begin{array}{lll}
%	\| \tilde{\delta} \|^2   &  = & \displaystyle  \sum_{i=1 }^{n}   \tilde{\delta}_i^2\\
%	&\leq  & \displaystyle  \sum_{i=1 }^{n}   \tilde{\delta}_i^2   1_{   \{  \vert \tilde{\delta_i }\vert  \leq L  \}  }       +   \sum_{i=1 }^{n}   \tilde{\delta}_i^2   1_{   \{  \vert\tilde{ \delta}_i \vert  >L  \}  }   \\
%	& \leq&     \max\{  \|\tilde{\delta}\|_{\infty}   ,1   \} \Delta^2 (\tilde{\delta}),
%	\end{array}
%	\]
%	and the claim follows  from Lemmas \ref{lem6}--\ref{lem7}.
%\end{proof}

\subsubsection{Proof of Theorem \ref{thm5}}
Finally, we present the proof of Theorem \ref{thm5}.
\begin{proof}
	
	This  follows immediately from Proposition \ref{lem9}  and Corollary \ref{cor:basic}   by setting  
	\[
	r_n\asymp  n^{  -  r/(2r+1)}   V^{1/(2r+1)} \max\left\{   1,        \left(  \frac{V }{n^{r-1}}\right)^{(2r-1)/(4r+2)}  \right\}. 
	\]
	%	with $c_1>0$ large enough.
	%We  start  by  defining $\tilde{D}^2(\delta)$, for  $\delta \in R^n$, as
	
\end{proof}

\iffalse

\subsubsection{ Proof of Theorem~\ref{thm2} }
We now present the proof of Theorem \ref{thm2}. Our strategy for proving Theorem \ref{thm2} is similar to that in the proof of Theorem \ref{thm5}. The main difference is that we focus on bounding 
\[
\displaystyle RW( \tilde{K}_t+\theta^*  )
% E\left\{   \underset{v \in \tilde{K}_t }{\sup}\,\,(M -\hat{M})(v+\theta^*)\right\} 
\]
where
%	Let
\begin{equation}
	\label{eqn:k_tilde}
	\tilde{K}_t =  \left\{ \delta   \in R^n \,:\,  \delta =   a(v-\theta^*),  \,a\in [0,1],\,v \in K, \,   \Delta^2(\delta) \leq   t^2   \right\},
\end{equation}
and the  tangent cone of $K$ is  defined as
\[
T_K(\theta^*) =    \text{Closure}\left\{   \delta   \in R^n\,:\, \delta =  a(v-\theta^*),\,\,   v\in K,\,\,  a\geq 0       \right\}.
\]\fi

\subsubsection{Proof of  Theorem ~\ref{thm2}  }
We now present our proof of Theorem ~\ref{thm2}. Throughout  we  write
\[
K   =\left\{   \theta   \in  R^n \,:\,  \|D^{(r)} \theta\|_1 \leq  \frac{V^*}{n^{r-1} }  \right\} .
\]
Notice that to arrive at the conclusion of Theorem \ref{thm2},  by Theorem \ref{thm:basic}, it is enough to bound

$$RW\left( \{    \delta \in   K - \theta^*\,:\, \Delta^2(\delta) \leq t^2     \}          \right), $$
for $t<cn^{1/2}$  where  $c>0$ is a constant.

However,
\[
\begin{array}{lll}
	RW\left( \{    \delta \in   K - \theta^*\,:\, \Delta^2(\delta) \leq t^2     \}          \right)  & \leq &	\displaystyle \,E\left\{   \underset{  \delta \in   K - \theta^*\,:\, \Delta^2(\delta) \leq t^2     }{\sup} \,\,  \xi^{\top} P_{  \mathcal{R}^{\perp}}\delta    \right\}   \\
	& & + \displaystyle \,E\left\{   \underset{ \delta \in   K - \theta^*\,:\, \Delta^2(\delta) \leq t^2      }{\sup} \,\,     \xi^{\top} P_{   \mathcal{R}  }\delta    \right\} \\
	&  =:& T_1 +T_2.
\end{array}
\]
Then, $T_1$  can be bounded with the same argument that $T_1$  was  bounded in the proof of Proposition \ref{lem9}.  To  control $T_2$, 
we define the   tangent cone of $\theta \in K$ as
\begin{equation}
	\label{eqn:cone}
	T_K(\theta) =    \text{Closure}\left\{   \delta   \in R^n\,:\, \delta =  a(v-\theta),\,\,   v\in K,\,\,  a\geq 0       \right\},
\end{equation}
and notice that as in Equation (\ref{eqn:t2}),
\begin{equation}
	\label{eqn:t2_2}
	\begin{array}{lll}
		T_2  &\leq & \displaystyle  E\left\{  \underset{  \delta    \,: \,   K- P_{   \mathcal{R}  }\theta^*,\,\, \|\delta \|\leq  m(t,n)        }{\sup} \,\,     \xi^{\top} \delta    \right\}.\\
		& \,\leq \, & \displaystyle \,E\left\{   \underset{  \delta   \in T_K( P_{   \mathcal{R}  }\theta^*), \,\,   \|\delta\|\leq  \max\{1,(V^*/n^{r-1})^{1/2} \}t  c(r)   }{\sup} \,\,     \xi^{\top} \delta    \right\}\\
		&\, =\,& \max\left\{1 ,\left(\frac{V^*}{n^{r-1} }\right)^{1/2} \right\}t  c(r)\,E\left\{   \underset{  \delta   \in T_K( P_{   \mathcal{R}  }\theta^*), \,\,   \|\delta\|\leq  1   }{\sup} \,\,     \xi^{\top} \delta    \right\}    ,\\
		& \leq&  \max\left\{1,\left(\frac{V^*}{n^{r-1} }\right)^{1/2} \right\}t  c(r)\, c_{r}   \left\{(s+1)     \log\left(  \frac{en}{s+1} \right) \right\}^{1/2}
	\end{array}
\end{equation}
for some positive constant  $c_r$,	where the last   inequality holds by Appendix   B.2 in \cite{guntuboyina2020adaptive}    and Lemma \ref{radamacher_width}.

\section{Proof of Theorem~ \ref{thm4} }

Throughout this section we will use $C$ as a generic positive  constant  that can change from line to line. Furthermore, for  an appropriate  $\lambda$ to be chosen later
we write
\[
\hat{\theta} \,=\,\underset{\theta \in  R^n}{\arg \,\min}\,\,\sum_{i=1}^{n} \rho_{\tau}(y_i-\theta_i)  \,+\, \lambda \|D^{(r)} \theta\|_1,
\]
%for some  $\lambda>0$.

\subsection{Proof outline}

We now provide a high level overview of the proof of Theorem \ref{thm4}.%, which  proceeds using ideas similar to the proof of  Theorem  \ref{thm:basic}. 

\textbf{Step 1.}  We  show in Proposition \ref{lem17} that for any given $\epsilon > 0$ if  $\tilde{\theta}$ is in the line segment between $\hat{\theta}$ and $\theta^*$, then  $\tilde{\delta} = \tilde{\theta} - \theta^{*}$ belongs to  a restricted set  $\mathcal{A}$ (depending on $\epsilon$) with probability  at least  $1 - \epsilon/4$. We call this event $\Omega_1$ and this restricted set is of the form 
\[
\mathcal{A} = \{  \delta \,  :\,   \mathrm{TV}^{(r)}(\delta)  \leq  C V^* + \text{some extra terms}          \}
\]
for some positive constant  $C$, see the precise definition in (\ref{eqn:restricted}).  Then we show, using the convexity of  $\hat{M}$, the optimality of $\hat{\theta}$, and  Lemma \ref{lem2}   that  for any $t>0$ it holds that
\[
\{ \Delta^2(\hat{\delta}) \geq  t^2     \} \subset   \left\{      \underset{ \delta \in \mathcal{A},\,\,\Delta^2(\delta)\leq  t^2  }{\sup} \,\,\left[ M(\theta^* +\delta )- \hat{M}(\theta^* +\delta) +   \lambda \|D^{(r)}\theta^*\|_1 - \lambda\|D^{(r)}(\theta^* +\delta)\| \right]\geq   c_0 t^2   \right\},
\]
where  $c_0>0$  is as in Lemma \ref{lem2}. This step uses ideas very similar to the proof of Theorem  \ref{thm:basic}.

\textbf{Step 2.}  We define another high probability event $\Omega_2$  as in (\ref{eqn:high_prob_events}). Then  based on Proposition \ref{lem17} and Lemma \ref{lem24},  we obtain that $\Omega_1 \cap \Omega_2$  happens with probability at least $1-\epsilon/2$.  Hence, we do our analysis conditioning on  $\Omega_1\cap \Omega_2$. We start with also assuming that 
$\Delta^2(\hat{\delta}) \geq  t^2 $  for  some $t>0$ (whose value is to be specified later).% with the goal of arriving at a contradiction.

\textbf{Step 3.}  We show that if $\delta \in \mathcal{A},\,\,\Delta^2(\delta)\leq  t^2  $ and     $\Omega_1 \cap \Omega_2$ holds   then 
\[%
%\|D^{(r)} \delta \|_1 \leq \frac{  C  }{n^{r-1}}
\mathrm{TV}^{(r)}(\delta )   \leq  C
\]
for some $C>0$.  It then follows from  Steps $1$ and $2$ above that we can
reduce our focus to upper bounding  the probability  of the event 
\[
\left\{      \underset{ \delta \in K  }{\sup} \,\,   \left[M(\theta^* +\delta )- \hat{M}(\theta^* +\delta) +   \lambda \|D^{(r)}\theta^*\|_1 - \lambda\|D^{(r)}(\theta^* +\delta)\| \right]\geq      c_0 t^2   \right\}
\]
where
\[
K \,:=\,    \left\{ \delta \,:\,  \mathrm{TV}^{(r)}(\delta )   \leq  C,\,\,\Delta^2(\delta) \leq  t^2 \right\}.
\]

\textbf{Step 4.}   Next we observe  that
\[
\underset{ \delta \in K  }{\sup} \,\{\,M(\theta^* +\delta )- \hat{M}(\theta^* +\delta) +   \lambda \|D^{(r)}\theta^*\|_1 - \lambda\|D^{(r)}(\theta^* +\delta)\| \}  \,\leq \,       \underset{ \delta \in K  }{\sup} \,\,\{ M(\theta^* +\delta )- \hat{M}(\theta^* +\delta) \}\,+\, \lambda \|D^{(r)}\theta^*\|_1.
\]
Hence,  to show that $\mathrm{pr}( \Delta^2(\hat{\delta}) \geq  t) \leq \epsilon$, from Step 3  and an application of Markov's inequality, it suffices to show that  
\[
\frac{1}{c_0 t^2} E \left[   \underset{ \delta \in K  }{\sup} \,\,\{ M(\theta^* +\delta )- \hat{M}(\theta^* +\delta) \}   \right]  +    \frac{  \lambda  \|D^{(r)}\theta^*\|_1  }{c_0 t^2} \leq \epsilon
\]

\textbf{Step 5.}  Setting \[
t  \,\asymp \, 	n^{  1/(4r+2) }  \left(  \log n \right)^{  1/(4r+2) }
\] 

we now proceed to show that for some positive constant $c$
\[ 
\frac{1}{c_0t^2} E \left[   \underset{ \delta \in K  }{\sup} \,\,\{ M(\theta^* +\delta )- \hat{M}(\theta^* +\delta) \}   \right] \leq c\epsilon 
\]
exactly similarly as in the proof of Theorem \ref{thm5}. Now by setting  $\lambda$ to satisfy
\[
\lambda \, \asymp\, n^{  (2r-1)/(2r+1)} \left( \log n \right)^{ 1/(2r+1) }   \|D^{(r)} \theta^*\|_1^{ - (2r-1)/(2r+1) },
\]
see  (\ref{eqn:lambda}),  we can also verify that 
\[
\frac{\lambda  \|D^{(r)}\theta^*\|_1  }{c_0 t^2} \leq c\epsilon
\]
and conclude the proof.

\iffalse
\newpage

\textbf{Step 2.}  
First we construct to events  $\Omega_1$ and $\Omega_2$, based on Lemmas \ref{lem24}--\ref{lem17},  that happen with probability at least $1-\epsilon/2$. Then we assume that  $\Omega_1\cap \Omega_2$  holds and proceed by contradiction assuming that 
$\Delta^2(\hat{\delta}) \geq  t $  for a certain  $t>0$.

\textbf{Step 3.} 

\textbf{Step 4.}  We show that if  $\delta$ satisfies $\delta \in \mathcal{A},\,\,\Delta^2(\delta)\leq  t^2  $ and     $\Omega_1 \cap \Omega_2$ holds   then 
\[
\|D^{(r)} \delta \|_1 \leq \frac{  C  }{n^{r-1}}
\]
for some $C>0$.  It then follows from  \textbf{Step 3}  that we can 
reduce our focus to upper bound the probability  of the event 
\[
\left\{      \underset{ \delta \in K  }{\sup} \,\,M(\theta^* +\delta )- \hat{M}(\theta^* +\delta) +   \lambda \|D^{(r)}\theta^*\|_1 - \|D^{(r)}(\theta^* +\delta)\| \geq   \frac{t^2}{c_0}   \right\}
\]
where
\[
K \,:=\,    \left\{ \delta \,:\,   \|D^{(r)} \delta \|_1 \leq \frac{  C  }{n^{r-1}}\right\}.
\]
\textbf{Step 5.}   Next we observe  that
\[
\underset{ \delta \in K  }{\sup} \,\{\,M(\theta^* +\delta )- \hat{M}(\theta^* +\delta) +   \lambda \|D^{(r)}\theta^*\|_1 - \|D^{(r)}(\theta^* +\delta)\| \}  \,\leq \,       \underset{ \delta \in K  }{\sup} \,\,\{ M(\theta^* +\delta )- \hat{M}(\theta^* +\delta) \}\,+\, \lambda \|D^{(r)}\theta^*\|_1
\]
\textbf{Step 6.} Finally, we proceed to bound 
\[
\underset{ \delta \in K  }{\sup} \,\,M(\theta^* +\delta )- \hat{M}(\theta^* +\delta)
\]
as in the proof of Theorem \ref{thm5}.\fi

\subsection{Restricted set for  Proof of    Theorem  \ref{thm4}  (Step 1)}

%Throughout we assume that  Assumption  \ref{as2}  holds   and write
%\begin{equation}
%\label{eqn:k_set}
%K  =   \left\{    \theta \,:\,    \| D^{(r)}  \theta \|_1    \leq   \frac{V^*}{n^{r-1}}   \right  \}.
%\end{equation}
%Also,  for  $\delta \in R^n$  we write  $\Delta(\delta )   =  \{\Delta^2(\delta)\}^{1/2}$ with  $\Delta^2(\cdot)$  defined as in   Lemma  \ref{lem2}. Furthermore,  we use the notation  $M$  and  $\hat{M}$  from Definition \ref{def1}.

%Before 

\begin{proposition}
	\label{lem17}
	%Suppose that  $\|D^{(r)}\theta^*\|_1  =  O(n^{1-r})$.  
	Let $\epsilon \in (0,1)$  and   $u_i     =  \tau  -  1\{  y_i \leq \theta_i^* \}$  for  $i=1,\ldots,n$.
	Then	there exists positive constants $C_{\epsilon}, \tilde{C}_{\epsilon} , a_{\epsilon} $ only depending on $\epsilon$ such that if we set 
	$$\lambda =   C_{\epsilon} n^{  (2r-1)/(2r+1)} \left( \log n \right)^{ 1/(2r+1) }   \|D^{(r)} \theta^*\|_1^{ - (2r-1)/(2r+1) },  $$
	then with probability   at least  $1-\epsilon/4$,
	\[
	\kappa	(\hat{\theta}  -  \theta^*)  \,\in \,\mathcal{A},\,\,\,\,\,\forall \kappa \in [0,1],%\,:=\, \left\{ \delta \,:\,  \|D^{(r)} \delta \|  \,\leq \, C_0  \max\left\{ \frac{1}{n^k},   \gamma_1  \Delta^2 (\delta), \| D^{(r) } \delta\|_1   +   A^{-1}(a^*)^{\top} P_{  R^{\perp} }  (\tilde{\theta} - \theta^*)\\ \right\} \right\}, \,\,\,\,\,%\forall \kappa \in [0,1],
	%  \left[\gamma_1 D(\delta)    +  A^{-1} (a^*)^{\top} P_{  R^{\perp} }\delta+    \|D^{(r)}\theta^*\|_1  \right]    + n^{-k} \right\}, \,\,\,\,\,\forall \kappa \in [0,1],
	\]
	%	for all $\kappa \in [0,1]$  such that 
	% \begin{equation}
	%\label{eqn:cond2}
	%\Delta^2\left[  P_{\mathcal{R}}\{	\kappa	(\hat{\theta}  -  \theta^*)  \}\right] \,\leq\, 	\frac{A^2}{B^2},
	%\end{equation}
	%and where
	where
	\begin{equation}
		\label{eqn:restricted}
		\mathcal{A}\,:=\, \left\{ \delta \,:\,  \|D^{(r)} \delta \|_1  \,\leq \, \tilde{C}_{\epsilon}  \max\left\{ \frac{V^*}{n^{r-1} },   \gamma  \Delta^2 (  P_{\mathcal{R}}  \delta), \gamma ^{1/2} \Delta(  P_{\mathcal{R}}  \delta), \frac{V^*}{n^{r-1} }  +   A^{-1} u^{\top} P_{  \mathcal{R}^{\perp} }  \delta\right\} \right\},
	\end{equation}
	where 
	$$\gamma :=   \frac{  B^2  }{A^2}  ,$$
	%  A^{-1} B    \frac{  \max\{L,\sqrt{L}\}  }{ \min\{L,\sqrt{L}\} }, $$  
	$B  = a_{\epsilon}\tilde{B}$,    with
	$A$  given by $$A = C n^{  (2r-1)/(2r+1) } \left( \log n \right)^{ 1/(2r+1) }   \|D^{(r)} \theta^*\|_1^{ - (2r-1)/(2r+1) },   $$  for any fixed large enough constant $C>1$, and  
	% and $\tilde{B}$ given as 
	\[
	\tilde{B}\,=\,n^{ (2r-1)/(4r +2)  }\left(  \log n \right)^{ 1/(4r+2)  }  \| D^{(r)} \theta^* \|_1^{  2/(4r+2)   }.
	\]
	
	%	$n^{ (2r-1)/(4r +2)  }\left(  \log n \right)^{ 1/(4r+2)  }  \| D^{(r)} \theta^* \|_1^{  2/(4r+2)   }$
	% defined as in Lemma \ref{lem15}, 
	%	and with  . 
\end{proposition}

\subsubsection{Auxiliary lemmas for proof of Proposition \ref{lem17}   }

First we state a result which was  proven in the proof of Corllary 7 from \cite{wang2016trend}.

\begin{lemma}
	\label{lem24}
	\textbf{\citep[Corollary 7 in ][]{wang2016trend}. }
	There  exists    $A$  satisfying  $$A = C n^{  (2r-1)/(2r+1) } \left( \log n \right)^{ 1/(2r+1) }   \|D^{(r)} \theta^*\|_1^{ - (2r-1)/(2r+1) },   $$  for any  constant $C>1$  large enough such that
	\[
	\underset{x  \in   \mathrm{row}\{  D^{(r)} \}  \,:\,      \|D^{(r) } x\|_1 \leq 1}{\sup}\,\,\frac{ u^{\top}  x -  A  }{  \|x\| }  =O_{\mathrm{pr}} \left(   \tilde{B}  \right),
	\]
	where  $u=  (u_1,\ldots,u_n)^{\top}$  is  a vector with independent coordinates satisfying $u_i   \sim \mathrm{subGaussian}(\sigma^2)$ for $i=1,\ldots,n$, with $\sigma>0 $ a constant 
	%	\begin{equation}
	%	\label{eqn:radem}
	%	\mathrm{pr}(\varepsilon_i = \tau) = 1-\tau,  \,\,\,\,\,\,  \mathrm{pr}(\varepsilon_i = \tau-1) = \tau,  \,\,\,\,\,\,\, \text{for}  \,\,\,i =1,\ldots,n,
	%	\end{equation}
	and
	% $\tilde{B}$  satisfies
	\[
	\tilde{B} =  n^{ (2r-1)/(4r +2)  }\left(  \log n \right)^{ 1/(4r+2)  }  \| D^{(r)} \theta^* \|_1^{  2/(4r+2)   }.
	%\|\Delta^{(k+1) \theta^*\|_1^{  \frac{2}{4k +6} }. 
	\]
\end{lemma}

\begin{lemma}
	\label{lem15}
	With the notation from Lemma \ref{lem24}, we have that
	%	There  exists    $A$  satisfying  $$A = C n^{  (2r-1)/(2r+1) } \left( \log n \right)^{ 1/(2r+1) }   \|D^{(r)} \theta^*\|_1^{ - (2r-1)/(2r+1) }   $$  for any  $C>1$  large enough such that
	\[
	\underset{x  \in   \mathrm{row}\{  D^{(r)} \}  \,:\,      \|D^{(r) } x\|_1 \leq 1}{\sup}\,\,\frac{  u^{\top}  x -  A    }{  \Delta(x)  }  =O_{\mathrm{pr}} \left(   \tilde{B}  \right),
	\]
	where  $u=  (u_1,\ldots,u_n)^{\top}$  is  a vector with independent coordinates satisfying %$\varepsilon_i   \sim \mathrm{SubGaussian}(\sigma^2)$ for $i=1,\ldots,n$, with $\sigma $ a constant 
	\begin{equation}
		\label{eqn:radem}
		\mathrm{pr}(u_i = \tau) = 1-\tau,  \,\,\,\,\,\,  \mathrm{pr}(u_i = \tau-1) = \tau,  \,\,\,\,\,\,\, \text{for}  \,\,\,i =1,\ldots,n.
	\end{equation}
	%	where  $\epsilon =  (\epsilon_1,\ldots,\epsilon_n)^{\top}$  is  a vector with independent coordinates satisfying
	%   \begin{equation}
	%	\label{eqn:radem}
	%	  \mathrm{pr}(\epsilon_i = \tau) = 1-\tau,  \,\,\,\,\,\,  \mathrm{pr}(a_i = \tau-1) = \tau,  \,\,\,\,\,\,\, \text{for}  \,\,\,i =1,\ldots,n,
	% \end{equation}
	% and $\tilde{B}$  satisfies
	%	\[
	%\tilde{B} \asymp   n^{ (2r-1)/(4r +2)  }\left(  \log n \right)^{ 1/(4r+2)  }  \| D^{(r)} \theta^* \|_1^{  2/(4r+2)   }.
	%\|\Delta^{(k+1) \theta^*\|_1^{  \frac{2}{4k +6} }. 
	%	\]
	%	%and  $\phi_n $ is any sequence  satisfying $\phi_n \rightarrow \infty$.
	
\end{lemma}

\begin{proof}
	Let  $x  $   be such that  $x  \in \mathrm{row}\{D^{(r)}\}$  and $  \|D^{(r)} x\|_1 \leq  1$, then by  Lemmas  \ref{lem18} and   \ref{lem20}  there exists a constant  $\tilde{C}_r >0$ independent of  $x$  such that
	%	\[
	%	\|x\|_{\infty} \leq   \tilde{C}_r   V^*  n^{1-r}.  
	%	\] 
	%	Then by Lemma \ref{lem18}
	\[
	\begin{array}{lll}
		\Delta(x)    %&  =  &   ( \sum_{i  \,:\,  \vert x_i  \vert >1  }  \vert x_i \vert  +   \sum_{i  \,:\,  \vert x_i  \vert \leq 1  }  \vert x_i \vert^2  )^{1/2} \\   
		&\geq&   \tilde{C}_r^{-1/2}   \|x\|.
	\end{array}
	\]
	Hence,
	\[
	\begin{array}{l}
		\displaystyle  \underset{x  \in   \mathrm{row}\{ D^{(r)} \} \,:\,      \|D^{(r) } x\|_1 \leq  1 }{\sup}\,\,\frac{  u^{\top}  x - A  }{  \Delta(x)  }   \\
		\leq \displaystyle    \tilde{C}_r ^{1/2}  \underset{x  \in   \mathrm{row}\{ D^{(r)} \}  \,:\,      \|D^{(r) } x\|_1 \leq  1  }{\sup}\,\,\frac{  u^{\top}  x -  A  }{  \|x\|  }  \\
		%	 = \displaystyle  \underset{x  \in   \mathrm{row}\{  D^{(r)} \}  \,:\,      \|D^{(r) } x\|_1 \leq  n^{1-r}   V^* }{\sup}\,\,\frac{  u^{\top}  (n^{r-1} x) -  A  V^*  }{  \|n^{r-1}  x\|  }  \\
		% =  \displaystyle  \underset{x  \in   \mathrm{row}\{  D^{(r)} \}  \,:\,      \|D^{(r) } x\|_1 \leq 1  }{\sup}\,\,\frac{  u^{\top}  x -  A  }{  \| x\|  }  ,
	\end{array}
	\]
	and the claim  follows   by  Lemma \ref{lem24}.
\end{proof}

\begin{lemma}
	\label{lem16}
	Let   $u =  (u_1,\ldots,u_n)^{\top}$  is  a vector with independent   coordinates satisfying (\ref{eqn:radem}). Recall that $\mathcal{R}  =  \mathrm{row}\{D^{(r)}\} $  and $\mathcal{R}^{\perp}$ denote its  orthogonal complement. Then
	\[
	\underset{x  \in R^n  }{\sup}\,\,\frac{  u^{\top}   P_{  \mathcal{R}^{ \perp } } x   }{   \frac{\Delta^2(x)}{n} +   \Delta(x)     }  = O_{\mathrm{pr}}(1),
	\]
	where  $P_{ \mathcal{R}^{ \perp } }$ denotes the orthogonal projection onto  $\mathcal{R}^{\perp}$.%, and with $\phi_n $ any sequence  satisfying $\phi_n \rightarrow \infty$.
\end{lemma}
\begin{proof}
	%This follows immediately as (\ref{eqn:calculation}).
	Let 
	$v_1,\ldots,v_{r}$  an orthonormal basis of $\mathcal{R}^{\perp}$. Then  proceeding as in Equation (\ref{eqn:calculation}),
	%by Lemma \ref{lem19},  it holds that   $\| v_j\|_{\infty}  \leq   b_r/n^{  1/2 }$, for  $j = 1,\ldots,r$.   Hence,  for any $\delta \in R^{n}$  with  $\Delta^2(\delta) \leq  t^2$,
	\begin{equation}
		\label{eqn:calculation2}
		\begin{array}{lll}
			u^{\top} P_{  \mathcal{R}^{\perp}}\delta %  &  \leq &  %\displaystyle    \left\vert \sum_{j=1}^{r}        \delta^{\top }v_j   \cdot   \xi^{\top} v_j          \right\vert \\
			%	& \leq & \displaystyle    \sum_{j=1}^{r}        \left\vert \delta^{\top }v_j   \right\vert  \cdot    \left\vert\xi^{\top} v_j          \right\vert \\
			& \leq &  \displaystyle     r\left( \underset{j = 1,\ldots, r}{\max}  \vert   u^{\top} v_j    \vert   \right)\left(\underset{j = 1,\ldots, r}{\max}  \vert   \delta^{\top} v_j    \vert   \right)\\
			& \leq  &  \displaystyle     r\left( \underset{j = 1,\ldots, r}{\max}  \vert   u^{\top} v_j    \vert   \right)\left\{\frac{b_r}{n^{ 1/2  } }  \Delta^2(\delta)  +    \Delta(\delta)   \right\}\\
			%	& \leq& \displaystyle   r \left\{ \underset{j = 1,\ldots, r}{\max}  \vert   \xi^{\top} v_j    \vert   \right\}\left\{    \frac{b_rt^2 }{n^{  1/2 }} +   t\right\},
		\end{array}
	\end{equation}
	where the second inequality follows from Lemma \ref{lem23}. The claim follows since 
	\[
	E\left\{  \left( \underset{j = 1,\ldots, r}{\max}  \vert   u^{\top} v_j    \vert   \right)  \right\}   \,=\, O(1).
	\]
	%last inequality   follows from 
	%Lemma \ref{lem23}. 
\end{proof}

\subsection{Proof  of  Proposition \ref{lem17}}

\begin{proof}
	Let  $B  = a_{\epsilon}\tilde{B}$,  $a_{\epsilon}>0$,  with $\tilde{B}$ as in Lemma \ref{lem15}, and  such that
	\begin{equation}
		\label{eqn:high_prob_event}
		\underset{x  \in   \mathrm{row}\{ D^{(r)} \} \,:\,      \|D^{(r) } x\|_1 \leq  1 }{\sup}\,\,\frac{  u^{\top}  x -  A  }{  \Delta(x)  } \leq  B ,
	\end{equation}
	happens with probability at least $1-\epsilon/4$. From here on, we suppose that (\ref{eqn:high_prob_event}) holds.% and first consider the case $V^*>n^{r-1}$.
	
	Now pick  $\kappa \in [0,1]$  fixed, and let $\tilde{ \delta}   =   \kappa  	(\hat{\theta}  -  \theta^*) $. Then
	by the optimality of $\hat{\theta}$ and  convexity of the quantile loss, we have that
	\[
	\displaystyle \sum_{i=1}^n  \rho_{\tau}(y_i  -     \tilde{\theta} _i  )   \,+\,   \lambda\|   D^{(r) } \tilde{\theta} \|_1  \,\leq \, 	\sum_{i=1}^n  \rho_{\tau}(y_i  - \theta^*_i )   \,+\,    \lambda\|   D^{(r) } \theta^*\|_1,
	\]
	where  $\tilde{\theta} =  \theta^* +\tilde{\delta}$.  Then as in the proof of  Lemma  3  from \cite{belloni2011}, 
	\begin{equation}
		\label{eqn:first_inequality}
		0 \,\leq \,  \lambda\left[\|   D^{(r) } \theta^*\|_1   -  \|   D^{(r) } \tilde{\theta}\|_1     \right]   \,+\,  (\tilde{\theta} - \theta^*)^{\top} u.		
	\end{equation} 
	%herefore, we proceed to bound  the  second term  in (\ref{eqn:first_ineuality}). 

	Next, notice that 
	%  suppose  that  $\| D^{(r) } (\tilde{\theta} - \theta^*)  \|_1> V^* n^{1-r}$. Then,
	% Hence, by (\ref{eqn:high_prob_event})
	%by Lemma  \ref{lem7}, we have that for some constant $a_1 >1$, we have that
	%\begin{equation}
	%	\label{eqn:infinity}
	%	 \left \|   \frac{1}{a_1  } P_R \frac{(\hat{\theta} - \theta^*)  }{   n^{k}     \| D^{(r) } (\hat{\theta} - \theta^*)  \|_1 }     \right\|_{\infty}   \leq  L,
	%\end{equation}
	%and 
	\begin{equation}
		\label{eqn:second_inequality3}
		\begin{array}{lll}
			(\tilde{\theta} - \theta^*)^{\top} u &  = &u^{\top}P_{  \mathcal{R}  }  (\tilde{\theta} - \theta^*)  + u^{\top} P_{  \mathcal{R}^{\perp} }  (\tilde{\theta} - \theta^*)\\
			& = &  \left(u^{\top}x\right)   \| D^{(r) } (\tilde{\theta} - \theta^*)  \|_1    +  u^{\top}P_{  \mathcal{R}^{\perp} }  (\tilde{\theta} - \theta^*),\\
			%	& \leq&     \left[    \left(\frac{V^*}{n^{r-1} }\right)^{1/2} B   \Delta\left\{\frac{  V^* P_{\mathcal{R}}(\tilde{\theta} - \theta^*) }{n^{r-1}\| D^{(r) } (\tilde{\theta} - \theta^*)\|_1} \right\} +  A n^{1-r}V^*\right]\frac{ n^{r-1}}{V^*} \| D^{(r) } (\tilde{\theta} - \theta^*)\|_1   +   (a^*)^{\top} P_{  \mathcal{R}^{\perp} }  (\tilde{\theta} - \theta^*)\\
			%	& \leq& B  \{ \| D^{(r) } (\tilde{\theta} - \theta^*)\|_1 \}^{1/2}   \Delta\left\{   P_{\mathcal{R}}(\tilde{\theta} - \theta^*) \right\} +  A \| D^{(r) } (\tilde{\theta} - \theta^*)\|_1  +   \\
			%	 & &(a^*)^{\top} P_{  \mathcal{R}^{\perp} }  (\tilde{\theta} - \theta^*)
		\end{array}%\| D^{(r) } ()\hat{\theta} - \theta^*\|_1
	\end{equation}
	where 
	\[
	x :=\frac{1 }{    \| D^{(r) } (\tilde{\theta} - \theta^*)  \|_1 } P_{  \mathcal{R} }(\tilde{\theta} - \theta^*) .
	\]	
	%with $ \tilde{C}_r $ as in Lemma \ref{lem20}. 	
	Hence,
	\[
	\|   D^{(r)} x \|_1 \,\leq \,1,
	\]
	which  combined with (\ref{eqn:second_inequality3}) and Lemma \ref{lem15} implies
	\begin{equation}
		\label{eqn:second_inequality4}
		\begin{array}{lll}
			(\tilde{\theta} - \theta^*)^{\top} u &  = &   \left\{ B   \Delta\left(x\right) +  A \right\}\| D^{(r) } (\tilde{\theta} - \theta^*)\|_1   +   u^{\top} P_{  \mathcal{R}^{\perp} }  (\tilde{\theta} - \theta^*)\\
			& \leq& B  \max\{ \| D^{(r) } (\tilde{\theta} - \theta^*)\|_1^{1/2}   ,1\} \Delta\left\{   P_{\mathcal{R}}(\tilde{\theta} - \theta^*) \right\} +  A \| D^{(r) } (\tilde{\theta} - \theta^*)\|_1  +   \\
			& &u^{\top} P_{  \mathcal{R}^{\perp} }  (\tilde{\theta} - \theta^*),
		\end{array}%\| D^{(r) } ()\hat{\theta} - \theta^*\|_1
	\end{equation}
	where the inequality follows from the fact that $\Delta( t v ) \leq  \max\{t,\sqrt{t}\} \Delta(v)$  for  $v\in R^n$  and  $t \geq 0$.
	
	%		\begin{equation}
	%	\label{eqn:second_inequality}
	%	\begin{array}{lll}
	%	& \leq&     \left[    \left(\frac{V^*}{n^{r-1} }\right)^{1/2} B   \Delta\left\{\frac{  V^* P_{\mathcal{R}}(\tilde{\theta} - \theta^*) }{n^{r-1}\| D^{(r) } (\tilde{\theta} - \theta^*)\|_1} \right\} +  A n^{1-r}V^*\right]\frac{ n^{r-1}}{V^*} \| D^{(r) } (\tilde{\theta} - \theta^*)\|_1   +   (a^*)^{\top} P_{  \mathcal{R}^{\perp} }  (\tilde{\theta} - \theta^*)\\
	%& \leq& B  \{ \| D^{(r) } (\tilde{\theta} - \theta^*)\|_1 \}^{1/2}   \Delta\left\{   P_{\mathcal{R}}(\tilde{\theta} - \theta^*) \right\} +  A \| D^{(r) } (\tilde{\theta} - \theta^*)\|_1  +   \\
	%& &(a^*)^{\top} P_{  \mathcal{R}^{\perp} }  (\tilde{\theta} - \theta^*)
	%\end{array}
	%\end{equation}
	
	%	where    the second inequality holds by the definition of $\Delta(\cdot)$ and by (\ref{eqn:high_prob_event}).%, and where $B$  satisfies $B \asymp \tilde{B}$.
	Suppose  now that $\| D^{(r) } (\tilde{\theta} - \theta^*)\|_1\geq 1$.	If  
	\[
	A \| D^{(r) } (\tilde{\theta} - \theta^*)\|_1 < B  \{ \| D^{(r) } (\tilde{\theta} - \theta^*)\|_1 \}^{1/2}  \Delta\left\{   P_{\mathcal{R}}(\tilde{\theta} - \theta^*) \right\} ,
	\]
	then
	\begin{equation}
		\label{final_upper_bound_p1}
		\| D^{(r) } (\tilde{\theta} - \theta^*)\|_1     \leq    \frac{  B^2 \Delta^2\left\{   P_{\mathcal{R}}(\tilde{\theta} - \theta^*) \right\}  }{A^2} 
	\end{equation}
	
	If  
	\[
	A \| D^{(r) } (\tilde{\theta} - \theta^*)\|_1 \geq B  \{ \| D^{(r) } (\tilde{\theta} - \theta^*)\|_1 \}^{1/2}   \Delta\left\{   P_{\mathcal{R}}(\tilde{\theta} - \theta^*) \right\},
	\]
	then
	% Then by (\ref{eqn:infinity}) and (\ref{eqn:second_inequality}), we have that for some  constant $a_2 >a_1$
	\begin{equation}
		\label{eqn:second_inequality2}
		\begin{array}{lll}
			(\tilde{\theta} - \theta^*)^{\top} u
			& \leq& 2A\| D^{(r) } (\tilde{\theta} - \theta^*)\|_1   +   u^{\top} P_{ \mathcal{R}^{\perp} }  (\tilde{\theta} - \theta^*).\\
			%+  \phi_n \left[\frac{\Delta^2(\hat{\theta} - \theta^*)}{n} +   D(\hat{\theta} - \theta^*)  \right]
		\end{array}%\| D^{(r) } ()\hat{\theta} - \theta^*\|_1
	\end{equation}
	Hence, choosing   $\lambda =  3 A$, and  combining   (\ref{eqn:first_inequality})  with (\ref{eqn:second_inequality2}),
	\[
	\begin{array}{lll}
		A\| D^{(r) } (\tilde{\theta} - \theta^*)\|_1  & \leq  & \lambda\left\{\|   D^{(r) } \theta^*\|_1   -  \|   D^{(r) } \tilde{\theta}\|_1     \right\}    + \\
		& &3A\| D^{(r) } (\tilde{\theta} - \theta^*)\|_1   + u^{\top} P_{  \mathcal{R}^{\perp} }  (\tilde{\theta} - \theta^*)\\
		%   \phi_n \left[\frac{\Delta^2(\hat{\theta} - \theta^*)}{n} +   D(\hat{\theta} - \theta^*)  \right] \\
		& \leq & 6 A   \|   D^{(r) } \theta^*\|_1    +       u^{\top} P_{  \mathcal{R}^{\perp} }  (\tilde{\theta} - \theta^*),\\,  
	\end{array}
	\]
	with the second inequality  follows by  the triangle inequality.
	Therefore,
	\[
	\| D^{(r) } (\tilde{\theta} - \theta^*)\|_1    \,\leq \max\left\{ \frac{V^*}{n^{r-1} },    \frac{6V^*}{n^{r-1} }  +    A^{-1}  u^{\top} P_{  \mathcal{R}^{\perp} }  (\tilde{\theta} - \theta^*),    \frac{  B^2 \Delta^2\left\{   P_{\mathcal{R}}(\tilde{\theta} - \theta^*) \right\}}{A^2}  \right\}.
	%\,4  \|   D^{(r) } \theta^*\|_1    +    \frac{   B    \frac{  \max\{L,\sqrt{L}\}  }{ \min\{L,\sqrt{L}\} }  }{  A } D(\hat{\theta} - \theta^*) +   \frac{(a^*)^{\top} P_{  R^{\perp} }  (\hat{\theta} - \theta^*)}{A}.
	\]
	
	Next	suppose   that $\| D^{(r) } (\tilde{\theta} - \theta^*)\|_1 < 1$.  	If  
	\[
	A \| D^{(r) } (\tilde{\theta} - \theta^*)\|_1 < B   \Delta\left\{   P_{\mathcal{R}}(\tilde{\theta} - \theta^*) \right\} ,
	\]
	then 
	\[
	\| D^{(r) } (\tilde{\theta} - \theta^*)\|_1 < \frac{B}{A} \Delta\left\{   P_{\mathcal{R}}(\tilde{\theta} - \theta^*) \right\} .
	\]
	If  
	\[
	A \| D^{(r) } (\tilde{\theta} - \theta^*)\|_1 \geq B  \Delta\left\{   P_{\mathcal{R}}(\tilde{\theta} - \theta^*) \right\},
	\]
	we proceed as before. The claim follows.
	
\end{proof}

\subsection{Proof of Theorem \ref{thm4} }
\label{sec:penalized_proof}
\begin{proof}
	\textbf{Steps 1--2 in proof outline.}

	Let   $\epsilon \in (0,1)$.  By   Proposition \ref{lem17} and Lemmas \ref{lem24}--\ref{lem16}
	%Lemmas  \ref{lem16}--    \ref{lem17}  
	we can   suppose that  
	the following  events  
	\begin{equation}
		\label{eqn:high_prob_events}
		\begin{array}{lll}
			\Omega_1  &= &\left\{
			\kappa	(\hat{\theta}  -  \theta^*)    \in \mathcal{A},  \,\,\,\,\forall \kappa\in [0,1] \,\,
			\right\},\\
			\Omega_2  &= &\left\{ 	\underset{x  \in R^n  }{\sup}\,\,\frac{  u^{\top}   P_{ \mathcal{R}^{ \perp } } x   }{   \frac{\Delta^2(x)}{n} +   \Delta(x)     }  \leq  E  \right\},\\
		\end{array}
	\end{equation}
	happen with probability at least  $1-\epsilon/2$ for some  constant $E$, and  with $\mathcal{A}$ as in Lemma \ref{lem17}. Furthermore, we set
	%recall form the proof of Lemma \ref{lem17} that $\theta^* \in K$, and write
	\begin{equation}
		\label{eqn:lambda}
		\lambda =  3 c_{\epsilon} n^{  (2r-1)/(2r+1)} \left( \log n \right)^{ 1/(2r+1) }   \|D^{(r)} \theta^*\|_1^{ - (2r-1)/(2r+1) }
	\end{equation}
	and 	$A =   \lambda/3$   	in   Lemma \ref{lem24}, where  $c_{\epsilon}>1$.
	%we  denote by  $C_1$  a positive constant such that $\|D^{(r)   }   \theta^* \|_1 \leq  C_1 n^{1-r}$. This constant exists since  $\theta^* \in K$.

	Then,  for a choice of   $t>0$ to be specified later, we have
	\[
	\begin{array}{lll}
		\mathrm{pr}\left\{     \Delta^2(\hat{\delta} )    > t^2  \right\}  & \leq&  \mathrm{pr}\left[  \left\{\Delta^2(\hat{\delta} )    > t^2  \right\} \cap  \Omega_1  \cap \Omega_2  \right]     +    \frac{\epsilon}{2}.
	\end{array} 
	\]
	Next suppose that  the event 
	\[
	\left\{\Delta^2(\hat{\delta} )    > t^2  \right\} \cap  \Omega_1  \cap \Omega_2  
	\]
	holds. Then,  proceeding as in the proof of Proposition \ref{prop:basic}  there exists $\tilde{ \delta} = u_{\hat{\delta}} \hat{\delta}$ with  $u_{\hat{\delta}} \in [0,1]$  such that  $\tilde{ \delta} \in \mathcal{A}$ and  $\Delta^2(\tilde{ \delta})  =  t^2$. Hence, by the basic inequality,
	\[
	\hat{M}(\theta^*+\tilde{ \delta} )        +   \lambda \left[     \|D^{(r)} (  \theta^*+\tilde{ \delta} )\|_1  - \|D^{(r)}   \theta^* \|_1    \right] \leq 0.
	\]
	Therefore,
	\[
	\begin{array}{lll}
		\underset{   \delta \in   \mathcal{A},   \Delta^2(  \delta ) \leq  t^2   }{\sup}\,\left[    M(\theta^*+\delta)  - \hat{M}(\theta^*+\delta)  
		+   \lambda\left\{ \|D^{(r)}   \theta^* \|_1  -\|D^{(r)} (  \theta^*+\tilde{ \delta} )\|_1        \right\} \right]  & \geq & M(\theta^* + \tilde{ \delta}) \\
		& \geq &c_0 t^2,
	\end{array}
	\]
	where the second inequality follows from  Lemma \ref{lem1}. Therefore,
	\begin{equation}
		\label{eqn:main}
		\begin{array}{lll}
			\mathrm{pr}\left[	\left\{\Delta^2(\hat{\delta} )    > t^2  \right\} \cap  \Omega_1  \cap \Omega_2 \right] & \leq& \mathrm{pr}\Bigg(  \Bigg\{\underset{   \delta \in   \mathcal{A},   \Delta^2(  \delta ) \leq  t^2   }{\sup}\,\bigg[    M(\theta^*+\delta)  - \hat{M}(\theta^*+\delta)  \\
			& &	+   \lambda\left\{ \|D^{(r)}   \theta^* \|_1  -\|D^{(r)} (  \theta^*+ \delta )\|_1        \right\} \bigg] \geq c_0 t^2\Bigg\} \cap \Omega_1  \cap \Omega_2 \Bigg)\\
			& \leq& \displaystyle 
			\frac{1}{c_0 t^2}	\,E\Bigg( 1_{\Omega_1  \cap \Omega_2} \underset{   \delta \in   \mathcal{A},   \Delta^2(  \delta ) \leq  t^2   }{\sup}\,\bigg[    M(\theta^*+\delta)  - \hat{M}(\theta^*+\delta)  \\
			& &	+   \lambda\left\{ \|D^{(r)}   \theta^* \|_1  -\|D^{(r)} (  \theta^*+\tilde{ \delta} )\|_1        \right\} \bigg] \Bigg)\\
			%	& \leq& \displaystyle 
			%	\frac{1}{c_0 t^2}	\,E\Bigg( 1_{\Omega_1  \cap \Omega_2} \underset{   \delta \in   \mathcal{A},   \Delta^2(  \delta ) \leq  t^2   }{\sup}\,\bigg[    M(\theta^*+\delta)  - \hat{M}(\theta^*+\delta)\bigg] \Bigg) \\
			%	& &\displaystyle 	+       \frac{\lambda }{c_0 \eta^2} E	\left[\,1_{\Omega_1  \cap \Omega_2} \,  \underset{   \delta \in   \mathcal{A},   \Delta^2(  \delta ) \leq  t^2   }{\sup}\, \left\{ \|D^{(r)}   \theta^* \|_1  -\|D^{(r)} (  \theta^*+\tilde{ \delta} )\|_1        \right\}\right] \\
			& \leq& \displaystyle  \frac{1}{c_0 t^2}	\,E\Bigg( 1_{\Omega_1  \cap \Omega_2} \underset{   \delta \in   \mathcal{A},   \Delta^2(  \delta ) \leq  t^2   }{\sup}\,\bigg[    M(\theta^*+\delta)  - \hat{M}(\theta^*+\delta)\bigg] \Bigg) \\
			& &\displaystyle 	+       \frac{\lambda }{c_0 t^2}	    \,  E\left\{1_{\Omega_1  \cap \Omega_2} \underset{   \delta \in   \mathcal{A},   \Delta^2(  \delta ) \leq  t^2   }{\sup}\, \|D^{(r)} \delta\|_1\right\},\\	
		\end{array}
	\end{equation}
	where the second inequality follows  from Markov's inequality, and the last from the triangle inequality.

	\textbf{Step 3 in proof outline.}

	Next, define  
	\begin{equation}
		\label{eqn:t}
		t:= c_{\epsilon} n^{  1/(4r+2) }  \left(  \log n \right)^{  1/(4r+2) }
	\end{equation}
	and notice that for  $ \delta  \in  \mathcal{A}$ with  $\Delta(\delta)  \leq  t$ it holds, by Lemma \ref{lem6}, that 
	%	\[
	%	\mathcal{H}(\eta)   =  \left\{  \delta  \in  \mathcal{A} \,:\,  \Delta(\delta)  \leq  \eta    \right\}%
	%	\]
	%	and  set 
	%	\[
	%	\eta := c_1 n^{  1/(4r+2) }  \left(  \log n \right)^{  1/(4r+2) } . % (V^*)^{\nu},  
	%	\]
	%	%for some  $\nu>0$.
	%	Notice that for   $\delta \in  \mathcal{H}(\eta)$ we have by Lemma \ref{lem6}%
	\[
	\Delta (     P_{\mathcal{R} }  \delta  )    \, \leq \,  \tilde{C}_r t , 
	\]
	for some positive constant $\tilde{C}_r$.  
	%Furthermore, 
	% \begin{equation}
	%\label{eqn:check}
	%\frac{B  \tilde{C}_r t  }{A}  <   C^{\prime} \frac{}{n^{r-1}}
	%%<  C^{\prime} n^{ \frac{4r -2 }{2r+1}   }  \left(  \log n \right)^{  1/(2r+1) }   <    \frac{A^2}{B^2}, 
	%\end{equation}
	%for some constant  $C^{\prime} >0$  which follows from simple algebra. 
	Hence,  	if in addition $\Omega_1 \cap \Omega_2$  holds then for a constant $\tilde{C} $ independent of  $c_{\epsilon}$,
	%. Hence,  for  some constants  $a_1,\tilde{C} >0$ independent  of  $c_1$, if  $\delta \in  \mathcal{H}(\eta)$   and  $\Omega_1 \cap \Omega_2$  holds  then\\
	\begin{equation}
		\label{eqn:bv}
		\begin{array}{lll}
			\|D^{(r)}\delta\|_1 & \leq & C_0    \max\left\{ \frac{V^*}{n^{r-1} },   \gamma  \Delta^2 ( P_{\mathcal{R}} \delta),  \gamma^{1/2}  \Delta ( P_{\mathcal{R}} \delta),\frac{V^*}{n^{r-1}}   +   A^{-1}u^{\top} P_{  R^{\perp} } \delta \right\} \\
			& \leq & C_0    \max\left\{ \frac{V^*}{n^{r-1} },   \gamma  \Delta^2 (   P_{\mathcal{R}} \delta),  \gamma^{1/2}  \Delta( P_{\mathcal{R}} \delta),\frac{V^*}{ n^{r-1}  }  +   A^{-1} \left(  \frac{\Delta^2(\delta)  }{n} +  \Delta(\delta)  \right) \right\} \\
			%\left[\gamma_1 D(\delta)    +  A^{-1} (a^*)^{\top} P_{  R^{\perp} }\delta+    \frac{C_1}{n^k}   \right]   +    \frac{1}{n^k } \\
			& \leq & C_0  \max\left\{ \frac{V^*}{n^{r-1} }, \tilde{C}_r^2  \frac{  B^2   }{A^2}  t^2, \tilde{C}_r  \frac{  B   }{A}  t, \frac{V^*}{ n^{r-1}  } +  A^{-1}\left(  \frac{t^2  }{n} +  t \right) \right\}\\
			%& \leq& a_1\bigg[  \max\bigg\{  \frac{V^*}{n^{r-1} } ,      \frac{ n^{ (2r-1)/(2r +1)  }\left(  \log n \right)^{ 1/(2r+1)  }  \| D^{(r)} \theta^* \|_1^{  2/(2r+1)   }  n^{r-1}  n^{  1/(2r+1) }  \left(  \log n \right)^{  1/(2r+1) } }{n^{  (4r-2)/(2r+1)   } \left( \log n \right)^{ 2/(2r+1) }   \|D^{(r)} \theta^*\|_1^{ - (4r-2)/(2r+1) } }    (V^*)^{2\nu},\\
			%& & \,\,\,\,\,\,\,\,\,  \frac{V^*}{n^{r-1}   }  +    \frac{  c_0  n^{  1/(2r+1) }  \left(  \log n \right)^{  1/(2r+1) }  }{   n^{  1+(2r-1)/(2r+1)   } \left( \log n \right)^{ 1/(2r+1) }   \|D^{(r)} \theta^*\|_1^{ - (2r-1)/(2r+1) } }    +  \\
			%&& \frac{  n^{  1/(4r+2) }  \left(  \log n \right)^{  1/(4r+2) } }{  n^{  (2r-1)/(2r+1)    } \left( \log n \right)^{ 1/(2r+1) }   \|D^{(r)} \theta^*\|_1^{ - (2r-1)/(2r+1) }  } \bigg\}\bigg]\\
			%& \leq& \tilde{C}    n^{1-r}   \max\{   V^*,    (V^*)^{  2\nu +  4r/(2r+1)   } \},
			& \leq& \tilde{C}    n^{1-r}   \max\{   V^*,    (V^*)^{  4r/(2r+1)   } \},
		\end{array}
	\end{equation}
	where the first inequality  follows  from the definition of  $\Omega_1$, the second because we are assuming  that  $\Omega_2$ holds,  the third  since $\Delta(\delta)\leq t$, and the  fourth  by definition of $A$, $B$ and $t$
	as simple algebra shows that
	\[
	\tilde{C}_r^2  \frac{  B^2 t^2   }{A^2}   \,\leq \,     \frac{C   (V^*)^{  4r/(2r+1)   }  }{ c_{\epsilon}^2  n^{r-1} } , 
	\]
	\[
	\tilde{C}_r	\frac{B   t  }{A}  <   C \frac{(V^*)^{4r/(4r+2)} }{n^{r-1}} ,
	\]
	and
	\[
	A^{-1}\left(  \frac{t^2  }{n} +  t \right)       \,\leq \,   \frac{C}{c_{\epsilon}} \left\{\frac{  (V^*)^{  (2r-1)/(2r+1) }   }{n^{r-1}} \right\} \frac{1}{    \left( n \log n\right)^{1/(4r+2)}   }.
	\]
	% and the last  by (\ref{eqn:eta}).
	
	\textbf{Steps 4--5 in proof outline.}

	Let us now define
	\[
	%\mathcal{L}(\eta)      \,=\, \left\{  \delta \,:\, \|D^{(r)}\delta\|_1  \leq  \tilde{C}    n^{1-r}   \max\{   V^*,    (V^*)^{  2\nu +  4r/(2r+1)   } \} ,\,\,\,\,\,  \text{and} \,\,\,\,\,  \Delta(\delta)    \leq \eta \right\}.
	\mathcal{H}(t)      \,=\, \left\{  \delta \,:\, \|D^{(r)}\delta\|_1  \leq  \tilde{C}    n^{1-r}   \max\{   V^*,    (V^*)^{   4r/(4r+2)   } \} ,\,\,\,\,\,  \text{and} \,\,\,\,\,  \Delta(\delta)    \leq t\right\}.
	\]
	%and 
	%\[
	%\eta =   c_1 n^{  1/(4r+2) }  \left(  \log n \right)^{  1/(4r+2) }   (V^*)^{\nu},  
	%\]
	%for some  $\nu >0$.
	Then  for some constant  $a>0$ that depends on $V^*$ but independent  of  $c_{\epsilon}$, (\ref{eqn:main}) and (\ref{eqn:bv}) imply

	\begin{equation}
		\label{eqn:main2}
		\begin{array}{lll}
			\mathrm{pr}\left[	\left\{\Delta^2(\hat{\delta} )    > t^2  \right\} \cap  \Omega_1  \cap \Omega_2 \right]
			& \leq& \displaystyle  \frac{1}{c_0 t^2}	\,E\Bigg[  \underset{   \delta \in  \mathcal{H}(t)        }{\sup}\,\bigg\{    M(\theta^*+\delta)  - \hat{M}(\theta^*+\delta)\bigg\} \Bigg] \\
			& &\displaystyle 	+       \frac{\lambda }{c_0 t^2}	    \, \underset{   \delta \in \mathcal{H}(t)       }{\sup}\, \|D^{(r)} \delta\|_1,\\
			& \leq&  \displaystyle  \frac{2}{c_0 t^2}   E\left\{    \underset{\delta \in \mathcal{H}(t)   }{\sup}\,\, \sum_{i=1}^{n}   \xi_i \delta_i  \right\}\\
			&& \displaystyle 	+       \frac{\lambda }{c_0 t^2}\left[\tilde{C}    n^{1-r}   \max\{   V^*,    (V^*)^{   4r/(4r+2)   } \}\right]\\	  %  \, \underset{   \delta \in \mathcal{H}(t)        }{\sup}\, \|D^{(r)} \delta\|_1,\\
			% & \leq&\displaystyle  \frac{2}{c_0 t^2}\Bigg[ C_r  \left\{   \frac{\eta^2}{n}   +       \eta \right\}  +  C_r    m(\eta,n) \{\log(en)\}^{1/2}  +\\
			% & &\displaystyle C_r m(t,n)\left\{  \frac{  n^{1/2}   \tilde{C}    n^{1-r}   \max\{   V^*,    (V^*)^{     4r/(2r+1)   } \} }{  m(\eta,n) }  \right\}^{  1/(2r) } \Bigg] \\
			%%% & &  \displaystyle  \Bigg] +\\
			% & & 	 +  \displaystyle   \frac{\lambda }{c_0 t^2}	    \,  \tilde{C}    n^{1-r}   \max\{   V^*,    (V^*)^{   4r/(2r+1)   } \}\\
			& \leq&\displaystyle     \bigg[ a  \left\{ c_{\epsilon}   (\log n )^{1/(4r+2)}\right\}^{1 -  1/(2r)   }   n^{     1/(2r+1)  }   \\%\max\left\{     (V^*)^{1/2r},  (V^*)^{ 2/(2r+1)  }\right\} \\
			%   & & \cdot \max\{  (V^*)^{ 1/2-1/(4r)} ,1\}\\
			& &+ 	    \displaystyle   	    \,  \tilde{C}     n^{1-r}   \max\{   V^*,    (V^*)^{   4r/(4r+2)   } \} \cdot \lambda  \bigg]\frac{1}{c_0   t^2 }  \\
			%   C_r \left\{  \frac{t^2 }{n^{ 1/2   }}  +   t\right\} +   C_r  m(t,n)  \left\{  \frac{  n^{ 1/2  }   V }{ m(t,n) } \right\}^{  1/(2r)    }  \\ %1/(2+2k)
			%& &  + C_r  m(t,n)  \{\log ( en)\}^{ 1/2  },
		\end{array}
		% where
		%for some positive  constants   $C_r, \tilde{c}_r$.
		%\end{array}
	\end{equation}
	with  $\xi_1,\ldots,\xi_n$ are   independent  Rademacher variables,  	where the the second inequality follows as in the proof of Lemmas \ref{lem3}--\ref{lem4}, and third  by  Proposition \ref{lem9}. Hence, given our choice of $\lambda$,
	
	\begin{equation}
		\label{eqn:main3}
		\begin{array}{lll}
			\mathrm{pr}\left[	\left\{\Delta^2(\hat{\delta} )    > t^2  \right\} \cap  \Omega_1  \cap \Omega_2 \right]
			& \leq& \displaystyle    \bigg[   a  \left\{ c_{\epsilon}   (\log n )^{1/(4r+2)}\right\}^{1 -  1/(2r)   }   n^{     1/(2r+1)  }   \\
			%\displaystyle  \frac{1}{c_0   \eta^2 }     a_2    \gamma^{-1}  c_1^{1 -  1/(2r)   }   n^{     1/(2r+1)  }   \max\left\{     (V^*)^{1/2r},  (V^*)^{ 2/(2r+1)  }\right\} \\
			%& & \cdot \max\{  (V^*)^{ 1/2-1/(4r)} ,1\}+\\
			& & 	    \displaystyle +  	    \,  \tilde{C}    n^{1-r}   \max\{   V^*,    (V^*)^{   4r/(4r+2)   } \} \cdot\\
			&&\displaystyle   3c_{\epsilon} n^{  (2r-1)/(2r+1)} \left( \log n \right)^{ 1/(2r+1) }   \|D^{(r)} \theta^*\|_1^{ - (2r-1)/(2r+1) }   \bigg]  \frac{1}{c_0   t^2 }\\
			& =& O\left(   \frac{1}{c_{\epsilon}} \right)\\
			& \leq&\displaystyle\frac{ \epsilon}{2}
		\end{array}
		% where
	\end{equation}
	provided that $c_{\epsilon}$ is large enough.
\end{proof}

\section{Theorem \ref{thm6} }

%Throughout this section we will use $C$ as a generic positive  constant  that can change from line to line. Furthermore, we denote by  $\hat{\theta}$  the solution to
%\[
%\underset{\theta \in  R^n}{\min}\,\,\sum_{i=1}^{n} \rho_{\tau}(y_i-\theta_i)  \,+\, \lambda \|D^{(r)} \theta\|_1,
%\]
%for some  $\lambda>0$.

Since we rely on proof machinery developed in~\cite{ortelli2019prediction}, we start by introducing some relevant notation from \cite{ortelli2019prediction}.

\subsection{Notation}
\label{sec:notation}
Throughout this section we will use $C$ as a generic positive  constant  that can change from line to line.  Let  $m = n - r$ be the number of rows of $D^{(r)}$.  For a vector  $b\in R^m$  and a set  $S \subset  \mathcal{D}:= \{1,\ldots,m\}$ we denote by  $b_S$  the vector $b_S = (b_j)_{j \in S}$ and we write  $b_{-S} = (b_j)_{ j \in \{1,\ldots,m\}\backslash S   }$.   

Following \cite{ortelli2019prediction}, we take   $S $ a subset of  $\{1,\ldots,m\}$ with  $s = \vert S\vert $.  We also denote by $t_1,\ldots,t_s$ the elements of $S$ and assume that  $r+1<t_1 < t_2< \ldots <t_s \leq n$, and let  $t_0 =  r$ and  $t_{s+1} =n$. Then we denote   $n_i  =  t_i - t_{i-1}$ for  $i \in \{1,\ldots,s+1\}$.%  and  
%$$n_{\max}  =   \underset{1\leq i\leq  s+1}{\max}  n_i.$$

In our entire proof   we take  $S $ to be the same as in~ \cite{ortelli2019prediction} which satisfies $\{   j \,:\,   (D^{(r)}\theta^*)_j  \neq  0      \} \subset S$, and  $s:=\vert  S \vert  \asymp  \vert \{   j \,:\,   (D^{(r)}\theta^*)_j  \neq  0      \}  \vert$. 
Furthermore,  we write  $\mathcal{N}_{-S} =   \{  \theta\in  R^n\,:\, (D^{(r)}\theta)_{-S} =0   \}$ and  denote  $r_S =  \text{dim}(\mathcal{N}_{-S})$. Also, the matrix  $D^{(r)}_{-S}$ denotes the matrix obtained after removing  from $D^{(r)}$ the   rows indexed by  $S$, and we set  $\Psi^{-S} = (D^{(r)}_{-S})^{\top} (   D^{(r)}_{-S} (D^{(r)}_{-S})^{\top})^{-1}$. The $j$th column of  $\Psi^{-S}$ is denoted as  $\psi^{-S}_j$.   Furthermore, we denote the orthogonal projections onto $\mathcal{N}_{-S}$ and $\mathcal{N}_{-S}^{\perp}$ as
$P_{\mathcal{N}_{-S}}$ and $P_{\mathcal{N}_{-S}^{\perp}}$ respectively.

For a vector  $w_{-S}$  such that $0\leq  w_j \leq  1$  we write  $(1-w_{-S}) ( D^{(r)} \theta )_{ -S  }   \,=\,\{  (1-w_j) ( D^{(r)} \theta )_j  \}_{  \{j  \in \mathcal{D}\backslash  S\}  }$.
We then study the estimator
\[
\hat{\theta}  = \underset{\theta \in  R^n}{\arg \min}\,\,\sum_{i=1}^{n} \rho_{\tau}(y_i-\theta_i)  \,+\, \lambda \|D^{(r)} \theta\|_1,
\]
for some  $\lambda>0$.

We also let  $A$ be such that
\begin{equation}
	\label{eqn:lower0}
	A \,\geq \, \underset{j \in  \mathcal{D}\backslash S }{\max}\,\|  \psi_j^{-S}\| (\log n)^{1/2},
	%\frac{    }{n^{1/2} },
\end{equation}
and by Section 3.1 in \cite{ortelli2019prediction}, we have that (\ref{eqn:lower0})  holds if 
\begin{equation}
	\label{eqn:lower}
	A \,\geq \, A^* \,:=\,  C^* n^{r}\left(\frac{1}{s+1}\right)^{r-1/2}\left(\frac{\log n}{n}\right)^{1/2},
\end{equation}
for some constant $C^* >0$.

With the notation from above, we also borrow the following definition from \cite{ortelli2019prediction}.

\begin{definition}
	\label{def2}
	For any sign vector  $q_S \in \{-1,1\}^s$  its noiseless effective sparsity is
	\[
	\Gamma^2(q_S )\,=\, \left(  \max\left\{   q_S^{\top}  (D^{(r)}\theta)_S   - \|(D^{(r)}\theta)_{-S}\|_1\,:\,  \|\theta\| = n^{1/2}   \right\}\right)^2.
	\]
	Its noisy effective sparsity is defined as
	\[
	\Gamma^2(q_S,w_{-S} )\,=\, \left(  \max\left\{   q_S^{\top}  (D^{(r)}\theta)_S   - \|(1-w_{-S})(D^{(r)}\theta)_{-S}\|_1\,:\,  \|\theta\| = n^{1/2}   \right\}\right)^2,
	\]
	with
	\[
	w_j \,=\,   \frac{\|\psi^{-S}_{j}\| (\log n)^{1/2} }{A^* }
	\]
	for  $j \in \mathcal{D} \backslash S$.
\end{definition}

\subsection{Proof outline}

We now provide a high level overview of the proof of Theorem \ref{thm6}. The first three steps in this proof are very similar to the first three steps in the proof outline of  Theorem \ref{thm4}.

\textbf{Step 1.}  

We  show in Proposition \ref{lcor2} that for any given $\epsilon > 0$ if  $\tilde{\theta}$ is in the line segment between $\hat{\theta}$ and $\theta^*$ then  $\tilde{\delta} = \tilde{\theta} - \theta^{*}$ belongs to  a restricted set  $\mathcal{A}$ (depending on $\epsilon$) with probability  at least  $1 - \epsilon/4$. We call this event $\Omega_1$ and this restricted set is of the form 
\begin{equation}
	\label{eqn:a}
	\mathcal{A} = \{  \delta \,  :\,   \mathrm{TV}^{(r)}(\delta)  \leq  C V^* + \text{some extra terms}          \}
\end{equation}
for some positive constant  $C$, see the precise definition in (\ref{eqn:restricted}). Here, the additional extra terms appearing in  (\ref{eqn:a}) are different to the corresponding ones in Step 1 of the proof of  Theorem \ref{thm4}.
Next we obtain, using the convexity of  $\hat{M}$, the optimality of $\hat{\theta}$, and  Lemma \ref{lem2}   that  for any $t>0$ it holds that
\[
\{ \Delta^2(\hat{\delta}) \geq  t^2    \} \subset   \left\{      \underset{ \delta \in \mathcal{A},\,\,\Delta^2(\delta)\leq  t^2  }{\sup} \,\,\left[ M(\theta^* +\delta )- \hat{M}(\theta^* +\delta) +   \lambda \|D^{(r)}\theta^*\|_1 - \|D^{(r)}(\theta^* +\delta)\| \right]\geq   c_0 t^2   \right\},
\]
where  $c_0>0$  is as in Lemma \ref{lem2}. Again, this step uses ideas very similar to the proof of Theorem  \ref{thm:basic}.

\textbf{Step 2.}  

We define another high probability event $\Omega_2$  as in (\ref{eqn:high_prob_events2}). Then  based on Proposition \ref{lcor2} and Lemma \ref{lem24},  we obtain that $\Omega_1 \cap \Omega_2$  happens with probability at least $1-\epsilon/2$.  Hence, we do our analysis conditioning on  $\Omega_1\cap \Omega_2$. We start with also assuming that 
$\Delta^2(\hat{\delta}) \geq  t^2 $  for a large enough $t>0$  (whose value is to be specified later).% with the goal of arriving at a contradiction.

\textbf{Step 3.}  

We show that if $\delta \in \mathcal{A},\,\,\Delta^2(\delta)\leq  t^2  $ and     $\Omega_1 \cap \Omega_2$ holds   then 
\[%
%\|D^{(r)} \delta \|_1 \leq \frac{  C  }{n^{r-1}}
\mathrm{TV}^{(r)}(\delta )   \leq  C
\]
for some $C>0$.  Here the details of the calculations are different to the corresponding ones in Step 3 of the proof of  Theorem \ref{thm4} but it leads us to obtaining a similar conclusion. It then follows from  Steps $1$ and $2$ above that we can
reduce our focus to upper bounding  the probability  of the event 
\[
\left\{      \underset{ \delta \in K  }{\sup} \,\,   \left[M(\theta^* +\delta )- \hat{M}(\theta^* +\delta) +   \lambda \|D^{(r)}\theta^*\|_1 - \lambda\|D^{(r)}(\theta^* +\delta)\| \right]\geq    c_0 t^2    \right\}
\]
where
\[
K \,:=\,    \left\{ \delta \,:\,  \mathrm{TV}^{(r)}(\delta )   \leq  C ,\,\,\Delta^2(\delta) \leq  t^2 \right\}.
\]

\textbf{Step 4.}

Using  Markov's inequality  and Step 3, it follows that $\mathrm{pr}( \Delta^2(\hat{\delta}) \geq  t) \leq \epsilon$ holds if 
\[
\frac{1}{ c_0 t^2}\,E\left[   \underset{ \delta \in K  }{\sup} \,\,   \left\{M(\theta^* +\delta )- \hat{M}(\theta^* +\delta) +   \lambda \|D^{(r)}\theta^*\|_1 - \lambda\|D^{(r)}(\theta^* +\delta)\| \right\}  \right]   \,\leq \,\epsilon.
\]
Then, using symmetrization and contraction results from Empirical Process Theory; see Lemmas  \ref{lem28} and \ref{lem29},  it reduces our task to show that 
\[
U : =  \frac{4}{c_0 t^2}E\left[  \underset{\delta \in K}{\sup}\,\,\left\{    \sum_{i=1}^{n}\xi_i \delta_i \  +    \frac{\lambda}{2}\|D^{(r)}\theta^*\|_1 -  \frac{\lambda}{2}\|D^{(r)}(\theta^*+\delta)\|_1    \right\}\right] \leq \epsilon,
\]
for  $\xi_1,\ldots,\xi_n$ independent Rademacher variables.

\textbf{Step 5.}

We now write 
\[
U \leq  T_1+T_2+T_3
\]
with
%where 
\[
T_1 \,:=\, \frac{4}{c_0 t^2}E\left\{  \underset{\delta \in K}{\sup}\,\,  \xi^{\top}  P_{ \mathcal{R}}^{\perp}\delta    \right\},
\]
\[
T_2 \,:=\, \frac{4}{c_0 t^2}E\left\{  \underset{\delta \in K}{\sup}\,\,  \xi^{\top}  P_{ \mathcal{N}_{-S}}P_{ \mathcal{R}}\delta    \right\},
\]
and
\[
T_3\,:=\, \frac{4}{c_0 t^2}E\left[  \underset{\delta \in K}{\sup}\,\,  \left\{\xi^{\top}  P_{ \mathcal{N}_{-S}^{\perp} }P_{ \mathcal{R}}\delta  +   \frac{\lambda}{2}\|D^{(r)}\theta^*\|_1 -  \frac{\lambda}{2}\|D^{(r)}(\theta^*+\delta)\|_1  \right\}\right]
\]
where  $  P_{ \mathcal{R}}^{\perp},    P_{ \mathcal{R}},P_{ \mathcal{N}_{-S}^{\perp} }$ and $ P_{ \mathcal{N}_{-S}}$  are defined in Section \ref{sec:notation}.  In the subsequent proof we set  $t$  and  $\lambda$  to satisfy:
\[
t  \asymp  (s+1)^{1/2}(\log^{1/2} n ) \left\{  \log \left( \frac{n}{s+1}   \right)  \right\}^{1/2}  \log ^{1/2} (s+1),
\]
and
\[
\lambda  \asymp  \max\left\{      \frac{n^{r-1} (s+1)  \log n  \log (s+1)  \log \frac{n}{s+1}    }{V^*},    n^{r-1/2}\left(\frac{1}{s+1} \right)^{r-1/2} (\log n )^{1/2}\right\}.
\]
As we see in the proof of Theorem \ref{thm6}, $T_1$ and $T_2$ turn out to be lower order terms as compared to $T_3$.
Therefore, from here our goal is to show that there exists a positive constant  $c$  such that  $\max\{T_1,T_2,T_3\} \leq c \epsilon$.

\textbf{Step 6.} 

Bounding $T_1$. This is done exactly similarly as in bounding the corresponding term $T_1$ inside  the proof of Theorem \ref{thm5}.

\textbf{Step 7.} 

Bounding $T_2$. This is handled using  Lemmas \ref{lem18}, \ref{lem20} and \ref{lem6}.

\textbf{Step 8.} 

We define an event  $\Omega_3 (b)$ for a constant  $b > 5$, see (\ref{eqn:omega3b}), and using a standard concentration inequality for maxima of subaussian random variables
we show that  $\Omega_3 (b)^c$  happens with high probability. Next we write
\[
T_3 \,= \,  T_{3,1}  + T_{3,2},
\]
where
\[
T_{3,1}\,:=\, \frac{4}{c_0 t^2}E\left[  \underset{\delta \in K}{\sup}\,\,  \left\{\xi^{\top}  P_{ \mathcal{N}_{-S}^{\perp} }P_{ \mathcal{R}}\delta  +   \frac{\lambda}{2}\|D^{(r)}\theta^*\|_1 -  \frac{\lambda}{2}\|D^{(r)}(\theta^*+\delta)\|_1  \right\} \bigg |  \Omega_3(b)\right] \mathrm{pr}\{\Omega_3(b)\}
\]
and
\[
T_{3,2} \,:=\,\frac{4}{c_0 t^2}E\left[  \underset{\delta \in K}{\sup}\,\,  \left\{\xi^{\top}  P_{ \mathcal{N}_{-S}^{\perp} }P_{ \mathcal{R}}\delta  +   \frac{\lambda}{2}\|D^{(r)}\theta^*\|_1 -  \frac{\lambda}{2}\|D^{(r)}(\theta^*+\delta)\|_1  \right\} \bigg |  \Omega_3(b)^c\right] \mathrm{pr}\{\Omega_3(b)^c\}.
\]

\textbf{Step 9.}  

Bounding  $T_{3,1}$.  This  a lower order  term that can be upper bounded  exploiting the definition of $\Omega_3(b)$.

\textbf{Step 10.}  

Bounding  $T_{3,2}$.  This is done  following the ideas for proving fast rates for trend filtering as laid out in Section 3.3 of \cite{ortelli2019prediction}.

\subsection{Restricted set for  Proof of    Theorem  \ref{thm6}   (Step 1)}

%Throughout we assume that  Assumption  \ref{as2}  holds   and write
%\begin{equation}
%\label{eqn:k_set}
%K  =   \left\{    \theta \,:\,    \| D^{(r)}  \theta \|_1    \leq   \frac{V^*}{n^{r-1}}   \right  \}.
%\end{equation}
%Also,  for  $\delta \in R^n$  we write  $\Delta(\delta )   =  \{\Delta^2(\delta)\}^{1/2}$ with  $\Delta^2(\cdot)$  defined as in   Lemma  \ref{lem2}. Furthermore,  we use the notation  $M$  and  $\hat{M}$  from Definition \ref{def1}.

%Before 
The following result is  obtained similarly to Proposition \ref{lem17}.

\begin{proposition}
	\label{lcor2}
	%Suppose that  $\|D^{(r)}\theta^*\|_1  =  O(n^{1-r})$.  
	%With the notation of Lemma \ref{lem25}, 
	Let $\epsilon \in (0,1)$
	then	there exists  $A$  satisfying  (\ref{eqn:lower})  such that  for
	%a choice   
	$$\lambda =   3A,  $$
	with probability   at least  $1-\epsilon/8$,
	\[
	\kappa	(\hat{\theta}  -  \theta^*)  \,\in \,\mathcal{A},\,\,\,\,\, \kappa\in [0,1],%\,:=\, \left\{ \delta \,:\,  \|D^{(r)} \delta \|  \,\leq \, C_0  \max\left\{ \frac{1}{n^k},   \gamma_1  \Delta^2 (\delta), \| D^{(r) } \delta\|_1   +   A^{-1}(a^*)^{\top} P_{  R^{\perp} }  (\tilde{\theta} - \theta^*)\\ \right\} \right\}, \,\,\,\,\,%\forall \kappa \in [0,1],
	%  \left[\gamma_1 D(\delta)    +  A^{-1} (a^*)^{\top} P_{  R^{\perp} }\delta+    \|D^{(r)}\theta^*\|_1  \right]    + n^{-k} \right\}, \,\,\,\,\,\forall \kappa \in [0,1],
	\]
	%		for all $\kappa \in [0,1]$  such that 
	%\begin{equation}
	%		\label{eqn:cond3}
	%	\Delta^2\left[  P_{\mathcal{R}}\{	\kappa	(\hat{\theta}  -  \theta^*)  \}\right] \,\leq\, 	\frac{A^2}{B^2},
	%	\end{equation}
	with 
	\[
	\mathcal{A}\,:=\, \left\{ \delta \,:\,  \|D^{(r)} \delta \|_1  \,\leq \, \tilde{C}_{\epsilon}  \max\left\{ \frac{V^*}{n^{r-1} },   \gamma  \Delta^2 (  P_{\mathcal{R}}  \delta),  \gamma^{1/2}  \Delta (  P_{\mathcal{R}}  \delta),  \frac{V^*}{n^{r-1} }  +   A^{-1} u^{\top} P_{  \mathcal{R}^{\perp} }  \delta\right\} \right\},
	\]
	where   $\tilde{C}_{\epsilon} >0$ is  a constant that depends on $\epsilon$, 
	$$\gamma :=   \frac{  B^2  }{A^2}  ,$$
	%  A^{-1} B    \frac{  \max\{L,\sqrt{L}\}  }{ \min\{L,\sqrt{L}\} }, $$  
	$B  = a_{\epsilon}\tilde{B}$,  for some constant $a_{\epsilon}>0$ that depends ,     with
	$u_i     =  \tau  -  1\{  y_i \leq \theta_i^* \}$  for  $i=1,\ldots,n$.  Here,
	\[
	\tilde{B} =    (s+1)^{1/2}  \log^{1/2} (s+1) .
	\]
\end{proposition}

\subsubsection{Auxiliary lemmas  for  proof  of Proposition \ref{lcor2}  }

\begin{lemma}
	\label{lem25}

	\textbf{\cite[Lemma A.2 in][]{ortelli2019prediction}. }
	With $A$ as in (\ref{eqn:lower}), it holds that 
	\[
	\underset{x   \in R^n }{\sup}\,\,\frac{  u^{\top}  x -  A\|w_{-S}(D^{(r) } x)_{-S}\|_1  }{  \|x\| }  = O_{\mathrm{pr}} (   \tilde{B} ),
	\]
	where  $u=  (u_1,\ldots,u_n)^{\top}$  is  a vector with independent coordinates satisfying $u_i   \sim \mathrm{SubGaussian}(\sigma^2)$ for $i=1,\ldots,n$, with $\sigma $ a constant 
	%	\begin{equation}
	%	\label{eqn:radem}
	%	\mathrm{pr}(\varepsilon_i = \tau) = 1-\tau,  \,\,\,\,\,\,  \mathrm{pr}(\varepsilon_i = \tau-1) = \tau,  \,\,\,\,\,\,\, \text{for}  \,\,\,i =1,\ldots,n,
	%	\end{equation}
	and 		$\tilde{B} =    (s+1)^{1/2}  \log^{1/2} (s+1)  $.
	% $\tilde{B}$  satisfies
	%		\[
	%
	%\|\Delta^{(k+1) \theta^*\|_1^{  \frac{2}{4k +6} }. 
	%	\]

\end{lemma}

\begin{proof}
	The proof is almost identical to that of  Lemma A.2 in \cite{ortelli2019prediction}. We start by noticing that for any $x \in R^n$, we have that 
	\[
	u^{\top} x\,=\, u^{\top }P_{ \mathcal{N}_{-S}  }x  + u^{\top }P_{ \mathcal{N}_{-S}^{\perp}  }x.
	\]
	Next let  $v_1,\ldots,v_{r_S}$ be an orthonormal basis of  $\mathcal{N}_{-S} $ and notice that
	\begin{equation}
		\label{eqn:proj}
		\begin{array}{lll}
			\vert u^{\top }P_{ \mathcal{N}_{-S}  }x   \vert^2 & \leq &    \|x\|^2  \,   \|  P_{ \mathcal{N}_{-S}  } u\|^2  \\
			&= & \displaystyle     \|x\|^2  \,\left\|\sum_{j=1}^{r_S}      (u^{\top}     v_j ) v_j\right\|^2\\
			&= &\displaystyle     \|x\|^2  \, \left\{ \sum_{j=1}^{r_S}  \vert u^{\top}     v_j \vert^2     \right\}\\
			& \leq&   \displaystyle   \|x\|^2   r_S\,   \underset{ j=1,\ldots,r_S }{\max }\vert   u^{T} v_j  \vert^2\\
		\end{array}
	\end{equation}
	where  $P_{ \mathcal{N}_{-S}  }$ and  $P_{ \mathcal{N}_{-S}^{\top}  }$ are the orthogonal projections onto $\mathcal{N}_{-S}$ and  $\mathcal{N}_{-S}^{\perp}$ respectively. Since by the subGaussian tail inequality,
	\[
	\underset{ j=1,\ldots,r_S }{\max }\vert   u^{T} v_j  \vert\,=\,	O_{\mathrm{pr}}\left(      \log^{1/2} (s+1)       \right),
	\]
	we obtain that
	\[
	\underset{ x  \in R^n }{\sup} \, \frac{ u^{\top }P_{ \mathcal{N}_{-S}  }x    }{\|x\|}\,=\,O_{\mathrm{pr}}\left[     (s+1)^{1/2}  \log^{1/2} (s+1)      \right].
	\]
	The rest of the proof concludes  by  proceeding as in the proof of Lemma A.2 in \cite{ortelli2019prediction}.
\end{proof}

As  Lemma \ref{lem15}  we obtain the following result.

\begin{lemma}
	\label{cor1}
	With the notation from Lemma \ref{lem25}, we have that
	%	There  exists    $A$  satisfying  $$A = C n^{  (2r-1)/(2r+1) } \left( \log n \right)^{ 1/(2r+1) }   \|D^{(r)} \theta^*\|_1^{ - (2r-1)/(2r+1) }   $$  for any  $C>1$  large enough such that
	\[
	\underset{x  \in   \mathrm{row}\{  D^{(r)} \}  \,:\,      \|D^{(r) } x\|_1 \leq 1 }{\sup}\,\,\frac{  u^{\top}  x -  A  }{  \Delta(x)  }  =O_{\mathrm{pr}} \left(  \tilde{B} \right),
	\]
	where  $u=  (u_1,\ldots,u_n)^{\top}$  is  a vector with independent coordinates satisfying %$\varepsilon_i   \sim \mathrm{SubGaussian}(\sigma^2)$ for $i=1,\ldots,n$, with $\sigma $ a constant 
	\begin{equation}
		\label{eqn:radem2}
		\mathrm{pr}(u_i = \tau) = 1-\tau,  \,\,\,\,\,\,  \mathrm{pr}(u_i = \tau-1) = \tau,  \,\,\,\,\,\,\, \text{for}  \,\,\,i =1,\ldots,n.
	\end{equation}
\end{lemma}

%\newpage
%\newpage
\subsection{Symmetrization and Contraction Lemmas for proof of  Theorem \ref{thm6} }

\begin{lemma}
	\label{lem28}
	%\label{eqn:symmetrization}
	(Symmetrization). For any set  $K$	and any $\lambda >0$ it holds that
	\[
	\begin{array}{l}
		\displaystyle E\left[  \underset{v \in K}{\sup}\,\,\left\{M(v) - \hat{M}(v)  +    \lambda \|D^{(r)}\theta^*\|_1 - \lambda\|D^{(r)}v\|_1    \right\}\right]\\
		\leq   	\displaystyle2 \,E\left\{   \underset{v \in K}{\sup}\,\,  \sum_{i=1}^{n}  \xi_i \hat{M}_{i}(v_i)    +    \frac{\lambda }{2}\|D^{(r)}\theta^*\|_1 -  \frac{\lambda }{2}\|D^{(r)}v\|_1    \right\}, 
	\end{array}
	\]
	where  $\xi_1,\ldots,\xi_n$ are   independent  Rademacher variables  independent  of  $\{y_i\}_{i=1}^n$.
\end{lemma}

\begin{remark}
	The above lemma (and its proof) is almost the same as the statement of Lemma~\ref{lem3} except that the term inside the supremum has an additional term involving the $r$th order total variation. 
\end{remark}

\begin{proof}
	We proceed using the notation argument from  the proof of Lemma \ref{lem3}. Then  	for  $\xi_1,\ldots,\xi_n$ independent Rademacher variables, independent of $y$ and $\tilde{y}$ we have that
	
	\[
	\begin{array}{l}
		\displaystyle   \underset{v \in K}{\sup}\,\,\left\{M(v) - \hat{M}(v)  +    \lambda \|D^{(r)}\theta^*\|_1 - \lambda\|D^{(r)}v\|_1    \right\}\\  \leq  E\left[\underset{v \in K}{\sup}\,\,\left\{\tilde{M}(v) - \hat{M}(v)  +    \lambda \|D^{(r)}\theta^*\|_1 - \lambda\|D^{(r)}v\|_1    \right\}    | y\right].\\
	\end{array}
	\]
	Hence,
	\[
	\begin{array}{l}
		\displaystyle E\left[  \underset{v \in K}{\sup}\,\,\left\{M(v) - \hat{M}(v)  +    \lambda \|D^{(r)}\theta^*\|_1 - \lambda\|D^{(r)}v\|_1    \right\}\right]\\
		\displaystyle  	\leq   	\,E\left[  \underset{v \in K}{\sup}\,\,\left\{\tilde{M}(v) - \hat{M}(v)  +    \lambda \|D^{(r)}\theta^*\|_1 - \lambda\|D^{(r)}v\|_1    \right\}\right]\\
		\displaystyle    =  \,E\left[  \underset{v \in K}{\sup}\,\,\left\{  \sum_{i=1}^{n} \xi_i\{\tilde{M}_i(v) - \hat{M}_i(v)\}\  +    \lambda \|D^{(r)}\theta^*\|_1 - \lambda\|D^{(r)}v\|_1    \right\}\right]\\
		\displaystyle \leq  \,E\left[  \underset{v \in K}{\sup}\,\,\left\{  \sum_{i=1}^{n}\xi_i \{\tilde{M}_i(v) \}\  +    \frac{\lambda}{2}\|D^{(r)}\theta^*\|_1 -  \frac{\lambda}{2}\|D^{(r)}v\|_1    \right\}\right]\\
		\displaystyle  \,\,\,\,\,\,\,+ \,E\left[  \underset{v \in K}{\sup}\,\,\left\{  \sum_{i=1}^{n}\xi_i \{\tilde{M}_i(v) \}\  +    \frac{\lambda}{2}\|D^{(r)}\theta^*\|_1 -  \frac{\lambda}{2}\|D^{(r)}v\|_1    \right\}\right]\\
		\displaystyle =\,  2\, E\left[  \underset{v \in K}{\sup}\,\,\left\{  \sum_{i=1}^{n}\xi_i \{\hat{M}_i(v) \}\  +    \frac{\lambda}{2}\|D^{(r)}\theta^*\|_1 -  \frac{\lambda}{2}\|D^{(r)}v\|_1    \right\}\right].\\
	\end{array}
	\]
\end{proof}

\begin{lemma}
	\label{lem29}
	(Contraction principle). Let  $h_1,\ldots,h_n \,:\,  R \rightarrow R$   $\eta$-Lipschitz functions  for some $\eta>0$.  Then for any compact set $K$ and  for  $\xi_1,\ldots,\xi_n$ independent Rademacher variables we have that 
	\[
	\begin{array}{l}
		\displaystyle   E\left[  \underset{v \in K}{\sup}\,\,\left\{  \sum_{i=1}^{n}\xi_i h_i(v_i) \  +    \frac{\lambda}{2}\|D^{(r)}\theta^*\|_1 -  \frac{\lambda}{2}\|D^{(r)}v\|_1    \right\}\right]\\
		\displaystyle \leq \, 	  E\left[  \underset{v \in K}{\sup}\,\,\left\{    \eta \sum_{i=1}^{n}\xi_i v_i \  +    \frac{\lambda}{2}\|D^{(r)}\theta^*\|_1 -  \frac{\lambda}{2}\|D^{(r)}v\|_1    \right\}\right]\\
	\end{array}
	\]
	for any $\lambda>0$.
\end{lemma}

\begin{remark}
	The above lemma (and its proof) is a version of the standard contraction result  \citep[Theorem  4.12  in][]{ledoux2013probability} except that the term inside the supremum has an additional term involving the $r$th order total variation. This lemma can be proved by following the standard proof argument of the original result. We provide this proof here for the sake of completeness.
\end{remark}

\begin{proof}
	We begin  by defining the function 
	\[
	g_{n-1}(v)   \,=\,  \sum_{i=1}^{n-1}\xi_i h_i(v_i)   \,+\,\frac{\lambda}{2}\|D^{(r)}\theta^*\|_1 -  \frac{\lambda}{2}\|D^{(r)}v\|_1.   
	\]
	Let  $v^+, v^- \in K$  such that 
	\[
	v^{+} \,\in \, \underset{  v \in K }{\arg \max}\,\,  \left\{  g_{n-1}(v) + h_n(v_n) \right\},
	\]
	and
	\[
	v^{-} \,\in \, \underset{  v \in K }{\arg \max}\,\,   \left\{ g_{n-1}(v) - h_n(v_n)\right\}.
	\]
	Next, letting $a =  \text{sign}(v_n^+ -  v_n^{-} )$, we notice that 
	\[
	\begin{array}{lll}
		\displaystyle 		E\left[    \underset{v \in K}{\sup}\,    \{g_{n-1}(v) +  \xi_nh_n(v_n)\}    \bigg|  \xi_1,\ldots,\xi_{n-1}\right] & = &  \displaystyle  \frac{1}{2}\underset{v \in K}{\sup}\,    \{g_{n-1}(v) + h_n(v_n)\}  \\
		& &  \displaystyle \,+\, \frac{1}{2}\underset{v \in K}{\sup}\,    \{g_{n-1}(v) -h_n(v_n)\} \\
		& = & \displaystyle  \frac{1}{2} \{g_{n-1}(v^+) + h_n(v_n^+)\}  \,+\,  \frac{1}{2}\{ g_{n-1}(v^-) -h_n(v_n^-)\}\\
		&\leq&   \displaystyle  \frac{1}{2}\{g_{n-1}(v^+)   +  g_{n-1}(v^-) \}  +   \frac{1}{2}\eta a \left\{ v_n^+   -v_n^-  \right\}\\
		& \leq&\displaystyle   \frac{1}{2} \underset{v\in K}{\sup}\{     g_{n-1}(v)  +  a \eta  v_n \}  \,+\,\frac{1}{2} \underset{v\in K}{\sup}\{     g_{n-1}(v)  -  a \eta  v_n  \}\\
		&= &E\left[    \underset{v \in K}{\sup}\,    \{g_{n-1}(v) +   \eta   \xi_n v_n\}    \bigg|  \xi_1,\ldots,\xi_{n-1}\right]
	\end{array}
	\]
	and the proof concludes by proceeding with a similar argument for the other $i$'s, $i\neq n$.
\end{proof}

\subsection{Proof of Theorem \ref{thm6} }

\begin{proof}
	
	\textbf{Steps 1--2 in proof outline.}
	
	Let   $\epsilon \in (0,1)$.  By  Lemmas  \ref{lem16}  and  Proposition \ref{lcor2}  we can   suppose that  
	the following  events  
	\begin{equation}
		\label{eqn:high_prob_events2}
		\begin{array}{lll}
			\Omega_1  &= &\left\{
			\kappa	(\hat{\theta}  -  \theta^*)    \in \mathcal{A},  \,\,\,\,\forall \kappa\in [0,1] \,\,\,\,
			\right\},\\
			\Omega_2  &= &\left\{ 	\underset{x  \in R^n  }{\sup}\,\,\frac{  u^{\top}   P_{ \mathcal{R}^{ \perp } } x   }{   \frac{\Delta^2(x)}{n} +   \Delta(x)     }  \leq  E_1  \right\},\\
			%	\Omega_3 & = &\left\{   \underset{x \in \mathbb{R}^n}{\sup}\,    \frac{\varepsilon^{\top}(x-\theta^*)   +     \lambda( \| D^{(r)} \theta^*\|_1 -  \| D^{(r)}  x  \|_1 )   }{\|x\| } \,\leq\,  E_{2}\lambda\Gamma(q_S,  w_{-S})  \right\},
		\end{array}
	\end{equation}
	happen with probability at least  $1-\epsilon/2$ for some  constant $E_1>0$, and  with $\mathcal{A}$   as in Lemma \ref{lcor2}.

	Following \cite{ortelli2019prediction}, we take   $S $  to be such that $\{   j \,:\,   (D^{(r)}\theta^*)_j  \neq  0      \} \subset S$, and  $s:=\vert  S \vert  \asymp  \vert \{   j \,:\,   (D^{(r)}\theta^*)_j  \neq  0      \}  \vert$.
	Then for $\epsilon>0$ we set
	\begin{equation}
		\label{eqn:A}
		A  \,=\,  c_{\epsilon} \max\left\{      \frac{n^{r-1} (s+1)  \log n  \log (s+1)  \log \frac{n}{s+1}    }{V^*},    n^{r-1/2}\left(\frac{1}{s+1} \right)^{r-1/2} (\log n )^{1/2}\right\}
	\end{equation}
	for a large enough constant  $c_{\epsilon}$  such that  the events in (\ref{eqn:high_prob_events2}) happen with probability at least  $1-\epsilon/2$.  We also set  $\lambda =  3A$.% and choose  $s$  such that   $\lambda \asymp  s  n^{r-1}$.

	\textbf{Step 3 in proof outline.}

	Next, 	 define  
	\[
	t:= c_{\epsilon} (s+1)^{1/2}(\log^{1/2} n ) \left\{  \log \left( \frac{n}{s+1}   \right)  \right\}^{1/2}  \log ^{1/2} (s+1),
	\]
	and notice that for  $ \delta  \in  \mathcal{A}$ with  $\Delta(\delta)  \leq  t$ it holds, by Lemma \ref{lem6}, that 
	\[
	\Delta^2 (     P_{\mathcal{R} }  \delta  )    \, \leq \,  \tilde{C}_r \left\{  t^2 + t^2 \gamma(t,n) + \frac{t^4}{n}    \right\} \leq  C t^2 , 
	\]
	for some positive constants $\tilde{C}_r$ and $C$, where we have used the fact that $t / n^{1/2}= O(1)$.   
	%Furthermore, by a simple calculation,
	%\[
	% Ct^2  \,\leq  \frac{A^2}{B^2}.
	%\]
	
	If in addition $\Omega_1 \cap \Omega_2$  holds then
	% for a constant $\tilde{C} $ independent of  $c_{\epsilon}$,
	\begin{equation}
		\label{eqn:bv3} 
		\begin{array}{lll}
			\|D^{(r)}\delta\|_1 & \leq &  \displaystyle C_0  \max\left\{ \frac{V^*}{n^{r-1} },  \frac{  B^2 }{A^2}  \Delta^2 (     P_{\mathcal{R} }  \delta  ),  \frac{  B}{A}  \Delta (     P_{\mathcal{R} }  \delta  ),\frac{V^*}{ n^{r-1}  } +  A^{-1}\left(  \frac{t^2  }{n} +  t \right) \right\}.\\
			%& \leq& \tilde{C}    n^{1-r}   \max\{   V^*,    (V^*)^{  4r/(2r+1)   } \},
		\end{array}
	\end{equation}
	
	Next, notice that
	\[
	\frac{B^2 t^2}{A^2} \,\leq\,    \frac{V^*}{n^{r-1}}   \frac{a_{\epsilon}^2      }{\log n   \, \log( \frac{n}{s+1} )}\frac{V^*}{n^{r-1}} \,< \, C\frac{V^*}{n^{r-1}} ,
	\]
	for large enough $n$.
	\iffalse 
	Furthermore, 
	\[
	\frac{B t}{A} \,\leq\, \frac{V^*}{n^{r-1} (\log n)^{1/2}},
	\]
	and 
	\[
	\frac{t^4B^2}{n A^2}\,\leq\, \left\{\frac{V^*}{n^{r-1}} \frac{a_{\epsilon}^2 c_{\epsilon}^2 (s+1)\log(s+1)}{n}  \right\}\frac{V^*}{n^{r-1}} <   \frac{V^*}{n^{r-1}},
	\]
	for large enough $n$. 
	Also,  by  its definition in Lemma \ref{lem5}, $\gamma(t,n)$ satisfies
	\[
	\gamma(t,n) \,\leq\,O\left(  \max\left\{ \frac{t}{n^{1/2} },\frac{t^2}{n} \right\} \right),
	\]
	and 
	\[
	\frac{ t^3 B^2}{A^2  n^{1/2}}  \,=\,  \frac{V^*}{n^{r-1}} \left\{  \frac{   c_{\epsilon} a_{\epsilon}^2  (s+1)^{1/2} \log^{1/2}(s+1) }{ n^{1/2} \,  \log^{1/2} n \,\log^{1/2} \frac{n}{s+1}   }  \right\}\frac{V^*}{n^{r-1}} <     \frac{V^*}{n^{r-1}}.
	\]
	\fi
	Furthermore, 
	\[
	\frac{B t}{A} \,\leq\, \frac{V^*}{n^{r-1} (\log n)^{1/2}},
	\]
	and
	\[
	\frac{t^2}{An} +  \frac{t}{A} \,<\, \frac{c_{\epsilon}}{n}\cdot\frac{V^*}{n^{r-1}} \cdot \frac{V^*}{n^{r-1}}+   \frac{V^*}{n^{r-1}}  \,< \,\frac{2V^*}{n^{r-1}}.  
	\]
	Hence, for a constant $\tilde{C}>0$ we have that
	\begin{equation}
		\label{eqn:bv4} 
		\begin{array}{lll}
			\|D^{(r)}\delta\|_1 & \leq & \displaystyle \tilde{C}\frac{V^*}{n^{r-1}}.
			%& \leq& \tilde{C}    n^{1-r}   \max\{   V^*,    (V^*)^{  4r/(2r+1)   } \},
		\end{array}
	\end{equation}

	\textbf{Steps 4--5 in proof outline.}
	
	Furthermore,  denoting  $\hat{\delta} =  \hat{\theta}-\theta^*$, we notice that 
	\begin{equation}
		\label{eqn:upper}
		\begin{array}{lll}
			\mathrm{pr}\left\{     \Delta^2(\hat{\delta} )    > t^2  \right\}  & \leq&  \mathrm{pr}\left[  \left\{\Delta^2(\hat{\delta} )    > t^2  \right\} \cap  \Omega_1  \cap \Omega_2  \right]     +    \frac{\epsilon}{2}.
		\end{array} 
	\end{equation}
	Then proceding as in the proof of Theorem \ref{thm4},
	
	\begin{equation}
		\label{eqn:first}
		\begin{array}{lll}
			\mathrm{pr}\left[	\left\{\Delta^2(\hat{\delta} )    > t^2  \right\} \cap  \Omega_1  \cap \Omega_2  \right] & \leq& \mathrm{pr}\Bigg(  \Bigg\{\underset{  \delta \in \mathcal{A},\,\,\Delta^2(\delta) \leq t^2  }{\sup}\,\bigg[    M(\theta^*+\delta)  - \hat{M}(\theta^*+\delta)  \\
			& &	\,\,\,\,\,+   \lambda \|D^{(r)}   \theta^* \|_1  -\lambda\|D^{(r)} (  \theta^*+ \delta )\|_1    \bigg] >c_0 t^2\Bigg\} \\
			& &\,\,\,\,\,\,\,\,\cap \Omega_1  \cap \Omega_2 \Bigg).\\
		\end{array}
	\end{equation}
	And so,    from (\ref{eqn:bv4}) and by Markov's inequality 
	\[
	\begin{array}{lll}
		\mathrm{pr}\left[	\left\{\Delta^2(\hat{\delta} )    > t^2  \right\} \cap  \Omega_1  \cap \Omega_2  \right] & \leq&\displaystyle \mathrm{pr}\Bigg(  \Bigg\{\underset{   \delta \,:\, \|D^{(r)} \delta\|_1\leq \tilde{C}V^*/n^{r-1},   \,\,  \Delta^2(  \delta ) \leq  t^2  }{\sup}\,\bigg[    M(\theta^*+\delta)  - \hat{M}(\theta^*+\delta)  \\
		& &	\displaystyle\,\,\,\,\,+   \lambda \|D^{(r)}   \theta^* \|_1  -\lambda\|D^{(r)} (  \theta^*+ \delta )\|_1    \bigg] >c_0 t^2\Bigg\} \\
		& &\displaystyle\,\,\,\,\,\,\,\,\cap \Omega_1  \cap \Omega_2 \Bigg)\\
		&\leq&\displaystyle\frac{1}{c_0 t^2}\,E\bigg\{ \underset{   \delta \in  K  }{\sup}\,\bigg[    M(\theta^*+\delta)  - \hat{M}(\theta^*+\delta)  \\
		& &\displaystyle	\,\,\,\,\,\,\,\,\,\,\,\,+   \lambda \|D^{(r)}   \theta^* \|_1  -\lambda\|D^{(r)} (  \theta^*+ \delta )\|_1    \bigg] \bigg\}
	\end{array}
	\]
	where 
	\begin{equation}
		\label{eqn:k}
		K\,:=\,\left\{ \delta \in R^n\,:\, \|D^{(r)} \delta\|_1\leq \tilde{C}V^*/n^{r-1},   \,\,  \Delta^2(  \delta ) \leq  t^2 \right\}.
	\end{equation}
	
	Therefore, by Lemmas \ref{lem28}--\ref{lem29}, we obtain that for $\xi_1,\ldots,\xi_n$ independent Rademacher variables independent of  $y$, it holds that
	\begin{equation}
		\label{eqn:first2}
		\begin{array}{lll}
			\mathrm{pr}\left[	\left\{\Delta^2(\hat{\delta} )    > t^2  \right\} \cap  \Omega_1  \cap \Omega_2  \right]   & \leq&\displaystyle  \frac{4}{c_0 t^2}E\left[  \underset{\delta \in K}{\sup}\,\,\left\{    \sum_{i=1}^{n}\xi_i \delta_i \  +    \frac{\lambda}{2}\|D^{(r)}\theta^*\|_1 -  \frac{\lambda}{2}\|D^{(r)}(\theta^*+\delta)\|_1    \right\}\right],\\
			% & \leq  &\displaystyle  \frac{4}{c_0 t^2}E\left[  \underset{v \in K}{\sup}\,\,\left\{    \xi^{\top}(v -\theta^*) \  +    \frac{\lambda}{2}\|D^{(r)}\theta^*\|_1 -  \frac{\lambda}{2}\|D^{(r)}(\theta^*+\delta)\|_1    \right\}\right]\\
			%	& \leq &\displaystyle T_1 +T_2+T_3, \\
		\end{array}
	\end{equation}
	which combined with (\ref{eqn:upper})   leads to 
	\begin{equation}
		\label{eqn:final}
		\begin{array}{lll}
			\mathrm{pr}\left\{     \Delta^2(\hat{\delta} )    > t^2  \right\}  & \leq&\displaystyle  \frac{4}{c_0 t^2}E\left[  \underset{\delta \in K}{\sup}\,\,\left\{    \sum_{i=1}^{n}\xi_i \delta_i \  +    \frac{\lambda}{2}\|D^{(r)}\theta^*\|_1 -  \frac{\lambda}{2}\|D^{(r)}(\theta^*+\delta)\|_1    \right\}\right]  \,+\,  \frac{\epsilon}{2}. \\
		\end{array}
	\end{equation}
	Then,   we must give an upper bound to 
	\begin{equation}
		\label{eqn:first3}
		\begin{array}{lll}
			%\mathrm{pr}\left[	\left\{\Delta^2(\hat{\delta} )    > t^2  \right\} \cap  \Omega_1  \cap \Omega_2  \right]   & \leq&
			\displaystyle      U \,:=\,  \frac{4}{c_0 t^2}E\left[  \underset{\delta \in K}{\sup}\,\,\left\{    \sum_{i=1}^{n}\xi_i \delta_i \  +    \frac{\lambda}{2}\|D^{(r)}\theta^*\|_1 -  \frac{\lambda}{2}\|D^{(r)}(\theta^*+\delta)\|_1    \right\}\right],\\
			% & \leq  &\displaystyle  \frac{4}{c_0 t^2}E\left[  \underset{v \in K}{\sup}\,\,\left\{    \xi^{\top}(v -\theta^*) \  +    \frac{\lambda}{2}\|D^{(r)}\theta^*\|_1 -  \frac{\lambda}{2}\|D^{(r)}(\theta^*+\delta)\|_1    \right\}\right]\\
		\end{array}
	\end{equation}
	%\ref{lem25}
	where $\xi_1,\ldots,\xi_n$ are independent Rademacher variables independent of  $y$ with  $K$ as in (\ref{eqn:k}).  Towards that end, notice that
	\[
	U \leq  T_1+T_2+T_3
	\]
	where
	%where 
	\[
	T_1 \,:=\, \frac{4}{c_0 t^2}E\left\{  \underset{\delta \in K}{\sup}\,\,  \xi^{\top}  P_{ \mathcal{R}}^{\perp}\delta    \right\},
	\]
	\[
	T_2 \,:=\, \frac{4}{c_0 t^2}E\left\{  \underset{\delta \in K}{\sup}\,\,  \xi^{\top}  P_{ \mathcal{N}_{-S}}P_{ \mathcal{R}}\delta    \right\},
	\]
	and
	\[
	T_3\,:=\, \frac{4}{c_0 t^2}E\left[  \underset{\delta \in K}{\sup}\,\,  \left\{\xi^{\top}  P_{ \mathcal{N}_{-S}^{\perp} }P_{ \mathcal{R}}\delta  +   \frac{\lambda}{2}\|D^{(r)}\theta^*\|_1 -  \frac{\lambda}{2}\|D^{(r)}(\theta^*+\delta)\|_1  \right\}\right]
	\]
	with  $\mathcal{R}$ as defined in Section \ref{sec:sketch}.

	\textbf{Step 6  in proof outline.}

	Next we proceed to bound  $T_1$,  $T_2$ and $T_3$.  First, we notice that as in (\ref{eqn:t1}) it follows that  $T_1 =  O(\frac{1}{t})$.   %\rightarrow 0$ as $n \rightarrow \infty$. %Furthermore,

	\textbf{Step 7 in proof outline.}
	
	To bound $T_2$ notice that  
	for a positive constant  $C>0$,
	
	% by the argument in (\ref{eqn:proj}),
	\[
	\begin{array}{lll}
		\displaystyle 	T_2 &\leq &\displaystyle \frac{4}{c_0 t^2}E\left\{  \underset{\delta  \in K\,\,:\,     \|P_R \delta \|^2\leq  h(t,n)\max\{  \tilde{C}_r \tilde{C} \frac{V^*}{n^{r-1}},1\}     }{\sup}\,\,  \xi^{\top}  P_{ \mathcal{N}_{-S} } P_{\mathcal{R}}\delta    \right\}\\
		&\leq& \displaystyle \frac{4}{c_0 t^2}E\bigg(   \bigg[ h(t,n)\max\bigg\{  \tilde{C}_r \tilde{C} \frac{V^*}{n^{r-1}},1\bigg\}      r_S\,   \underset{ j=1,\ldots,r_S }{\max }\vert   \varepsilon^{T} v_j  \vert^2\bigg] ^{1/2} \bigg)\\
		&\leq &\displaystyle \frac{C}{c_0 t^2}\{h(t,n)\}^{1/2}(s+1)^{1/2}  \,E\left(\underset{ j=1,\ldots,r_S }{\max }\vert   \varepsilon^{T} v_j  \vert\right)\\
		&\leq &\displaystyle \frac{2C}{c_0 t^2}\{h(t,n)\}^{1/2}(s+1)^{1/2}(\log n)^{1/2}\\
		%    & \leq & \displaystyle \frac{C}{c_{\epsilon}}\\
		%  & \leq&  \frac{\epsilon}{3},
		%\,+\,   \\
		% & &\displaystyle \frac{4}{c_0 t^2}E\left\{  \underset{\delta \in K}{\sup}\,\,  \xi^{\top}  P_{ \mathcal{N}_{-S} } P_{\mathcal{R}^{\perp}}\delta    \right\}\\
		%&\leq&\displaystyle \frac{4}{c_0 t^2}   \,  E\bigg\{     t(\tilde{s}+1)^{1/2}\left(\frac{V}{n^{r-1}}\right)^{1/2}   \bigg(\underset{j =1, \ldots,  \tilde{s}  }{\max}   \vert   \xi^{\top} v_j \vert  \bigg)^{1/2}    \bigg\}
	\end{array}
	\] 
	where the first inequality follows from Lemmas  \ref{lem18}, \ref{lem20} and \ref{lem6}, the second as  (\ref{eqn:proj}), and the fourth  by the expected value of maxima of subGaussian random variables inequality.
	
	Therefore   for a universal constant $C$ independent of  $c_{\epsilon}$ we obtain,%, and  by choosing  $c_{\epsilon}$ large enough we arrive at
	\begin{equation}
		\label{eqn:upper_t2}
		T_2 \,\leq \, \frac{C}{c_{\epsilon}}  \,\leq \,   \frac{\epsilon}{3},
	\end{equation}
	where the first inequality  follows from our choice  of  $t$ and the second inequality in (\ref{eqn:upper_t2}) follows by choosing  $c_{\epsilon}$ large enough.

	\textbf{Step 8 in proof outline.}

	Next we proceed to bound  $T_3$. Based on (\ref{eqn:A}), suppose that 
	\[
	\lambda \,=\, 3c_{\epsilon}  n^{r-1/2}\left(\frac{1}{s+1} \right)^{r-1/2} (\log n)^{1/2}.  
	\]
	Then, for $\delta\in K$
	we have by Hölder's inequality and choosing  $c_{\epsilon} $ large enough that
	\begin{equation}
		\label{eqn:t3.4}
		\begin{array}{lll}
			\displaystyle 	\xi^{\top}  P_{ \mathcal{N}_{-S}^{\perp} }P_{ \mathcal{R}}\delta  +   \frac{\lambda}{2}\|D^{(r)}\theta^*\|_1 -  \frac{\lambda}{2}\|D^{(r)}(\theta^*+\delta)\|_1
			& \leq &\displaystyle \| w_{-S}(D^{(r)} \delta)_{-S } \|_1 
			\frac{A}{5(\log n)^{1/2}}\,    \underset{j\in   \mathcal{D}\backslash S  }{\max} \,\frac{\vert  \xi^{\top} \psi^{-S }_j\vert }{\|\psi^{-S }_j\| }  \\
			& &\displaystyle   +  \frac{\lambda}{2}\|D^{(r)}\theta^*\|_1 -  \frac{\lambda}{2}\|D^{(r)}(\theta^*+\delta)\|_1.
		\end{array}
	\end{equation}
	Next let 
	\begin{equation}
		\label{eqn:omega3b}
		\Omega_3 (b)\,=\, \left\{    \underset{j\in   \mathcal{D}\backslash S  }{\max} \,\frac{\vert  \xi^{\top} \psi^{-S }_j\vert }{\|\psi^{-S }_j\| }     \,\geq \,  b\sqrt{\log n}  \right \},
	\end{equation}
	for  $b >0$. Then
	\begin{equation}
		\label{eqn:t3.5}
		%\[
		\mathrm{pr}\{  \Omega_3(b) \} \,\leq \, \exp\{ -b^2 (\log n)/2 +   \log n \}  \,\leq \,  n^{1-b^2/2}.
		%\]
	\end{equation}
	by the tail inequality for Rademacher variables and by union bound. 
	Furtheremore,  from  Subsections 3.3.2, 3.3.3, 3.3.4 and 3.3.5 in \cite{ortelli2019prediction}
	\begin{equation}
		\label{eqn:t3.6}
		\underset{j\in   \mathcal{D}\backslash S  }{\max} w_{j} \,\leq \,1.
	\end{equation}
	%for a positve constant $C_1>0$.
	Hence, combining (\ref{eqn:t3.4})--(\ref{eqn:t3.6}) we obtain that 
	\begin{equation}
		\label{eqn:t3_part1}
		\begin{array}{l}
			\displaystyle \frac{4}{c_0 t^2}E\left[  \underset{\delta \in K}{\sup}\,\,  \left\{\xi^{\top}  P_{ \mathcal{N}_{-S}^{\perp} }P_{ \mathcal{R}}\delta  +   \frac{\lambda}{2}\|D^{(r)}\theta^*\|_1 -  \frac{\lambda}{2}\|D^{(r)}(\theta^*+\delta)\|_1  \right\} \bigg |  \Omega_3(b)\right] \mathrm{pr}\{\Omega_3(b)\}\\
			\displaystyle \leq 	\frac{4}{c_0 t^2}\left\{   n^{1/2} \{h(t,n) \}^{1/2}+    3c_{\epsilon}  n^{1/2}\left(\frac{1}{s+1} \right)^{r-1/2} (\log n)^{1/2}   V^*  \right\} n^{1-b^2/2}\\
			\underset{n\rightarrow \infty }{ \rightarrow }  0,
		\end{array}
	\end{equation}
	where  the inequality holds by Lemma \ref{lem6} and Cauchy–Schwarz inequality, and the limit by the definition of $t$ and choosing $b=5$.

	\textbf{Step 9  in proof outline.}
	
	Additionally,  given that $\Omega_3(b)^c$ holds then (\ref{eqn:t3.4}) and (\ref{eqn:t3.6})  imply that  for large enough $c_{\epsilon}$, defining  $q_S \,:=\, \mathrm{sign}\{(D^{(r)} \theta^*)_S\}$, it holds that
	\[
	\begin{array}{lll}
		\displaystyle 	  \underset{\delta \in K}{\sup}\,\,  \left\{\xi^{\top}  P_{ \mathcal{N}_{-S}^{\perp} }P_{ \mathcal{R}}\delta  +   \frac{\lambda}{2}\|D^{(r)}\theta^*\|_1 -  \frac{\lambda}{2}\|D^{(r)}(\theta^*+\delta)\|_1  \right\}  \\
		\leq   	\displaystyle \frac{\lambda}{2}  \,\underset{\delta \in K}{\sup}\,\,  \left[ \| w_{-S}\{D^{(r)} \delta\}_{-S } \|_1  +   \|D^{(r)}\theta^*\|_1 -  \|D^{(r)}(\theta^*+\delta)\|_1  \right]  \\
		= \displaystyle \frac{\lambda}{2}  \,\underset{\delta \in K}{\sup}\,\,  \left[ 	\|w_{-S}\{D^{(r) } \delta\}_{-S}\|_1 + \| \{D^{(r)} \theta^*\}_S\|_1 -  \| D^{(r)}  (\theta^*+\delta)\|_1  \right]  \\
		= \displaystyle \frac{\lambda}{2}  \,\underset{\delta \in K}{\sup}\,\,  \left[\| \{D^{(r)} \theta^*\}_S\|_1 -  \| \{D^{(r)} (\theta^*+\delta)\}_S\|_1 -   \|  (1-w_{-S}) \{D^{(r)}(\theta^*+\delta)\}_{-S}     \|_1 	 \right]  \\
		\leq  \displaystyle \frac{\lambda}{2}  \,\underset{\delta \in K}{\sup}\,\,  \left\{	\frac{\Gamma(q_S,  w_{-S})}{n^{1/2} } \| P_{ \mathcal{R}}\delta\|, \right\} \\
		\leq    \displaystyle  \frac{ \lambda  \Gamma(q_S,  w_{-S})  \{h(t,n)  \}^{1/2}  }{2 n^{1/2} }   \bigg( \max\bigg\{    \tilde{C}_r \tilde{C} \frac{V^*}{n^{r-1}}  ,1 \bigg\}  \bigg)^{1/2},
	\end{array}
	\]
	where the second to last inequality follows from Lemma A.3 in \cite{ortelli2019prediction}, and the last from Lemmas  \ref{lem18}, \ref{lem20} and \ref{lem6}.

	However, by  Section 3.3  in  \cite{ortelli2019prediction}, it holds that  for some  $C_2>0$,
	\[
	\Gamma(q_S,  w_{-S})    \,\leq \,  C_2[\log  \{n/(s+1)\} ]^{1/2} \{ n(s+1) \}^{1/2} \left(  \frac{  s+1  }{n}\right)^{r-1/2}.
	%\log^{1/2}(s+1)\,\left(\frac{s+1}{n}\right)^{r-1}.
	\]
	Hence, 
	%   by Lemma \ref{lem6}
	\[
	\begin{array}{lll}
		\displaystyle 	  \underset{\delta \in K}{\sup}\,\,  \left\{\xi^{\top}  P_{ \mathcal{N}_{-S}^{\perp} }P_{ \mathcal{R}}\delta  +   \frac{\lambda}{2}\|D^{(r)}\theta^*\|_1 -  \frac{\lambda}{2}\|D^{(r)}(\theta^*+\delta)\|_1  \right\}  \\
		\leq   	\displaystyle  C \left[  (s+1)\log n  \log \{n/(s+1)\}   \right]^{1/2}\{h(t,n)\}^{1/2}.
	\end{array}
	\]
	Therefore, for large enough $c_{\epsilon}$,
	\begin{equation}
		\label{eqn:t3_second}
		\begin{array}{lll}
			\displaystyle \frac{4}{c_0 t^2}E\left[  \underset{\delta \in K}{\sup}\,\,  \left\{\xi^{\top}  P_{ \mathcal{N}_{-S}^{\perp} }P_{ \mathcal{R}}\delta  +   \frac{\lambda}{2}\|D^{(r)}\theta^*\|_1 -  \frac{\lambda}{2}\|D^{(r)}(\theta^*+\delta)\|_1  \right\} \bigg |  \Omega_3(b)^c\right] \mathrm{pr}\{\Omega_3(b)^c\}\\
			\displaystyle 	\leq \frac{4C  \{h(t,n) \}^{1/2}    \left[(s+1)\log n  \log \left\{   n/(s+1)  \right\}\right]^{1/2}}{c_0t^2}\\
			\displaystyle 	\leq \frac{\epsilon}{6},
		\end{array}
	\end{equation}
	which together with (\ref{eqn:t3_part1}) implies that $T_3 < \epsilon/3$.
	
	Next assume that 
	\[
	\lambda \,=\, 3c_{\epsilon}   \frac{n^{r-1} (s+1)  \log n  \log (s+1)   \log \frac{n}{s+1}  }{V^*}.  
	\]
	Then,   letting
	\[
	\tilde{\lambda }  = 3c_{\epsilon}  n^{r-1/2}\left(\frac{1}{s+1} \right)^{r-1/2} (\log n)^{1/2},
	\]
	we have that $\lambda \geq \tilde{\lambda}$ by the definition of $\lambda$,  which implies that
	\[
	\begin{array}{l}
		\displaystyle \underset{\delta \in K}{\sup}\left\{	\xi^{\top}  P_{ \mathcal{N}_{-S}^{\perp} }P_{ \mathcal{R}}\delta  +   \frac{\lambda}{2}\|D^{(r)}\theta^*\|_1 -  \frac{\lambda}{2}\|D^{(r)}(\theta^*+\delta)\|_1  \right\} \\ 
		\displaystyle 	\leq \underset{\delta \in K}{\sup}\left\{	\xi^{\top}  P_{ \mathcal{N}_{-S}^{\perp} }P_{ \mathcal{R}}\delta  +   \frac{\tilde{\lambda }}{2}\|D^{(r)}\theta^*\|_1 -  \frac{\tilde{\lambda }}{2}\|D^{(r)}(\theta^*+\delta)\|_1  \right\} \,+\,
		\\
		\displaystyle \frac{1}{2}	\underset{\delta \in K}{\sup} \left\{  (\lambda-\tilde{\lambda}) \|D^{(r)}\theta^*\|_1 -  (\lambda-\tilde{\lambda}) \|D^{(r)}(\theta^*+\delta)\|_1\right\}\\
		\displaystyle 	\leq \underset{\delta \in K}{\sup}\left\{	\xi^{\top}  P_{ \mathcal{N}_{-S}^{\perp} }P_{ \mathcal{R}}\delta  +   \frac{\tilde{\lambda }}{2}\|D^{(r)}\theta^*\|_1 -  \frac{\tilde{\lambda }}{2}\|D^{(r)}(\theta^*+\delta)\|_1  \right\} \,+\, \frac{\lambda}{2}\|D^{(r) }\theta^*\|_1\\
	\end{array}
	\]
	and we notice that
	\[
	\lambda\|D^{(r) }\theta^*\|_1 \,\,= \,  3c_{\epsilon}   (s+1)  \log n \, \log  (s+1)   \,  \log \frac{n}{s+1},
	\]
	and in this case we also obtain that $T_3 \leq \epsilon/3$ by proceeding as before. The proof follows.
	
\end{proof}

\section{Proof Theorem   \ref{thm:2dtv} }

\subsubsection{Controlling the Rademacher width }

%We  first need an  auxiliary lemma.

%\begin{lemma}
%	\label{lem21}
%	Let  $\delta \in R^n$  such that $\Delta^2(\delta)\leq  t^2$ for  $t>0$. Let  $\Pi$  be the orthogonal projection onto $\text{span}\{(1,\ldots,1)^{\top} \} \subset R^n$. Then
%	\[
%	\|  \Pi \delta \|_{\infty} \leq      \frac{t^2}{  n }   +   \frac{t}{n^{1/2}}.
%	\]
%\end{lemma}

%\begin{proof}

%	Notice that\[5
%	 \begin{array}{lll}
%	 \| \Pi \delta\|_{\infty}  & = &   \displaystyle  \frac{1}{n^{1/2}  }  \left\vert \sum_{i=1}^{n }      \frac{\delta_i}{  n^{1/2} }   \right\vert\\
%	  & \leq&  \displaystyle  \frac{1}{n}  \sum_{i=1}^{n}   \vert  \delta_i \vert 1_{  \{     \vert \delta_i \vert  >1 \} }   \,+\, \frac{1}{n^{1/2}  }  \left(\sum_{i=1}^{n}   \vert  \delta_i \vert^2 1_{  \{     \vert \delta_i \vert  \leq 1 \} } \right)^{1/2}
%	 \end{array}
%	\]
%	by H\"{o}lder and  Cauchy–Schwarz inequalities. 
%\end{proof}

\begin{proposition}
	\label{prop1}
	Let  $t>0$  and  
	\[
	K =    \left\{  \theta \,:\,  \mathrm{TV}(\theta)\leq   V\right\}.
	\]	
	Then 
	\[
	RW[  \{  \delta  \in   K -\theta^*\,:\,   \Delta^2(\delta ) \leq  t^2        \} ] \,\leq \,  C\left\{  V  m^{d-1}  c(d,n)    +      \frac{t^2}{n^{1/2}}  +   t       \right\},
	\]
	where  
	\[
	c(d,n)    \leq    \begin{cases}
		\tilde{C}  \log n  & \text{if} \,\, d=2,\\
		\tilde{C} (\log n)^{1/2} & \text{if} \,\, d>2,
	\end{cases}
	\]
	for some positive constant  $\tilde{C}$ that depends on $d$.
\end{proposition}

\begin{proof}
	Let  $\nabla$ be an incidence matrix  of $L_{d,n}$ ,  	$N$ the number of rows of $\nabla$,  and  $\Pi$ the orthogonal projection onto the span of $(1,\ldots,1)^{\top} \in \mathbb{R}^n$. Notice that  for  $\xi_1,\ldots,\xi_n$ independent Rademacher  variables we have that 
	\begin{equation}
		\label{eqn:e1}
		\begin{array}{lll}
			RW(  \{  \delta  \in   K -\theta^*\,:\,   \Delta^2(\delta ) \leq  t^2        \} )   & \leq&  	E\left\{   \underset{\delta \,:\,    \Delta^2(\delta)    \leq  t^2, \, \|\nabla\delta\|_1\leq 2V m^{d-1}   \,  }{\sup}\,\, \xi^{\top}   \nabla^{+} \nabla  \delta \right\}   \,+\,\\   
			& &\displaystyle 	E\left\{   \underset{\delta \,:\,    \Delta^2(\delta)    \leq  t^2 \,  }{\sup}\,\,  \xi^{\top} \Pi \delta \right\}\\
			& \leq&\displaystyle   2V m^{d-1} E\left\{    \|  \left(\nabla^{+}\right)^{\top}\xi\|_{\infty} \right\} +    \left(   \frac{b_1 t^2}{n} +   \frac{t}{n^{1/2} }  \right)\,E\left(  \left\vert  \sum_{i=1}^{n}  \xi_i \right\vert   \right)\\
			& \leq&C\left\{  V  m^{d-1}  \underset{j =1,\ldots,N}{\max} \|\nabla_{\cdot,j}^{+}  \|  (\log   n )^{1/2} +      \frac{t^2}{n^{1/2}}  +   t       \right\},
		\end{array}
	\end{equation}
	for some positive  constant $C$, where the  second inequality follows by H\"{o}lder's  inequality and  Lemma \ref{lem23}, and the  last  by the  Sub-Gaussian maximal   inequality.
	
	Next,  we  recall from Propositions  4 and 6 from \cite{hutter2016optimal}, that 
	\begin{equation}
		\label{eqn:e2}
		\underset{j =1,\ldots,N}{\max} \|\nabla_{\cdot,j}^{+}  \|    \leq    \begin{cases}
			\tilde{C}  (\log n)^{1/2}  & \text{if} \,\, d=2,\\
			\tilde{C}  & \text{if} \,\, d>2,
	\end{cases}\end{equation}
	for some positive constant  $\tilde{C}$ that depends on $d$.  Hence, the claim follows combining  (\ref{eqn:e1})--(\ref{eqn:e2}).
\end{proof}

\subsubsection{Proof  of Theorem \ref{thm:2dtv} }

Theorem \ref{thm:2dtv}  follows  immediately from Theorem \ref{thm:basic} and Proposition \ref{prop1},  by choosing 
\[
t \asymp    \begin{cases}
	(V  m^{d-1} \log^{2} n  )^{1/2}& \text{if}   \,\,\,   d=2,\\
	\{	V m^{d-1}\log n \}^{1/2} &\text{if }\,\,\,d>2.   
\end{cases}
\]

\subsection{Theorem   \ref{thm:lasso} }

\begin{proof}
	Let
	\[
	K_V\,:=\,\left\{\beta \in R^p: \|\beta\|_1 \leq V\right\}.
	\]
	Also,  let  $\{\xi_i\}_{i=1}^n$  be independent  Rademacher  random variables independent of  $\{(y_i)\}_{i=1}^n$. 	By Corollary \ref{cor:risk}, there exists a constant $C>0$ such that 
	%Lemma \ref{lem13} and  with the  same  argument  that was used to prove Corollary  \ref{cor:risk}, we obtain that
	\[
	\begin{array}{lll}
		\displaystyle \frac{1}{n}E\left\{\Delta^2( \hat{\theta} -\theta^*  ) \right\} %&\leq &\displaystyle  \frac{C}{n}\,E\left\{\underset{v \in K_V}{\sup}\,\,  \sum_{i=1}^{n} \xi_i x_i^{\top} ( v -\beta^* )   \right\}\\ 
		&\leq  &\displaystyle  \frac{CV }{n} E\left( \underset{v \in K_1}{\sup}\,\,  \sum_{i=1}^{n} \xi_i x_i^{\top} v   \right)\\ 
		&= &\displaystyle  \frac{CV }{n} E\left( \underset{v \in X K_1}{\sup}\,\,  \xi^{\top} v   \right),\\ 
	\end{array}
	\]
	where 
	\[
	K_1  = \left\{  v\,:\,  \|v\|_1\leq   1 \right\}.
	\]
	By the proof  of Theorem  2.4 in \cite{rigollet2015high},  there exists a constant  $C_2>0$  such that 
	\[
	E\left( \underset{v \in X K_1}{\sup}\,\,  \xi^{\top} v   \right)\leq       \frac{ C_2 (\log p )^{1/2}   \,  \underset{j\in [p]}{\max}    \|X_{\cdot,j} \| }{n}.
	\]
	The claim of the theorem then follows.
\end{proof}

\newpage

\section{Additional experiments}
\label{sec:additonal}

%$10\, \Delta_n^2(\theta^* -\hat{\theta})$
\begin{table}[h!]
	\centering
	\caption{\label{tab1}  Average $\Delta_n^2$ distance  times  10,   $10\cdot\Delta_n^2(\theta^* -\hat{\theta}) $,  averaging over 100 Monte carlo simulations for the different methods considered. Captions are described in the main paper.  }
	\medskip
	\setlength{\tabcolsep}{14pt}
	\begin{small}
		\begin{tabular}{ rrrrrrrr}
			\hline
			$n$ & Scenario            &$\tau$                     & PQTF1           & PQTF2          &QS                  &TF1               & TF2     \\  
			\hline	   
			10000 &1                     &    0.5                   &       0.023    &   0.08   &    0.21     &  0.016 & 0.4\\			
			5000 &1                     &    0.5                      &  0.046       &     0.12    &  0.23       &  0.034&0.65  \\	
			1000 &1                     &    0.5                     &     0.18  &    0.29      &  0.32        &    0.12& 0.94 \\			
			\hline				
			10000 &2                    &    0.5                     &    0.037 &   0.11 &   0.13 & 5.67  & 6.33\\		
			5000 &2                    &    0.5                      & 0.066 &  0.15    &  0.17 &    2.45      & 2.80     \\
			1000 &2                     &    0.5                      & 0.29    & 0.43   & 0.45      &     8.08   &  9.41      \\	
			
			\hline	
			10000 &3                    &    0.5                      & 0.015&  0.063    &   0.17   &   0.18  &    0.54  \\			
			5000 &3                   &    0.5                      &   0.029 &   0.092              &     0.18       &    0.13      &    0.65             \\				
			1000 &3                     &    0.5                      &     0.13 &  0.24      & 0.26     &      0.38         &     1.04            \\							
			\hline		         
			10000 &4                    &    0.5                      &  0.045   & 0.009 &     0.015   &      0.063      &     0.016      \\			 			
			5000 &4                   &    0.5                      &   0.075     &   0.019     &     0.027          &   0.10             &      0.031            \\						
			1000 &4                     &    0.5                      &0.30       &   0.082      &   0.098       &       0.28   &           0.31         \\						
			\hline	
			10000 &5                   &    0.5                      &    0.13      &  0.056    &   0.041  &   1.55        &      1.91   \\									
			5000 &5                   &    0.5                      &    0.24  &     0.099    &   0.085     &     3.24          &   3.8   \\	
			1000 &5                   &    0.5                      & 1.91    &      0.35    &  0.35     &        5.38       &    6.00   \\					
			\hline			
			%10000 &6                     &    0.9                      &                &        \textbf{}          &             &         *        &            *            \\	               
			10000 &6                   &    0.9                     &      0.18     &0.070  &    0.075         &         *       &                  *      \\	 		  
			5000&6                   &   0.9                      &    0.29        &  0.13  &   0.14          &             *    &                 *       \\			         
			1000&6                   &   0.9                      &      1.19      &  0.39  &  0.40           &             *    &                 *       \\					
			10000 &6                   &   0.1                      &  0.16    & 0.065 &  0.070        &          *       &      *                  \\								
			5000 &6                   &    0.1                     & 0.31   &0.13&   0.14&            *     &         *               \\				
			1000 &6                     &    0.1                      & 1.27    &  0.46    & 0.47    &           *      &              *          \\            
			\hline  
		\end{tabular}
	\end{small}
\end{table}

\clearpage

\bibliographystyle{plainnat}
\bibliography{references}

\begin{thebibliography}{46}
\providecommand{\natexlab}[1]{#1}
\providecommand{\url}[1]{\texttt{#1}}
\expandafter\ifx\csname urlstyle\endcsname\relax
  \providecommand{\doi}[1]{doi: #1}\else
  \providecommand{\doi}{doi: \begingroup \urlstyle{rm}\Url}\fi

\bibitem[Belloni and Chernozhukov(2011)]{belloni2011}
Alexandre Belloni and Victor Chernozhukov.
\newblock $\ell_1$-penalized quantile regression in high-dimensional sparse
  models.
\newblock \emph{The Annals of Statistics}, 39\penalty0 (1):\penalty0 82--130,
  2011.

\bibitem[Brantley et~al.(2020)Brantley, Guinness, and
  Chi]{brantley2019baseline}
Halley~L Brantley, Joseph Guinness, and Eric~C Chi.
\newblock Baseline drift estimation for air quality data using quantile trend
  filtering.
\newblock \emph{Annals of Applied Statistics}, 14\penalty0 (2):\penalty0
  585--604, 2020.

\bibitem[Brown et~al.(2008)Brown, Cai, and Zhou]{brown2008robust}
Lawrence~D Brown, T~Tony Cai, and Harrison~H Zhou.
\newblock Robust nonparametric estimation via wavelet median regression.
\newblock \emph{The Annals of Statistics}, 36\penalty0 (5):\penalty0
  2055--2084, 2008.

\bibitem[Chatterjee and Goswami(2019{\natexlab{a}})]{chatterjee2019adaptive}
Sabyasachi Chatterjee and Subhajit Goswami.
\newblock Adaptive estimation of multivariate piecewise polynomials and bounded
  variation functions by optimal decision trees.
\newblock \emph{To appear in the Annals of Statistics}, 2019{\natexlab{a}}.

\bibitem[Chatterjee and Goswami(2019{\natexlab{b}})]{chatterjee2019new}
Sabyasachi Chatterjee and Subhajit Goswami.
\newblock New risk bounds for 2d total variation denoising.
\newblock \emph{To appear in IEEE Transctions of Information Theory},
  2019{\natexlab{b}}.

\bibitem[Chatterjee et~al.(2015)Chatterjee, Guntuboyina, and
  Sen]{chatterjee2015risk}
Sabyasachi Chatterjee, Adityanand Guntuboyina, and Bodhisattva Sen.
\newblock On risk bounds in isotonic and other shape restricted regression
  problems.
\newblock \emph{The Annals of Statistics}, 43\penalty0 (4):\penalty0
  1774--1800, 2015.

\bibitem[Chatterjee(2013)]{chatterjee2013assumptionless}
Sourav Chatterjee.
\newblock Assumptionless consistency of the lasso.
\newblock \emph{arXiv preprint arXiv:1303.5817}, 2013.

\bibitem[Cox(1983)]{cox1983asymptotics}
Dennis~D Cox.
\newblock Asymptotics for m-type smoothing splines.
\newblock \emph{The Annals of Statistics}, pages 530--551, 1983.

\bibitem[Donoho(1997)]{donoho1997cart}
David~L Donoho.
\newblock Cart and best-ortho-basis: a connection.
\newblock \emph{The Annals of statistics}, 25\penalty0 (5):\penalty0
  1870--1911, 1997.

\bibitem[Donoho and Johnstone(1994)]{donoho1994ideal}
David~L Donoho and Jain~M Johnstone.
\newblock Ideal spatial adaptation by wavelet shrinkage.
\newblock \emph{Biometrika}, 81\penalty0 (3):\penalty0 425--455, 1994.

\bibitem[Eubank(1988)]{eubank1988spline}
Randall~L Eubank.
\newblock \emph{Spline smoothing and nonparametric regression}, volume~90.
\newblock M. Dekker New York, 1988.

\bibitem[Fan et~al.(2014)Fan, Fan, and Barut]{fan2014adaptive}
Jianqing Fan, Yingying Fan, and Emre Barut.
\newblock Adaptive robust variable selection.
\newblock \emph{The Annals of statistics}, 42\penalty0 (1):\penalty0 324, 2014.

\bibitem[Guntuboyina and Sen(2015)]{guntuboyina2015global}
Adityanand Guntuboyina and Bodhisattva Sen.
\newblock Global risk bounds and adaptation in univariate convex regression.
\newblock \emph{Probability Theory and Related Fields}, 163\penalty0
  (1-2):\penalty0 379--411, 2015.

\bibitem[Guntuboyina et~al.(2020)Guntuboyina, Lieu, Chatterjee, and
  Sen]{guntuboyina2020adaptive}
Adityanand Guntuboyina, Donovan Lieu, Sabyasachi Chatterjee, and Bodhisattva
  Sen.
\newblock Adaptive risk bounds in univariate total variation denoising and
  trend filtering.
\newblock \emph{The Annals of Statistics}, 48\penalty0 (1):\penalty0 205--229,
  2020.

\bibitem[He and Shi(1994)]{he1994convergence}
Xuming He and Peide Shi.
\newblock Convergence rate of b-spline estimators of nonparametric conditional
  quantile functions.
\newblock \emph{Journaltitle of Nonparametric Statistics}, 3\penalty0
  (3-4):\penalty0 299--308, 1994.

\bibitem[Hjort and Pollard(2011)]{hjort2011asymptotics}
Nils~Lid Hjort and David Pollard.
\newblock Asymptotics for minimisers of convex processes.
\newblock \emph{arXiv preprint arXiv:1107.3806}, 2011.

\bibitem[Hochbaum and Lu(2017)]{hochbaum2017faster}
Dorit~S Hochbaum and Cheng Lu.
\newblock A faster algorithm solving a generalization of isotonic median
  regression and a class of fused lasso problems.
\newblock \emph{SIAM Journal on Optimization}, 27\penalty0 (4):\penalty0
  2563--2596, 2017.

\bibitem[Huber(1964)]{huber1992robust}
Peter~J Huber.
\newblock Robust estimation of a location parameter.
\newblock \emph{The Annals of Statistics.}, page 73–101, 1964.

\bibitem[Hutter and Rigollet(2016)]{hutter2016optimal}
Jan-Christian Hutter and Philippe Rigollet.
\newblock Optimal rates for total variation denoising.
\newblock \emph{Annual Conference on Learning Theory}, 29:\penalty0 1115--1146,
  2016.

\bibitem[Johnson(2013)]{johnson2013dynamic}
Nicholas Johnson.
\newblock A dynamic programming algorithm for the fused lasso and
  $l_0$-segmentation.
\newblock \emph{Journal of Computational and Graphical Statistics}, 22\penalty0
  (2):\penalty0 246--260, 2013.

\bibitem[Kim et~al.(2009)Kim, Koh, Boyd, and Gorinevsky]{kim2009ell_1}
Seung-Jean Kim, Kwangmoo Koh, Stephen Boyd, and Dimitry Gorinevsky.
\newblock $\ell_1$ trend filtering.
\newblock \emph{SIAM Review}, 51\penalty0 (2):\penalty0 339--360, 2009.

\bibitem[Knight and Fu(2000)]{knight2000asymptotics}
Keith Knight and Wenjiang Fu.
\newblock Asymptotics for lasso-type estimators.
\newblock \emph{The Annals of Statistics}, pages 1356--1378, 2000.

\bibitem[Koenker and Bassett~Jr(1978)]{koenker1978regression}
Roger Koenker and Gilbert Bassett~Jr.
\newblock Regression quantiles.
\newblock \emph{Econometrica: Journal of the Econometric Society}, pages
  33--50, 1978.

\bibitem[Koenker et~al.(1994)Koenker, Ng, and Portnoy]{koenker1994quantile}
Roger Koenker, Pin Ng, and Stephen Portnoy.
\newblock Quantile smoothing splines.
\newblock \emph{Biometrika}, 81\penalty0 (4):\penalty0 673--680, 1994.

\bibitem[Ledoux and Talagrand(2013)]{ledoux2013probability}
Michel Ledoux and Michel Talagrand.
\newblock \emph{Probability in Banach Spaces: isoperimetry and processes}.
\newblock Springer Science \& Business Media, 2013.

\bibitem[Li and Zhu(2007)]{li2007analysis}
Youjuan Li and Ji~Zhu.
\newblock Analysis of array cgh data for cancer studies using fused quantile
  regression.
\newblock \emph{Bioinformatics}, 23\penalty0 (18):\penalty0 2470--2476, 2007.

\bibitem[Mammen and van~de Geer(1997)]{mammen1997locally}
Enno Mammen and Sara van~de Geer.
\newblock Locally apadtive regression splines.
\newblock \emph{The Annals of Statistics}, 25\penalty0 (1):\penalty0 387--413,
  1997.

\bibitem[Mangasarian and Schumaker(1971)]{mangasarian1971discrete}
Olvi~L Mangasarian and Larry~L Schumaker.
\newblock Discrete splines via mathematical programming.
\newblock \emph{SIAM Journal on Control}, 9\penalty0 (2):\penalty0 174--183,
  1971.

\bibitem[Nussbaum(1985)]{nussbaum1985spline}
Michael Nussbaum.
\newblock Spline smoothing in regression models and asymptotic efficiency in $
  \ell_2$.
\newblock \emph{The Annals of Statistics}, 13\penalty0 (3):\penalty0 984--997,
  1985.

\bibitem[Ortelli and van~de Geer(2019{\natexlab{a}})]{ortelli2019prediction}
Francesco Ortelli and Sara van~de Geer.
\newblock Prediction bounds for (higher order) total variation regularized
  least squares.
\newblock \emph{arXiv preprint arXiv:1904.10871}, 2019{\natexlab{a}}.

\bibitem[Ortelli and van~de Geer(2019{\natexlab{b}})]{ortelli2019synthesis}
Francesco Ortelli and Sara van~de Geer.
\newblock Synthesis and analysis in total variation regularization.
\newblock \emph{arXiv preprint arXiv:1901.06418}, 2019{\natexlab{b}}.

\bibitem[Padilla et~al.(2018)Padilla, Sharpnack, and Scott]{padilla2017dfs}
Oscar Hernan~Madrid Padilla, James Sharpnack, and James~G Scott.
\newblock The {D}{F}{S} fused lasso: Linear-time denoising over general graphs.
\newblock \emph{The Journal of Machine Learning Research}, 18\penalty0
  (1):\penalty0 6410--6445, 2018.

\bibitem[Padilla et~al.(2020)Padilla, Sharpnack, Chen, and
  Witten]{madrid2020adaptive}
Oscar Hernan~Madrid Padilla, James Sharpnack, Yanzhen Chen, and Daniela~M
  Witten.
\newblock Adaptive nonparametric regression with the k-nearest neighbour fused
  lasso.
\newblock \emph{Biometrika}, 107\penalty0 (2):\penalty0 293--310, 2020.

\bibitem[Rigollet and H{\"u}tter(2015)]{rigollet2015high}
Phillippe Rigollet and Jan-Christian H{\"u}tter.
\newblock High dimensional statistics.
\newblock \emph{Lecture notes for course 18S997}, 2015.

\bibitem[Rudin et~al.(1992)Rudin, Osher, and Fatemii]{rudin1992nonlinear}
Leonid Rudin, Stanley Osher, and Emad Fatemii.
\newblock Nonlinear total variation based noise removal algorithms.
\newblock \emph{Physica {D}: Nonlinear Phenomena}, 60\penalty0 (1):\penalty0
  259--268, 1992.

\bibitem[Sadhanala et~al.(2016)Sadhanala, Wang, and
  Tibshirani]{sadhanala2016total}
Veeranjaneyulu Sadhanala, Yu-Xiang Wang, and Ryan~J. Tibshirani.
\newblock Total variation classes beyond 1d: {Minimax} rates, and the
  limitations of linear smoothers.
\newblock \emph{In Advances in Neural Information Processing Systems}, pages
  3513--3521, 2016.

\bibitem[Steidl et~al.(2006)Steidl, Didas, and Neumann]{steidl2006splines}
Gabriele Steidl, Stephan Didas, and Julia Neumann.
\newblock Splines in higher order tv regularization.
\newblock \emph{International journal of computer vision}, 70\penalty0
  (3):\penalty0 241--255, 2006.

\bibitem[Sun et~al.(2019)Sun, Zhou, and Fan]{sun2019adaptive}
Qiang Sun, Wen-Xin Zhou, and Jianqing Fan.
\newblock Adaptive huber regression.
\newblock \emph{Journal of the American Statistical Association}, pages 1--24,
  2019.

\bibitem[Tibshirani(1996)]{tibshirani1996regression}
Robert Tibshirani.
\newblock Regression shrinkage and selection via the lasso.
\newblock \emph{Journal of the Royal Statistical Society: Series B
  (Methodological)}, 58\penalty0 (1):\penalty0 267--288, 1996.

\bibitem[Tibshirani(2014)]{tibshirani2014adaptive}
Ryan~J. Tibshirani.
\newblock Adaptive piecewise polynomial estimation via trend filtering.
\newblock \emph{The Annals of Statistics}, 42\penalty0 (1):\penalty0 285--323,
  2014.

\bibitem[Utreras(1981)]{utreras1981computing}
Florencio~I Utreras.
\newblock On computing robust splines and applications.
\newblock \emph{SIAM Journal on Scientific and Statistical Computing},
  2\penalty0 (2):\penalty0 153--163, 1981.

\bibitem[Van~de Geer(1990)]{van1990estimating}
Sara Van~de Geer.
\newblock Estimating a regression function.
\newblock \emph{The Annals of Statistics}, pages 907--924, 1990.

\bibitem[Van Der~Vaart and Wellner(1996)]{van1996weak}
AW~Van Der~Vaart and JA~Wellner.
\newblock Weak convergence and empirical processes: With applications to
  statistics springer series in statistics.
\newblock \emph{Springer}, 58:\penalty0 59, 1996.

\bibitem[Wainwright(2019)]{wainwright2019high}
Martin~J Wainwright.
\newblock \emph{High-dimensional statistics: A non-asymptotic viewpoint},
  volume~48.
\newblock Cambridge University Press, 2019.

\bibitem[Wang et~al.(2014)Wang, Smola, and Tibshirani]{wang2014falling}
Yu-Xiang Wang, Alex Smola, and Ryan Tibshirani.
\newblock The falling factorial basis and its statistical applications.
\newblock In \emph{International Conference on Machine Learning}, pages
  730--738, 2014.

\bibitem[Wang et~al.(2016)Wang, Sharpnack, Smola, and
  Tibshirani]{wang2016trend}
Yu-Xiang Wang, James Sharpnack, Alex Smola, and Ryan~J Tibshirani.
\newblock Trend filtering on graphs.
\newblock \emph{Journal of Machine Learning Research}, 17\penalty0
  (105):\penalty0 1--41, 2016.

\end{thebibliography}

%	\bibliographystyle{plain}
%\bibliographystyle{plainnat}
%\bibliography{quantile_trend_filtering}	
%\bibliography{references}	
	
\end{document}